\documentclass[11pt]{amsart}

\usepackage[a4paper,hmargin=3cm,vmargin=3cm]{geometry}
\usepackage{amsfonts,amssymb,amscd,amstext,hyperref}
\usepackage{graphicx}
\usepackage{color}
\usepackage[dvips]{epsfig}
\usepackage[latin1]{inputenc}

\usepackage{fancyhdr}
\usepackage{times}

\usepackage{enumerate}
\usepackage{titlesec}
\usepackage{mathrsfs}
\usepackage{stmaryrd}

\def\R{\mathbb{R}}
\def\B{\mathbb{B}}
\def\N{\mathbb{N}}
\def\Z{\mathbb{Z}}
\def\C{\mathbb{C}}
\def\Q{\mathbb{Q}}

\def\M{\mathbb{M}}
\def\D{\mathbb{D}}

\newcommand{\ben}{\begin{enumerate}}
\newcommand{\bit}{\begin{itemize}}
\newcommand{\een}{\end{enumerate}}
\newcommand{\eit}{\end{itemize}}

\newcommand{\ed}{\end{document}}

\def\cA{\mathcal{A}}

\def\cP{\mathcal{P}}

\def\cC{\mathcal{C}}

\def\cR{\mathcal{R}}
\def\cB{\mathcal{B}}
\def\cW{\mathcal{W}}

\def\cN{\mathcal{N}}

\def\cL{\mathcal{L}}

\def\cF{\mathcal{F}}
\def\cN{\mathcal{N}}

\def\cO{\mathcal{O}}

\let\8=\infty \let\0=\emptyset  
\let\hat=\widehat

\let\tilde=\widetilde

\let\landa=\lambda
\let\alfa=\alpha

\let\parc=\partial

\def\ep{\varepsilon}

\def\landa{\lambda}

\def\flecha{\rightarrow}
\def\esiz{\langle}
\def\esde{\rangle}

\def\S{\Sigma}
\def\mh{\mathfrak{h}}

\def\cte.{\mathop{\rm cte.}\nolimits}

\def\cosh{\mathop{\rm cosh }\nolimits}

\def\N{\mathbb{N}}

\def\B{\mathbb{B}}
\def\Q{\mathbb{Q}}
\def\R{\mathbb{R}}
\def\Z{\mathbb{Z}}
\def\C{\mathbb{C}}
\def\A{\mathbb{A}}
\def\D{\mathbb{D}}

\def\H{\mathbb{H}}
\def\S{\mathbb{S}}

\newfont{\bb}{msbm10 at 12pt}

\titleformat{\section}
{\filcenter\bfseries\large} {\thesection{.}}{0.2cm}{}
\titleformat{\subsection}[runin]
{\bfseries} {\thesubsection{.}}{0.15cm}{}[.]
\titleformat{\subsubsection}[runin]
{\em}{\thesubsubsection{.}}{0.15cm}{}[.]

\usepackage[up,bf]{caption}

\newtheorem{theorem}{Theorem}[section]
\newtheorem{lemma}[theorem]{Lemma}
\newtheorem{proposition}[theorem]{Proposition}

\newtheorem{remark}[theorem]{Remark}
\newtheorem{corollary}[theorem]{Corollary}
\newtheorem{definition}[theorem]{Definition}

\theoremstyle{definition}

\numberwithin{equation}{section}
\numberwithin{figure}{section}

\usepackage{color}

\begin{document}
\fancyhead[LO]{Free boundary CMC annuli}
\fancyhead[RE]{Alberto Cerezo, Isabel Fernández, Pablo Mira}
\fancyhead[RO,LE]{\thepage}

\thispagestyle{empty}

\begin{center}
{\bf \LARGE Free boundary CMC annuli in spherical \\[0.3cm] and hyperbolic balls}

\vspace*{5mm}

\hspace{0.2cm} {\Large Alberto Cerezo, Isabel Fernández and Pablo Mira}
\end{center}

\vspace{0.5cm}


\footnote[0]{
\mbox{} \hspace{0.25cm} \emph{Keywords}: constant mean curvature, free boundary, capillary surfaces, spherical curvature lines, nodoid, critical catenoid}

\vspace*{7mm}

\begin{quote}
{\small
\noindent {\bf Abstract}\hspace*{0.1cm}
We construct, for any $H\in \R$, infinitely many free boundary annuli in geodesic balls of $\S^3$ with constant mean curvature $H$ and a discrete, non-rotational, symmetry group. Some of these free boundary CMC annuli are actually embedded if $H\geq 1/\sqrt{3}$. We also construct embedded, non-rotational, free boundary CMC annuli in geodesic balls of $\H^3$, for all values $H>1$ of the mean curvature $H$.

\vspace*{0.1cm}

}
\end{quote}

\section{Introduction and statement of the results}\label{sec:intro}

\subsection{Historical introduction}

For a long time, it was believed that the only  \emph{closed}, i.e., compact without boundary, surfaces with constant mean curvature (CMC) in Euclidean space $\R^3$ were the round, totally umbilic spheres. Hopf \cite{Ho0} proved in 1951 the validity of this statement for the particular case of genus zero. However, the conjecture was unexpectedly solved in the negative by Wente \cite{W0} in 1986, by constructing CMC tori in $\R^3$ with self-intersections and a discrete symmetry group. Subsequently, Abresch \cite{Ab} and Walter \cite{Wa1,Wa2} gave a more explicit construction, by prescribing that the CMC tori were foliated by \emph{planar curvature lines}. 

A fundamental achievement of CMC theory springing from these works was the classification of all CMC tori in the space forms $\R^3$, $\S^3$ and $\H^3$ by Pinkall-Sterling \cite{PS} (in $\R^3$), Hitchin \cite{Hi} (for minimal tori in $\S^3$) and then by Bobenko \cite{Bob} (for CMC tori in $\S^3$ and $\H^3$), using algebro-geometric methods from integrable systems. In these theorems, CMC tori were described in connection with \emph{finite type} solutions of the sinh-Gordon equation. Roughly, for any natural $N\geq 1$, the space of type $N$ sinh-Gordon solutions is finite dimensional, all CMC tori are of finite type, and they can be detected within their associated finite dimensional space by explicit closing conditions. 

This classification did not detect the possible embeddedness of CMC tori in $\S^3$ (in $\R^3$ and $\H^3$ there are no closed embedded CMC surfaces, by Alexandrov's theorem). The fundamental achievement in this direction was obtained more recently by Brendle \cite{B}, who proved that the Clifford torus is the only embedded minimal torus in $\S^3$, thereby solving in the affirmative a long-standing conjecture by Lawson. The idea in \cite{B} was then adapted by Andrews and Li \cite{AL} to show that any embedded CMC torus in $\S^3$ is rotational; this solved a conjecture by Pinkall-Sterling.

There is a natural boundary version of the classification problem for CMC tori discussed above: to classify all free boundary CMC annuli in geodesic balls of $\M^3(\ep)=\R^3, \S^3$ or $\H^3$. Here, we say that a compact CMC surface $\Sigma$ has free boundary in a geodesic ball ${\bf B}\subset \M^3(\ep)$ if $\Sigma\subset {\bf B}$ and $\Sigma$ intersects $\parc {\bf B}$ orthogonally along $\parc \Sigma$. This problem was already considered by Nitsche \cite{Nit} in 1985. The solutions are critical points associated to a natural variational problem for the area funcional, see \cite{Nit,RS}.

Nitsche proved in \cite{Nit} that any free boundary CMC disk in ${\bf B}\subset \R^3$ is totally umbilic. Ros-Souam \cite{RS} extended this result to $\S^3$ and $\H^3$. In \cite{Nit}, Nitsche claimed without proof that any free boundary CMC annulus in a ball ${\bf B}\subset \R^3$ should be rotational. This claim was proved incorrect by Wente in 1995 \cite{W2}, through the construction of immersed free boundary CMC annuli with very large mean curvature in the unit ball of $\R^3$. However, two questions remained: the existence of non-rotational free boundary minimal annuli and of embedded free boundary CMC annuli in the unit ball.

The first question was recently answered by Fernández, Hauswirth and Mira in \cite{FHM}, where they constructed immersed free boundary minimal annuli in the unit ball. For that, they used the Weierstrass representation, and a study of minimal surfaces in $\R^3$ with spherical curvature lines; later on, Kapouleas-McGrath \cite{KM} presented an alternative, more analytical construction via doubling. The second question has also been recently answered by the authors in \cite{CFM}, by constructing embedded non-rotational free boundary CMC annuli in the unit ball of $\R^3$, through a study of an overdetermined system for the sinh-Gordon equation. This answered an open problem posed by Wente in 1995 \cite{W2}.

These results still leave unsolved the \emph{critical catenoid conjecture}, i.e., that any embedded free boundary minimal annulus in the unit ball of $\R^3$ should be the critical catenoid. See e.g. \cite{F,FL,L}. This conjecture is conceived as the free boundary version of the aforementioned Lawson conjecture for minimal tori in $\S^3$.

The problem of whether the only embedded free boundary minimal annulus in a geodesic ball of $\S^3$ is a (spherical) critical catenoid has been studied in two interesting recent works by Lima-Menezes \cite{LM2} and Medvedev \cite{Me}. In \cite{LM2}, the uniqueness of this critical catenoid is obtained among immersed annuli by imposing that its coordinate functions are first eigenfunctions of a suitable Steklov problem. In \cite{Me} it is shown that the Morse index of the critical catenoid is $4$, that its spectral index is $1$, and that any free boundary minimal annulus with spectral index $1$ is a (spherical) critical catenoid. 

\subsection{Main results}

Our aim in this paper is to show that there exist many non-rotational free boundary CMC annuli in geodesic balls of $\S^3$ and $\H^3$, which in many cases are actually embedded. In particular, this shows that the class of minimal annuli in $\S^3$ considered in \cite{LM2,Me} is non-trivial, and that the classification of all free boundary CMC annuli in geodesic balls of space forms is a very rich geometric problem, both in the immersed and embedded cases. Our main results are described in Theorems \ref{th:main1} and \ref{th:main2} below.  They indicate that:

\begin{enumerate}
    \item
In $\S^3$ there exist immersed, non-rotational free boundary $H$-annuli in geodesic balls ${\bf B}\subset \S^3$, for any $H\in \R$. In particular, there exist free boundary minimal annuli in $\S^3$ with a finite symmetry group. This answers a problem in \cite{Me}.
\item For $H\geq 1/\sqrt{3}$, some of these free boundary CMC annuli in $\S^3$ are embedded.
\item 
In $\H^3$, there exist embedded non-rotational free boundary $H$-annuli in geodesic balls ${\bf B}\subset \H^3$, for any $H>1$.
\item 
All these $H$-annuli come in $1$-parameter families, and are foliated by spherical curvature lines. They can be seen as free boundary bifurcations from finite covers of adequate rotational examples (nodoids in $\H^3$, catenoids or nodoids in $\S^3$).
\end{enumerate}

We will also show that there exist embedded, non-rotational, \emph{capillary} minimal annuli in geodesic balls of $\S^3$.

We should note that the analytic results by Kilian and Smith in \cite{KS} prove that any free boundary CMC annulus in a geodesic ball of $\R^3,\S^3$ or $\H^3$ is associated to a finite type solution of the sinh-Gordon equation. The spherical curvature lines condition of our examples is very natural in this context, since they correspond to solutions of type $N=2$, see \cite{PS,PS2,W}. The study of CMC surfaces in $\R^3$ with spherical curvature lines dates back to classical works by differential geometers of the 19th century, like Enneper, Dobriner or Darboux. See \cite{W,Wa1,Wa2} for more modern approaches, and \cite{BHS,CPS} for studies on isothermic surfaces with spherical curvature lines.

In the next theorems, we let $H\geq 0$ and $\ep\in \{-1,1\}$ so that $H^2+\ep>0$, and denote $\M^3(1)=\S^3\subset \R^4$ and $\M^3(-1)=\H^3\subset \mathbb{L}^4$.

\begin{theorem}\label{th:main1}
There exists an open interval $\mathcal{J}=\mathcal{J}(H,\ep)$ contained in $(0,1)$ such that, for any irreducible $q=m/n\in \mathcal{J}\cap \Q$, there exists a real analytic $1$-parameter family $\cF_q:=\{\A_q (\eta): \eta\in [0,\epsilon_0(q))\}$ of compact annuli in $\M^3(\ep)$ with the following properties:
\begin{enumerate}
\item
Each annulus $\A_q(\eta)$ has constant mean curvature $H$, and has free boundary in a geodesic ball ${\bf B}={\bf B}(q,\eta)\subset \M^3(\ep)$ centered at ${\bf e}_4=(0,0,0,1)\in \M^3(\ep)$.
\item 
$\A_q(\eta)$ is symmetric with respect to the totally geodesic surface ${\bf S}:= \M^3(\ep)\cap \{x_3 = 0\}$.
\item 
The closed geodesic ${\bf S}\cap \A_q(\eta)$ of $\A_q(\eta)$ has rotation index $-m$ in ${\bf S}$.
\item 
$\A_q(0)$ is a finite $m$-cover of a compact embedded piece of a Delaunay surface in $\M^3(\ep)$.
\item 
If $\eta>0$, then $\A_q(\eta)$ is not rotational, and its symmetry group is generated by the symmetry with respect to ${\bf S}$, and by the symmetries with respect to $n>1$ equiangular totally geodesic surfaces of $\M^3(\ep)$ orthogonal to ${\bf S}$. That is, the symmetry group of $\A_q(\eta)$ is prismatic of order $4n$.
\item 
Each annulus $\A_q(\eta)$ is foliated by spherical curvature lines, so that both of its boundary components are elements of this foliation.
\end{enumerate}
\end{theorem}

\begin{theorem}\label{th:main2}
Assume in Theorem \ref{th:main1} that $\ep =-1$ or that $\ep=1$ and $H\geq 1/\sqrt{3}$. Then, there exist elements in $\mathcal{J}\cap \Q$ of the form $q=1/n$, and for any such $q$, the free boundary $H$-annuli $\A_q(\eta)$ are embedded, for $\eta$ sufficiently small. When $\ep=-1$, it actually holds $1/n\in \mathcal{J}$ for any $n\geq 2$.
\end{theorem}

\subsection{Organization of the paper and sketch of the proof}

The strategy to prove Theorems \ref{th:main1} and \ref{th:main2} is inspired by our previous works \cite{CFM,FHM} on free boundary CMC annuli in $\R^3$. There are however several sources of complication in the process when we consider $\S^3$ or $\H^3$ as our ambient space, and their resolution requires new ideas. For instance, we cannot use the Weierstrass representation of minimal surfaces as in \cite{FHM}. And, in contrast with the $\R^3$ case in \cite{CFM}, in our spherical or hyperbolic setting we do not have an explicit expresion for the \emph{period map} that controls the periodicity of the spherical curvature lines of our examples. Moreover, we cannot use CMC surfaces with \emph{planar} curvature lines as a limit in order to control the centers of the spherical curvature lines, as we did in \cite{CFM} for the Euclidean case, since such surfaces do not exist in $\S^3$ or $\H^3$.

We next explain the basic steps of the proof.

In Section \ref{sec:prelim} we present some preliminaries on CMC surfaces in space forms foliated by \emph{spherical} curvature lines, i.e. each curvature line of this foliation lies in some $2$-dimensional totally umbilic surface of the ambient space. 

In Section \ref{sec:solsg} we recall our construction in \cite{CFM} of some special solutions to an overdetermined problem for the sinh-Gordon equation. When this equation is viewed as the Gauss equation of a CMC surface in $\M^3(\ep)=\S^3$ or $\H^3$ parametrized by curvature lines (with $H>1$ in $\H^3$), these solutions yield complete CMC surfaces in $\M^3(\ep)$ foliated by spherical curvature lines. Along such curvature lines, the surface intersects the corresponding totally umbilic surface at a constant angle, by Joachimstal's theorem. In this way, we end up with a family of conformal CMC immersions $\psi(u,v):\R^2\flecha \M^3(\ep)$ paramerized by curvature lines, which depends on three parameters $(a,b,c)$, and so that the curves $v\mapsto \psi(u,v)$ are spherical.

In Section \ref{sec:romega} we study the geometry of the immersions $\psi(u,v)$. We prove that they are symmetric with respect to a \emph{horizontal} totally geodesic  surface ${\bf S}\subset \M^3(\ep)$ and with respect to a number of \emph{vertical} totally geodesic surfaces $\Omega_k\subset \M^3(\ep)$ that, in adequate conditions, intersect along a vertical geodesic of $\M^3(\ep)$ that contains all the \emph{centers} of the totally umbilic surfaces where the spherical curvature lines $v\mapsto \psi(u,v)$ of the immersion lie.

We also define a real analytic \emph{period map} $\Theta(a,b,c)$ with the property that, when $\Theta(a,b,c)=m/n\in \Q$, the spherical curvature lines $v\mapsto \psi(u,v)$ are periodic curves with rotation index $m$, while $n$ gives the number of vertical symmetry surfaces $\Omega_k$. In particular, when $\Theta(a,b,c)=m/n$, the restriction of $\psi$ to $[-u_0,u_0]\times \R$ covers a CMC annulus $\Sigma_0=\Sigma_0(a,b,c,u_0)$ in $\M^3(\ep)$ that intersects with a constant angle two (isometric) totally umbilic surfaces of $\M^3(\ep)$.

Thus, our objetive will be to show that along some real analytic curves in the parameter space $(a,b,c,u_0)$, the resulting CMC annuli $\Sigma_0$ are free boundary in some geodesic ball ${\bf B}\subset \M^3(\ep)$. That is, we will need to control simultaneously that both boundary curves of $\Sigma_0$ lie in the same totally umbilic surface of $\M^3(\ep)$, that this surface bounds a compact geodesic ball ${\bf B}\subset \M^3(\ep)$, that $\Sigma_0$ is totally contained in ${\bf B}$, and that the intersection angle along $\parc \Sigma_0$ with $\parc {\bf B}$ is $\pi/2$. All of this while controlling the possible embeddedness of $\Sigma_0$.

To achieve this, the main idea will be to bifurcate from some rotational free boundary CMC surface, so that the spherical curvature lines conditions is preserved. For this, we will need a quite detailed description, that will be carried out in Section \ref{sec:rotational}, on the existence of \emph{critical} (spherical) catenoids in $\S^3$, and \emph{critical nodoids} in geodesic balls of $\S^3$ and $\H^3$. Section \ref{sec:rotational} is essentially independent from the rest of the paper, and the proofs there are postergated to an appendix.

In Section \ref{sec:double} we will show that, when $a=1$, the immersion $\psi(u,v)$ parametrizes a compact piece of a rotational CMC surface. More specifically, of either a nodoid (in $\H^3$), or a nodoid, a (spherical) catenoid or a flat torus in $\S^3$. The parameter $c$ will control the necksize of this example.  The parameter $b$ is needed to account for the fact that, in the rotational case, each curve $v\mapsto \psi(u,v)$ is a circle, and so there is a priori an infinite number of totally umbilic surfaces in $\M^3(\ep)$ that contain it. The parameter $b$ determines a choice for such umbilic surfaces.

In Section \ref{sec:perrot} we will give an explicit expression for the period map $\Theta(1,b,c)$ in the case $a=1$ that allows for a good control on the periodicity of the parametrized curvature lines $v\mapsto \psi(u,v)$, which in this $a=1$ case are merely circles.

In Section \ref{detectcritical} we will restrict to a certain \emph{free boundary region} of our parameter space, and control there, also for the case 
$a=1$, the situation in which $\Sigma_0(1,b,c,u_0)$ covers a critical catenoid or nodoid in $\M^3(\ep)$. We will show that these rotational free boundary surfaces must appear along certain curves of the parameter domain in that $a=1$ case.

In Section \ref{sec:proofmain} we prove Theorems \ref{th:main1} and \ref{th:main2}. For that we use our study of the rotational $a=1$ case in the previous sections in what regards the free boundary annulus structure of the examples, and induce it to the non-rotational $a>1$ case. The embeddedness is obtained for values of the period $\Theta(a,b,c)$ of the  form $-1/n$, with $a>1$ close to $1$.

Finally, in Section \ref{sec:capillary} we use some of the analysis of the previous sections to construct embedded capillary CMC surfaces in geodesic balls of $\S^3$ and $\H^3$, for all values of $H$.

\subsection{Open problems}
It is very natural to conjecture that the (spherical) critical catenoids are the only embedded free boundary minimal annuli in geodesic balls of $\S^3$; see \cite{LM2,Me}. Our results imply that the embeddedness assumption cannot be removed from this conjecture.

More generally, aligned with the theorems in \cite{FHM,CFM} for the Euclidean case, one can conjecture that any embedded free boundary CMC annulus in a geodesic ball of $\R^3, \S^3$ or $\H^3$ should be foliated by spherical curvature lines. We note that the existence of an embedded free boundary minimal annulus with spherical curvature lines in a geodesic ball of $\S^3$ is also open (in $\R^3$, the results in \cite{FHM} show that these do not exist).

In the immersed case, it would be interesting to construct a free boundary CMC annulus in a geodesic ball that is not foliated by spherical curvature lines, or to show that such an example cannot exist. We also do not know if there exist continuous deformations of free boundary CMC annuli in a fixed geodesic ball of $\R^3$, $\S^3$ or $\H^3$, with a fixed mean curvature $H$.

\section{Preliminaries}\label{sec:prelim}

Let $\M^3(c_0)$ denote the space form of constant curvature $c_0\in \R$. In the case $c_0\neq 0$, we view $\M^3(c_0)$ in the usual way as a hyperquadric of $\R_{\ep}^4=(\R^4,\esiz,\esde)$, $$\esiz, \esde=dx_1^2+dx_2^2+dx_3^2+\ep dx_4^2,$$ where $\ep$ is the sign of $c_0$. That is, $\R_{\ep}^4$ is either the Euclidean $4$-space (if $\ep=1$) or the Lorentz-Minkowski space $\mathbb{L}^4$ (if $\ep=-1$) and $$\M^3(c_0)=\{x\in \R_{\ep}^4 : \esiz x,x\esde = 1/c_0\},$$ with $x_4>0$ if $\ep=-1$. 

Let $\Sigma$ denote an immersed oriented surface in $\M^3(c_0)$ with constant mean curvature $H\in \R$. Let $N$ denote the unit normal vector field of $\Sigma$ in $\M^3(c_0)$. Let $\zeta:=u+iv$ denote a conformal parameter for $\Sigma$, so that its first fundamental form is $I=e^{2\omega} (du^2+dv^2)$. Then, the Codazzi equation gives that the Hopf differential $Q:=\esiz \psi_{\zeta \zeta},N\esde$ is holomorphic, and the Gauss equation for $\Sigma$ in the $(u,v)$-parameters is 
\begin{equation}\label{gauss1}
\Delta \omega + (H^2+c_0) e^{2\omega}- 4|Q|^2 e^{-2\omega}=0.  
\end{equation}

Assume that $\Sigma$ is simply connected and does not have umbilical points. Then, after a change of conformal parameter, we can assume that $Q$ is constant. In that way, the Gauss equation \eqref{gauss1} is of the form
\begin{equation}\label{gauss2}
\Delta \omega + A e^{2\omega}- B e^{-2\omega}=0,   \hspace{0.5cm} A,B\in \R, \, B>0.
\end{equation}

Conversely, let $\omega(u,v):\R^2\flecha \R$ satisfy \eqref{gauss2} with respect to constants $A,B$, and let $H,c_0,Q\in \R$ be constants, with $Q\neq 0$, so that 
\begin{equation}\label{valorab}
A=H^2+c_0, \hspace{0.5cm} B=4|Q|^2
\end{equation} hold. Then, there exists an immersion $$\psi(u,v):\R^2\flecha \M^3(c_0)$$ with constant mean curvature $H$, Hopf differential $Q$, and whose first and second fundamental forms are 
\begin{equation}\label{12s}
I=e^{2\omega} (du^2+dv^2), \hspace{0.5cm} II= (H e^{2\omega} + 2Q)du^2 + (H e^{2\omega} -2Q)dv^2
\end{equation} 
The principal curvatures $\kappa_1> \kappa_2$ of $\psi$ are 
\begin{equation}\label{pricur1}
\kappa_1 =H + 2|Q| e^{-2\omega}, \hspace{0.5cm} \kappa_2 =H- 2|Q| e^{-2\omega}.
\end{equation}
The smallest principal curvature $\kappa_2$ corresponds to the principal $u$-curves (resp. $v$-curves) if $Q<0$ (resp. $Q>0$).

This surface $\Sigma$ is unique up to orientation preserving ambient isometries, or equivalently, up to prescribing the moving frame $\{\psi,\psi_u,\psi_v,N\}$ at some point $(u_0,v_0)$. Here, $N$ is the unit normal of $\Sigma$ associated to the parametrization $\psi$. The Gauss-Weingarten formulas of $\psi$ are 
\begin{equation}\label{gw}
\def\arraystretch{1.7}\begin{array}{l}
\psi_{uu}= \omega_u \psi_u - \omega_v \psi_v +(H e^{2\omega} +2Q) N - c_0 e^{2\omega} \psi \\
\psi_{uv}= \omega_v \psi_u + \omega_u \psi_v  \\
\psi_{vv}= -\omega_u \psi_u + \omega_v \psi_v + (H e^{2\omega} -2Q) N - c_0 e^{2\omega} \psi\\
N_{u}= - (H +2Qe^{-2\omega}) \psi_u  \\
N_{v}= - (H -2Qe^{-2\omega}) \psi_v  
\end{array}
\end{equation}

We next analyze the property that a curvature line $v\mapsto \psi(u_,v)$ is \emph{spherical}, i.e. it lies in a totally umbilical surface of $\M^3(c_0)$.

The totally umbilical surfaces of $\M^3(c_0)$ are given by the intersection of hyperplanes of $\R^4_{\ep}$ with $\M^3(c_0)$. They can be described by 
\begin{equation}\label{eq:esferas}
S[m,d]:=\{x\in \M^3(c_0) : \esiz x,m\esde =d\},    
\end{equation}
for some $m\in \R_{\ep}^4\setminus\{\bf 0\}$ and $d\in \R$ (here $m$ and $d$ are determined up to a common multiplicative factor). The condition for $S[m,d]$ being a (non-empty) surface is that 
$$\esiz m,m\esde - c_0  d^2 >0$$ and that, if $\ep=-1$ and $\esiz m,m\esde \geq 0$, $d$ and the $x_4$-coordinate of $m$ have opposite signs. When $\ep=-1$ ($c_0<0$), we have that $S[m,d]$ is a sphere (resp. horosphere, pseudosphere) in $\H^3(c_0)$ if $\esiz m,m\esde$ is negative (resp. zero, positive). If $d=0$, then $S[m,d]$ is totally geodesic in $\M^3(c_0)$.

\begin{lemma}\label{lem:spherical}
For each fixed $u\in\R$, the $v$-curvature line $\psi(u,v)$ of $\psi$ is spherical if and only if there exist $\alfa(u),\beta(u)\in\R$ such that 
\begin{equation}\label{overdet}
2\omega_u = \alfa(u) e^{\omega} + \beta(u) e^{-\omega}.
\end{equation}
In that situation, $\psi$ intersects $S[m(u),d(u)]$ at a constant angle $\theta(u)$ along $v\mapsto\psi(u,v)$ and we have
\begin{equation}\label{expab}
\alfa=2 \frac{|\hat{N}| H \cos \theta - c_0 d}{ |\hat{N}| \sin\theta} ,\hspace{0.5cm} \beta = -4Q \frac{\cos\theta}{\sin\theta},
\end{equation}
where $|\hat{N}|= \sqrt{\esiz m,m\esde -  c_0 d^2 }.$ In particular, $\beta=0$ if and only if $\cos\theta=0$, and $\alfa=\beta=0$ if and only if $\cos \theta=0$ and $S[m,d]$ is totally geodesic in $\M^3(c_0)$. 
\end{lemma}

\begin{proof}
If for a fixed $u=u_0$ the curve  $v\mapsto \psi(u_0,v)$ lies in $S[m,d]$, then $\esiz m,\psi_{v}\esde=0$. One sees from here and $\esiz m,\psi\esde=d$ that 
$$\hat{N}:= m - c_0 d \psi$$ 
lies in $T\mathbb{M}^3(c_0)$ and is normal to $S[m,d]$ along $\psi(u_0,v)$, since $\esiz \hat{N},{\bf t}\esde =0$ for every ${\bf t}\in \R_{\ep}^4$ orthogonal to both $\psi$ and $m$. Defining $\theta$ by $\esiz N,\hat{N}\esde =|\hat{N}| \cos\theta$, one easily obtains that, changing $\theta$ by $-\theta$ if necessary,
\begin{equation}\label{eq:m}
    m= e^{-\omega} |\hat{N}| \sin\theta\, \psi_u + |\hat{N}| \cos \theta\, N + c_0 d \psi.
\end{equation}

Using this expression together with \eqref{gw} and $\esiz \psi_{vv},m\esde=0$, we obtain \eqref{overdet} for  $\alfa=\alfa(u_0),\beta=\beta(u_0)$ given by \eqref{expab}.

Conversely, assume \eqref{overdet} holds along the line $v\mapsto (u_0,v)$. Then, \eqref{overdet} together with \eqref{gw}, imply that the expression
\begin{equation}\label{emetil}
\tilde{m}=4Q e^{-\omega} \psi_u - \beta N -(2Q \alpha + H\beta)\psi
\end{equation}
satisfies $\tilde{m}_v(u_0,v)=0$, i.e., $\tilde{m}$ is constant along $v\mapsto (u_0,v)$. Therefore, $\esiz \tilde{m},\psi\esde$ is also constant along $v\mapsto (u_0,v)$, i.e., the curvature line $\psi(u_0,v)$ is spherical.
\end{proof}

\section{Special solutions of the sinh-Gordon equation}\label{sec:solsg}

In this section we recall the construction in our previous paper \cite{CFM} of some special solutions of the elliptic sinh-Gordon equation, that will be used later on. We will make a more detailed discussion than in \cite{CFM}, in order to motivate their origin. We will also indicate some new additional properties of these solutions that will be important for our purposes here.

\subsection{Wente's overdetermined system}

We will seek solutions $\rho(u,v):\R^2\flecha \R$ to the overdetermined system
\begin{equation}\label{sinhg}
\Delta \rho + \cosh \rho \sinh \rho =0,
\end{equation}
\begin{equation}\label{oversin}
2\rho_u = \hat{\alfa}(u) e^{\rho} + \hat{\beta}(u) e^{-\rho}.
\end{equation}
for functions $\hat{\alfa},\hat{\beta}:\R\flecha \R$. Note that the system \eqref{sinhg}-\eqref{oversin} is precisely the system \eqref{gauss2}-\eqref{overdet} for the choices $A=B=1/4$. We will write \eqref{oversin} in the alternative form
\begin{equation}\label{omu}
\rho_u = y(u)\cosh\rho +z(u) \sinh\rho,
\end{equation}
where $y(u),z(u)$ are real functions. The next discussion is taken  from Wente \cite{W}.

To start, in order for \eqref{sinhg}-\eqref{omu} to hold, $y(u), z(u)$ should be a solution to the differential system 
\begin{equation}\label{system1}
\left\{\def\arraystretch{1.3} \begin{array}{lll} y'' & = & (\hat{a}-1) y - 2 y (y^2-z^2), \\ z'' & = & \hat{a}z - 2 z (y^2-z^2),\end{array} \right.
\end{equation}
with respect to some constant $\hat{a}$.

Moreover, if we denote $Z(u,v):=e^{\rho(u,v)}$, then
\begin{equation}\label{edop}
4Z_v^2 = p(u,Z),
\end{equation}
where 
\begin{equation}\label{def:pu}
p(u,x):= -(1+(y+z)^2) x^4 -4(y'+z') x^3+ 6\hat{\gamma} x^2+4(y'-z')x- (1+(y-z)^2).
\end{equation}
Here, we are denoting $y=y(u)$, $z=z(u)$, and $6\hat{\gamma}=6(y^2-z^2)-4(\hat{a}-1/2)$. Thus, for each fixed $u$ value, $p(u,x)$ is a polynomial of degree four.

This process can be reversed under some additional conditions, once we prescribe the initial values 
 \begin{equation}\label{ivasis2}
 (y(0),z(0),y'(0),z'(0),\hat{a},\rho(0,0)).
 \end{equation} 

The system \eqref{system1} has a Hamiltonian structure. The basic Hamiltonian constants of \eqref{system1} are 
\begin{equation}\label{def:h}
y'^2-z'^2 - (\hat{a}-1)y^2 + \hat{a} z^2 + (y^2-z^2)^2 = h \in \R
\end{equation}
and
\begin{equation}\label{def:k}
(z y'-y z')^2 + z'^2 + z^2 (y^2-z^2 -\hat{a}) = k\in \R.
\end{equation}

One can then follow the classical procedure to solve the Hamilton-Jacobi equations by separation of variables, see \cite{W}. First, we apply the change of variables
 \begin{equation}\label{change}
y^2 = -(1-s) (1-t), \hspace{1cm} z^2 = -s t.
\end{equation}
Using \eqref{change}, the system \eqref{system1} transforms into the autonomous first order system
\begin{equation}\label{system2}
\left\{\def\arraystretch{1.3} \begin{array}{lll} s'(\landa)^2 & = & s (s-1) q(s), \hspace{0.5cm} (s\geq 1), \\ t'(\landa)^2 & = & t (t-1) q(t), \hspace{0.6cm} (t\leq 0),\end{array} \right.
\end{equation}
where $q(x)$ is the third-degree polynomial
\begin{equation}\label{def:q}
q(x)=-x^3 + (\hat{a}+1) x^2 + (h-\hat{a})x + k,
\end{equation}
and the parameter $\landa$ of \eqref{system2} is linked to $u$ by 
\begin{equation}\label{changecor}
2 u'(\landa)= s(\landa)-t(\landa)>1.
\end{equation}

\subsection{Constructions of special solutions to the system}
 
We now explain our construction in \cite{CFM}. To start, we will fix 
\begin{equation}\label{yzz}
y(0)=z(0)=0
\end{equation} 
by geometrical reasons. More specifically, we intend to use \eqref{sinhg} as the Gauss equation for a CMC surface $\Sigma$ in some space $\M^3(c_0)$, see Section \ref{sec:romega} below. These initial conditions will determine that $\Sigma$ intersects orthogonally a totally geodesic surface of $\M^3(c_0)$. See also Lemma \ref{lem:spherical}. In particular $\Sigma$ will have a useful \emph{symmetry plane}.

We are interested in the possibility of obtaining embedded examples of free boundary CMC annuli in geodesic balls of $\M^3(c_0)$. After fixing the initial conditions \eqref{yzz}, a lengthy, detailed inspection of all the cases that we will not reproduce here seems to indicate that this embeddedness is only possible when $p(0,x)$ in \eqref{def:pu} has two positive roots and two negative roots. So, we should prescribe the values of $y'(0), z'(0)$ and $\hat{a}$ in a way that this property holds for $p(0,x)$ in \eqref{def:pu}. We will also be interested in the situation where the two positive roots of $p(0,x)$ collapse into one, since this situation will detect Delaunay examples in $\M^3(c_0)$.

Because of this, it seems convenient to seek initial conditions $y'(0),z'(0),\hat{a}$ so that $p(0,x)$ in \eqref{def:pu} can be written in the factor form below, which will force it to have the desired root structure: 
 \begin{equation}\label{def:pab}
p(0,x)= -\left(x-\frac{a}{c}\right)\left(x-\frac{1}{a c}\right)\left(x+b c\right)\left(x+\frac{c}{b}\right),
 \end{equation}
where $(a,b,c)$ lie in the parameter domain 
\begin{equation}\label{domain}
\cO :=\{(a,b,c)\in \R^3 : a\geq 1, b\geq 1, c\geq 1\}.
\end{equation}
Clearly, $a=1$ if and only if $p(0,x)$ has a  double positive root. Note that, from \eqref{yzz} and \eqref{def:pu}, we have
\begin{equation}\label{def:pu2}
p(0,x):= -x^4-4(y'(0)+z'(0))x^3 + (2-4\hat{a})x^2 +4(y'(0)-z'(0))x -1.
\end{equation}
By comparing \eqref{def:pab} with \eqref{def:pu2} we can express $\hat{a},y'(0),z'(0)$ in terms of $(a,b,c)$. Specifically, denoting
\begin{equation}\label{deABC}
\cA:=\frac{1}{2}\left(a+\frac{1}{a}\right), \hspace{0.5cm}\cB:=\frac{1}{2}\left(b+\frac{1}{b}\right), \hspace{0.5cm}\cC:=\frac{1}{2}\left(c-\frac{1}{c}\right),
\end{equation}
we find that
\begin{equation}\label{aabc}
\hat{a}=1-\cA \cB + \cC^2,
\end{equation}
and 
\begin{equation}\label{pqabc}
y'(0)=\frac{(\cA+\cB)\cC}{2},\hspace{0.5cm} z'(0)= \frac{(\cB-\cA) \sqrt{\cC^2+1}}{2}.
\end{equation}
We will also fix the initial condition $\rho(0,0)$ for $\rho(u,v)$ by
\begin{equation}\label{omini}
e^{\rho(0,0)}=\frac{1}{ac}.
\end{equation}

As a result, we have:

\begin{theorem}[\cite{CFM}]\label{th:otro}
For each $(a,b,c)\in \cO$ there exists a solution $\rho(u,v)$ to \eqref{sinhg} globally defined on $\R^2$ that satisfies the additional overdetermined condition \eqref{omu}. The solution $\rho(u,v)$ is unique after fixing the initial conditions in \eqref{ivasis2} by \eqref{yzz}, \eqref{aabc}, \eqref{pqabc} and \eqref{omini}.
\end{theorem}

We note that, since $y(0)=z(0)=0$, we have from \eqref{omu} that $\rho_u(0,v)=0$ for all $v$. Thus, by uniqueness to the Cauchy problem for \eqref{sinhg} along $(0,v)$, we have
\begin{equation}\label{simro1}
\rho(-u,v)=\rho(u,v), \hspace{0.5cm} \forall \, (u,v)\in \R^2.
\end{equation}

Denote 
\begin{equation}\label{defx}
X(v):=e^{\rho(0,v)}.
\end{equation}
Then, it holds
\begin{equation}\label{eq:Xv0}
4X_v^2 = p(0,X),
\end{equation}
where $p(0,x)$ is given by \eqref{def:pab}. It follows from \eqref{edop} and \eqref{omini} that $X(v)$ takes its values in the interval $[1/(ac),a/c]$. Recall that when $a=1$, the two positive roots of $p(0,x)$ collapse into a double root at $x=1/c$, and in that case $X(v)=1/c$, constant.

Because of \eqref{edop} and \eqref{def:pab}, one can prove (see \cite{CFM}) that, when $a>1$, $X(v)$ varies monotonically from $[0,\sigma]$ to the interval $[1/(ac),a/c]$, where 
\begin{equation}\label{desi}
\sigma:= \int_{1/(ac)}^{a/c} \frac{2}{\sqrt{p(0,x)}} \, dx >0.
\end{equation} 
In this way, $X(v)$ is $2\sigma$-periodic, and $X'(v)=0$ only when $v=k\sigma$, with $k\in \Z$.

Note that $\sigma=\sigma(a,b,c)$ is only defined at first in $\cO\cap \{a>1\}$. However we have from \cite[Remark 5.2]{CFM}:

\begin{proposition}\label{prop:exsi}
The function $\sigma(a,b,c)$ extends analytically to the set $\{(a,b,c): a,b>0,c\geq 1\}$, and 
\begin{equation}\label{exsi}
\sigma(1,b,c)=\frac{2\pi c}{\sqrt{1+(b+1/b)c^2 + c^4}}.
\end{equation}
\end{proposition}

It was also proved in \cite{CFM} that the solution $\rho(u,v)$ has the symmetries 
\begin{equation}\label{simom}
\rho(u,k\sigma +v)= \rho(u,k\sigma-v),
\end{equation} 
with respect to $\sigma$, for any $k\in \Z$. 

Given $(a,b,c)\in \cO$, let $(y(u),z(u)):\R\flecha \R^2$ be the associated solution to system \eqref{system1} with our choice of initial conditions. The following result was proved in \cite[Prop. 4.1]{CFM}:

\begin{proposition}\label{prosis1}
If $c>1$, then there exists a unique $u_1>0$ such that:
\begin{enumerate}
\item
$y(u_1)=0$.
\item
$y(u)>0$ for every $u\in (0,u_1)$.
 \item
If $z(u)$ is not identically zero, then $z(u)\neq 0$ for every $u\in (0,u_1]$.
\end{enumerate}
Moreover, the map $u_1=u_1(a,b,c)$ is real analytic in $\cO\cap \{c>1\}$.
\end{proposition}

Given $(a,b,c)\in \cO$, we define the \emph{free boundary region} 
\begin{equation}\label{defi:W}
\cW :=\left\{(a,b,c)\in \cO : b\geq a,  \, \cC^2>\frac{(\cA-\cB)^2}{4\cA\cB}\right\},
\end{equation}
where $\cA,\cB,\cC$ are defined in \eqref{deABC}. We remark that if $(a,b,c)\in \cW$, then $a\geq 1$ and $c>1$. The second inequality in \eqref{defi:W} is equivalent to $z'(0)^2< y'(0)^2$ for the initial values in \eqref{pqabc}.

By Lemma \ref{lem:spherical}, the intersection angle of the spherical curvature line $\psi(u_0,v)$ with its supporting totally umbilic surface is $\pi/2$ if and only if $\hat\beta(u_0)=0$, 
i.e., $y(u_0)=z(u_0)$. Thus, $y=z$ detects the free boundary condition. The following result was proved in \cite[Prop. 4.3]{CFM}:

\begin{proposition}\label{prosis2}
Let $(a,b,c)\in \cW$. Then, there is a unique $\tau\in (0,u_1]$ such that $y(\tau)=z(\tau)$, and $y(u)>z(u)$ for every $u\in (0,\tau)$. 

Moreover, the map $(a,b,c)\in \cW \mapsto \tau(a,b,c)$ is real analytic.
\end{proposition}

Also, by \cite[Rem. 4.4]{CFM}, in case $(a,b,c)\in \cO$ with $c>1$, $b\geq a$ and 
\begin{equation}\label{fbcon}
\cC^2\leq \frac{(\cA-\cB)^2}{4\cA\cB},
\end{equation}
it holds $y(u) < z(u)$ for $u > 0$ small enough.

\section{CMC surfaces with spherical curvature lines}\label{sec:romega}

\subsection{A special class of CMC surfaces with spherical curvature lines}

We now fix $(a,b,c)\in \cO$ as in \eqref{domain}, and let $\rho(u,v)$ denote the solution to \eqref{sinhg}-\eqref{omu} constructed in Section \ref{sec:solsg} from $(a,b,c)$, see Theorem \ref{th:otro}. 

We consider constants $H\geq 0$ and $\ep\in \{1,-1\}$ so that $\mu^2:=H^2+\ep>0$, and define $Q:=1/(8\mu) >0$. We also define $\omega(u,v):\R^2\flecha \R$ by
\begin{equation}\label{cambiometrica}
\omega := \rho - \log(2\mu).
\end{equation} 
By \eqref{sinhg}, $\omega$ satisfies the Gauss equation \eqref{gauss1} for $c_0=\ep$. Therefore, as explained in Section \ref{sec:prelim}, we obtain a unique (up to ambient isometry) immersion $\psi(u,v):\R^2\flecha \M^3(\ep)= \S^3 \text{ or } \H^3$ with constant mean curvature $H$, with first and second fundamental forms given by 
\begin{equation}\label{unodos}
I= \frac{e^{2\rho}}{4\mu^2} (du^2+dv^2), \hspace{0.5cm} II= \left(\frac{H e^{2\rho}}{4\mu^2}+\frac{1}{4\mu}\right) du^2 + \left(\frac{H e^{2\rho}}{4\mu^2}-\frac{1}{4\mu}  \right)  dv^2.
\end{equation} 

The Gauss-Weingarten formulas \eqref{gw} for $\psi(u,v)$ in $\M^3(\ep)\subset \R_{\ep}^4$ are here 
\begin{equation}\label{gw2}
\def\arraystretch{1.7}\begin{array}{l}
\psi_{uu}= \rho_u \psi_u - \rho_v \psi_v +\frac{1}{4\mu^2}(H e^{2\rho} +\mu) N -  \frac{\ep}{4\mu^2}  e^{2\rho} \psi \\
\psi_{uv}= \rho_v \psi_u + \rho_u \psi_v  \\
\psi_{vv}= -\rho_u \psi_u + \rho_v \psi_v +\frac{1}{4\mu^2}(H e^{2\rho} - \mu) N -  \frac{\ep}{4\mu^2}  e^{2\rho} \psi \\
N_{u}= -(H  + \mu e^{-2\rho}) \psi_u  \\
N_{v}= - (H  - \mu e^{-2\rho}) \psi_v  
\end{array}
\end{equation}

The principal curvatures $\kappa_1>\kappa_2$ of the immersion are 
\begin{equation}\label{princur}
\kappa_1 = H + \mu e^{-2\rho}, \hspace{0.5cm} \kappa_2= H - \mu e^{-2\rho},
\end{equation}
where $\kappa_2$ corresponds to the principal curvature associated to the $v$-curvature lines $u={\rm const}$. 

\begin{definition}\label{def:sabc}
Given $H\geq 0$ and $\ep\in \{-1,1\}$ with $H^2+\ep>0$, and $(a,b,c)\in \cO$, we will let $\Sigma=\Sigma(a,b,c)$ denote the immersed surface in $\M^3(\ep)$ with constant mean curvature $H$ constructed above from $\rho(u,v;a,b,c)$.
\end{definition}
Because of \eqref{omu} and \eqref{cambiometrica} we see that 
$$2\omega_u = \alfa(u) e^{\omega} +\beta(u) e^{-\omega}$$
holds for $\alfa,\beta:\R\to\R$ given by 
\begin{equation}\label{eq:hat}
\alfa=2\mu(y+z), \qquad \beta = \frac{y -z }{2\mu} ,
\end{equation}
where $y(u),z(u)$ are the functions in \eqref{omu}. Thus, by Lemma \ref{lem:spherical}, the surface $\Sigma=\Sigma(a,b,c)$ is foliated by spherical curvature lines. Specifically, for each $u\in\R$ there exists $m(u)\in\R^4_\ep\setminus\{\bf 0\}$ and $d(u)\in\R$ such that the  curvature line $v\mapsto \psi(u,v)$ lies in a totally umbilic surface (see \eqref{eq:esferas})
\begin{equation}\label{eum}
S[m(u),d(u)]\subset \M^3(\ep).
\end{equation}
Since $m(u)$ and $d(u)$ are defined up to a multiplicative factor, we can assume that $\langle m(u),m(u)\rangle \in\{0,-1,1\}$. So, for $\ep=1$ we will assume $m(u)\in\M^3(\ep)=\S^3$, whereas for $\ep=-1$ we will suppose $m(u)\in\M^3(\ep)=\H^3$ when $\esiz m(u),m(u)\esde<0$, i.e., whenever the surface $S[m(u),d(u)]$ is compact.

\begin{remark}
\emph{Note that we are assuming that the spherical curvature lines of $\Sigma$ are the ones associated to the smaller principal curvature $\kappa_2$ in \eqref{princur}.}
\end{remark}

\subsection{Study of the center map}
We study now the behaviour of $m(u)$ in \eqref{eum}. 

\begin{proposition}\label{pro:centros} 
The map $m(u):\R\flecha \R_{\ep}^4\setminus \{{\bf 0}\}$ lies in a $2$-dimensional subspace $\mathcal{P}\subset \R^4_\varepsilon$. In particular, if $\ep=1$, or if $\ep=-1$ and $S[m(u_0),d(u_0)]$ is compact for some $u_0 \in \R$, then $m(u_0)$ lies in the geodesic of $\M^3(\ep)$ given by $\mathcal{P}\cap \M^3(\ep)$. Moreover, $m(u)$ is analytic at $u = u_0$ and $m'(u_0)\neq 0$. 
\end{proposition}

\begin{proof} 
Let us consider the function $\tilde m = \tilde m(u)$ defined in \eqref{emetil}. Since the conformal factor of the metric of $\Sigma$ is $e^\omega = e^\rho/2\mu$ and its Hopf differential is $Q=1/8\mu$, (see \eqref{unodos}), and taking into account equations \eqref{eq:hat}, \eqref{eq:m} and \eqref{expab}, we can write 
\begin{equation}\label{funcioncentro}
\tilde m=  e^{-\rho}\psi_u  - \frac{(y-z)}{2\mu} N - \frac{1}{2\mu}\left(\mu(y+z) + H(y-z)  \right)\psi
\end{equation}  
and it holds
\begin{equation}\label{mtildem}
    m =  2\mu|\hat N|\sin\theta\, \tilde m.
\end{equation}
So, it suffices to prove that $\tilde m(u)$ lies in a plane $\cP$ of $\R_{\ep}^4$. 

A long but straightforward computation using \eqref{omu}, \eqref{system1} and \eqref{gw2} shows that $\tilde{m}''(u)$ is proportional to $\tilde m(u)$. More specifically,  
\begin{equation}\label{tildem}
    \tilde m'' = \left(\hat a - \frac{H + \mu}{2\mu} - 2y^2 + 2z^2\right)\tilde m.
\end{equation}

We prove first that $\tilde m(0)$ and $\tilde m'(0)$ are linearly independent. 
This follows from a computation using \eqref{pqabc}, \eqref{omini} which shows that $\langle \tilde m'(0), N(0,0) \rangle = \frac{2H + 2a\mathcal{B} \mu + (a^2 - 1)c^2 \mu}{4ac\mu} > 0$, whereas $\langle \tilde m(0), N(0,0) \rangle = 0$ and $\tilde m(0)\neq 0$. 

Let $\mathcal{P}\subset \R_\varepsilon^4$  be the plane generated by $\tilde m(0)$ and $\tilde m'(0)$, and consider $w_1, w_2 \in \R_\varepsilon^4$ two linearly independent vectors which are orthogonal to $\mathcal{P}$. The functions $f_i(u) := \langle \tilde m(u), w_i\rangle$, $i = 1,2$ satisfy the differential equation
\begin{equation}
    f_i''(u) = \left(\hat a - \frac{H + \mu}{2\mu} - 2y^2 + 2z^2\right) f_i(u),
\end{equation}
as well as the initial conditions $f_i(0) = f_i'(0) =0$, which imply that $f_i(u) \equiv 0$. In particular, $\tilde m(u) \in \mathcal{P}$. Since $m(u)$ is just a scalar multiple of $\tilde m(u)$, we deduce that $m(u)$ also lies in that plane.

Now assume that either $\varepsilon = 1$, or $\varepsilon = -1$ and $S[m(u_0),d(u_0)]$ is compact for some $u_0 \in \R$. Let us see that $m(u)$ is analytic at $u = u_0$. By the normalization condition on $m$, $\langle m(u_0), m(u_0)\rangle = \varepsilon$ and so $m(u_0) \in \M^3(\varepsilon) \cap \mathcal{P}$. Moreover, if $\varepsilon = -1$ then $\tilde m$ is timelike at $u_0$, and so the same holds on a neighbourhood of $u_0$. Thus, we can write 
$$m(u) = \pm \frac{\tilde m(u)}{\sqrt{\varepsilon \langle \tilde m(u), \tilde m(u) \rangle}}$$
for $u$ near $u_0$. Since $\tilde m(u)$ is analytic for all $u \in \R$ (see \eqref{funcioncentro}),  this ensures the analyticity of $m(u)$ at $u=u_0$.

Finally, assume by contradiction that $m'(u_0) = 0$. This implies that $\tilde m'(u_0)$ is a scalar multiple of $\tilde m(u_0)$, so we can find a vector $w_3 \in \mathcal{P}$ orthogonal to both $\tilde m(u_0)$ and $\tilde m'(u_0)$. By the same arguments shown before, we deduce that the function $f_3(u) := \langle \tilde m(u), w_3\rangle$ must vanish identically. Consequently, $\tilde m(u)$ lies in a line of $\mathcal{P} \subset\R^4_\varepsilon$ and so $\tilde m'(u)$ and $\tilde m(u)$ are linearly dependent for all $u$. However, as we proved before, this is not the case when $u = 0$, so we reach a contradiction.
\end{proof}

\subsection{Study of the symmetries of \texorpdfstring{$\Sigma$}{}}

In order to have $\psi(u,v)$ uniquely defined (and not just up to ambient isometry), we will fix the following initial data for the moving frame $(\psi,\psi_u,\psi_v,N)$ at $(u,v)=(0,0)$:
\begin{equation}\label{idat}
\psi(0,0)={\bf e}_4, \hspace{0.2cm} \psi_u(0,0)= \frac{e^{\rho(0,0)}}{2\mu} {\bf e}_3, \hspace{0.2cm} \psi_v(0,0)= -\frac{e^{\rho(0,0)}}{2\mu} {\bf e}_2, \hspace{0.2cm} N(0,0)={\bf e}_1,
\end{equation}
where $({\bf e}_1,{\bf e}_2,{\bf e}_3,{\bf e}_4)$ is the canonical basis of $\R_{\ep}^4$.

\begin{lemma}\label{plaor}
If $\Phi$ denotes the symmetry in $\R_{\ep}^4$ with respect to $x_3=0$, then it holds 
\begin{equation}\label{simor}
\psi(-u,v)=\Phi(\psi(u,v)), \hspace{0.5cm} \forall (u,v)\in \R^2.
\end{equation}
Thus, $\Sigma$ is symmetric with respect to the totally geodesic surface ${\bf S}:= \M^3(\ep)\cap \{x_3=0\}$.
\end{lemma}
\begin{proof}
It is a direct consequence of the fundamental theorem for surfaces in $\M^3(\ep)$ and the expressions for $I,II$ in \eqref{unodos}, using \eqref{idat} and the symmetry condition \eqref{simro1} for $\rho(u,v)$. Note that, in particular, the curve $v\mapsto \psi(0,v)$ lies in $\{x_3=0\}$.
\end{proof}

The surface $\Sigma$ actually has infinitely many more \emph{symmetry planes} orthogonal to $\{x_3=0\}$, as detailed below.

\begin{proposition}\label{pro:plaver}
For each $k\in \Z$, let $P_k$ denote the $3$-dimensional subspace $P_k$ of $\R_{\ep}^4$ that is orthogonal to the vector $\nu_k :=\psi_v(0,k\sigma)$, where 
$\sigma>0$ is given by \eqref{desi}. Then the following properties hold:
\begin{enumerate}
\item
$P_k$ is orthogonal to $\{x_3=0\}$, and it is timelike if $\ep=-1$.
\item
If $\Psi_k$ denotes the symmetry of $\R_{\ep}^4$ with respect to $P_k$, it holds 
\begin{equation}\label{simver}
\psi(u,k\sigma-v)=\Psi_k(\psi(u,k\sigma+v)) \hspace{0.5cm} \forall (u,v)\in \R^2.
\end{equation}
In particular, the curve $u\mapsto \psi(u,k\sigma)$ lies in the totally geodesic surface of $\M^3(\ep)$ $$\Omega_k:=P_k \cap \M^3(\ep).$$
\item
If the $P_k$'s do not all coincide, then $\Lambda:={\rm span} \{ \nu_k  : k\in \Z\}$ is a $2$-dimensional subspace of $\R_{\ep}^4$. Moreover, the angle between $\nu_k$ and $\nu_{k+1}$ in $\Lambda$ does not depend on $k$.
\item
If the $P_k$'s do not all coincide, then $\cap_{k\in \Z} P_k = \Lambda^{\perp}=\mathcal{P}$, where $\mathcal{P}$ is the $2$-dimensional subspace of $\R_{\ep}^4$ defined in Proposition \ref{pro:centros}.
\end{enumerate}
\end{proposition}
\begin{proof}
Recall that, by Lemma \ref{plaor}, the curve $v\mapsto \psi(0,v)$ lies in $\{x_3=0\}$. Thus, $\esiz \nu_k,{\bf e}_3\esde=0$, for all $k$. Since $\langle \nu_k,\nu_k\rangle=e^{2\omega(0,k\sigma)}>0$, $P_k$ is timelike when $\ep=-1$. This proves item (1). We note that $P_0=\{x_2=0\}$, from \eqref{idat}.

Item (2) follows from the fundamental theorem of surface theory in $\M^3(\ep)$, equation \eqref{unodos} and the symmetry condition \eqref{simom} for $\rho(u,v)$.

As regards item (3), we have from \eqref{simver} that $-\psi_v(0,k\sigma-v)=\Psi_k(\psi_v(0,k\sigma+v))$. From here we obtain that $-\nu_{k+1}=\Psi_k(\nu_{k-1})$. This easily implies that $\Lambda$ is $2$-dimensional (unless all the planes $P_k$ coincide, in which case it is $1$-dimensional), and that $\esiz \nu_{k+1},\nu_k\esde =\esiz \nu_k,\nu_{k-1}\esde$, from where the family $\{\nu_k : k\in \Z\}$ is equiangular.

Regarding item (4), we first note that $\cap_{k\in \Z} P_k = \Lambda^{\perp}$ is immediate from item (3). On the other hand, it follows from the proof of Proposition \ref{pro:centros} that $\mathcal{P}={\rm span}\{\tilde{m}(0),\tilde{m}'(0)\}$, where $\tilde{m}(u)$ is given by \eqref{funcioncentro}. By \eqref{funcioncentro} we have $\esiz \tilde{m}(0),\nu_k\esde =0$ for all $k\in \Z$. Differentiating $\tilde m$ with respect to $u$ in \eqref{funcioncentro} and using \eqref{unodos} and \eqref{gw2} we get 
$$\esiz \tilde m', \psi_v\esde = e^{-\rho}\langle \psi_{uu},\psi_v\rangle = -\frac{e^{\rho}}{4\mu^2} \rho_v. $$
In particular, since $\rho_v(0,k\sigma)=0$ (see \eqref{simom}), then $\esiz \tilde{m}'(0),\nu_k\esde =0$. Thus, $\Lambda^{\perp}=\mathcal{P}$, as claimed. 
\end{proof}

\begin{definition}\label{omenos}
We let $\cO^{-}$ be the open set of $\cO$ (see \eqref{domain}) for which the condition 
\begin{equation}\label{space}
\esiz \nu_0,\nu_0\esde \esiz \nu_1,\nu_1\esde -\esiz \nu_0,\nu_1\esde^2 >0
\end{equation}
holds, where $\nu_0:=\psi_v(0,0)$ and $\nu_1:=\psi_v(0,\sigma)$.
\end{definition}

We note that if $\ep=1$, the condition \eqref{space} merely indicates that $\{\nu_0,\nu_1\}$ are linearly independent. This is equivalent to asking that the linear subspace $\Lambda\subset \R^4$ in Proposition \ref{pro:plaver} has dimension two, i.e. that the hyperplanes $P_k$ are not all coincident.

If $\ep=-1$, the condition \eqref{space} tells additionally that $\Lambda$ is a spacelike subspace of $\mathbb{L}^4$, and so $\mathcal{P}\cap \H^3$ is a geodesic of $\H^3$, by item (4) of Proposition \ref{pro:plaver}.

\subsection{The period map}\label{sec:periodo}

Choose now $(a,b,c)\in \cO^-$, and let $\Sigma=\Sigma(a,b,c)$ be the immersion $\psi(u,v):\R^2\flecha \M^3(\ep)$ of Definition \ref{def:sabc}. Each principal curve $\psi(u_0,v)$ is spherical, and $\Sigma$ intersects orthogonally the totally geodesic slice $\M^3(\ep)\cap \{x_3=0\}$ of $\M^3(\ep)$ along $\Gamma(v):=\psi(0,v)$, see Lemma \ref{plaor}).
We want to characterize next the periodicity of the curve $\Gamma(v)$, and to understand its rotation index and symmetry group when it is periodic. Note that the periodicity of $\Gamma(v)$ automatically implies that all the (spherical) $v$-curvature lines of $\Sigma$ are closed.

Since $(a,b,c)\in \cO^-$, the subspace $\Lambda$ in Proposition \ref{pro:plaver} is two-dimensional, and it is spacelike if $\ep =-1$; see the discussion after Definition \ref{omenos}. Moreover, $\Lambda^\perp = \mathcal{P}$. Thus, $\M^3(\ep)\cap \mathcal{P}$ is a geodesic of $\M^3(\ep)$. Note that ${\bf e}_3\in \mathcal{P}$, since $\esiz \nu_k,{\bf e}_3\esde = 0$ for all $k$. Also, $\mathcal{P} \subset \{x_2=0\}$, since $\nu_0$ is collinear with ${\bf e}_2$ by \eqref{idat}. In this way, after a linear isometry of $\R_{\ep}^4$ that fixes ${\bf e}_2, {\bf e}_3$, we can assume that 
\begin{equation}\label{lander}
\mathcal{P} =\{x_1=x_2=0\}.
\end{equation}
Note that this isometry changes the initial conditions for $\psi(0,0)$ and $N(0,0)$ in \eqref{idat}, but it does so in a real analytic way, since $\sigma=\sigma(a,b,c)$ is real analytic (see Proposition \ref{prop:exsi}).

\begin{remark}\label{newinicondi}From now on, we will use these new initial conditions for any surface $\Sigma(a,b,c)$ with $(a,b,c) \in \cO^-$.
\end{remark}

Let $\varphi$ be the stereographic projection of $\M^3(\ep)$ from $-{\bf e}_4$. So, if $\ep=1$, $\varphi$ maps $\S^3\setminus\{ -{\bf e}_4 \}$ into $\R^3$, while if $\ep=-1$, $\varphi$ maps $\H^3$ into the unit ball $\B^3\subset \R^3$.
Since $\Gamma$ lies in $\{x_3=0\}$, the curve $\gamma:=\varphi\circ \Gamma$ is then a planar curve in $\{z=0 \} \subset \R^3$, where $(x,y,z)$ denote the Euclidean coordinates of $\R^3$. 

\begin{definition}\label{def:period}
Using the above notation, we define the period map as
$$ \Theta: \cO^- \to \R, $$
\begin{equation}\label{def:per0}
\Theta(a,b,c):=\frac{1}{\pi}\int_0^{\sigma} \kappa_{\gamma} ||\gamma'|| dv,
\end{equation} 
where $\kappa_{\gamma}$ and $||\gamma'||$ denote the Euclidean curvature and the length of $\gamma$, respectively. Note that $\pi \Theta$ represents the variation of the (Euclidean) unit normal of $\gamma(v)$ along the interval $[0,\sigma]$.
\end{definition}

\begin{proposition}\label{antet}
The map $\Theta=\Theta(a,b,c)$ in \eqref{def:per0} is real analytic in $\cO^-$.
\end{proposition}
\begin{proof}
It is an immediate consequence of the real analyticity of $\sigma=\sigma(a,b,c)$ (Proposition \ref{prop:exsi}) and $\rho=\rho(u,v;a,b,c)$, together with the analytic dependence of the Gauss-Weingarten system \eqref{gw2} with respect to initial conditions.
\end{proof}

We explain next the geometry of $\Sigma$ when its associated period is a rational number. We start by describing the geometry of the planar geodesic $\Gamma(v)=\psi(0,v)$.

\begin{proposition}\label{pro:peri}
Assume that $\Theta(a,b,c)=m/n\in \Q$, with $n\in \N\setminus\{0\}$ and $m/n$ irreducible. Then $\Gamma(v+2n\sigma)=\Gamma(v)$. In particular $\Gamma(v)$ is a closed curve. 

Moreover, $\Gamma(v)$ has rotation index $m\in \Z$ and, if $a>1$, a dihedral symmetry group $D_n$ with $n\geq 2$.
\end{proposition}

\begin{proof}
Let us see first that $\gamma=\varphi\circ\Gamma$ satisfies $\gamma(v+2n\sigma)=\gamma(v)$, where as before $\varphi$ denotes the stereographic projection of $\M^3(\ep)$ from $-{\bf e}_4$. 
Consider the totally geodesic surfaces $\Omega_k=P_k\cap \M^3(\ep)$ (see Proposition \ref{pro:plaver}). Then, $\varphi$ maps $\Omega_k$ into vertical planes $\Pi_k$ containing the $z$-axis. Let $L_k:=\Pi_k\cap \{z=0\}$. Then $\{L_k:k\in \Z\}$ is a family of equiangular lines in $\{z=0\}\equiv\R^2$ passing through the origin. It follows from Proposition \ref{pro:plaver} and the fact that $\varphi$ is conformal that the angle between $L_k$ and $L_{k+1}$ is the angle betweeen $\gamma'(0)$ and $\gamma'(\sigma)$. We denote this angle by $\vartheta$. Note that $\gamma'(k\sigma)$ is orthogonal to $L_k$.
This implies that 
\begin{equation}\label{2teta}
\pi \Theta=2\pi l+ \vartheta
\end{equation}
for some $l\in \Z$. In particular, $n\vartheta\in \pi \Z$.

By \eqref{simver}, we have 
\begin{equation}\label{siga}
\gamma(k\sigma-v)=T_k(\gamma(k\sigma +v)),
\end{equation} where $T_k$ denotes the symmetry of $\R^2$ that fixes $L_k$. If $\cR$ denotes the rotation around the origin of angle $2\vartheta$, then we have by \eqref{siga} that $\gamma(v+2\sigma)=\cR(\gamma(v))$. From here and $n\vartheta\in \pi \Z$ we obtain $\Gamma(v+2n\sigma)=\Gamma(v)$.

Also from \eqref{siga} we obtain that $f:= ||\gamma'|| \kappa_{\gamma}$ satisfies $f(k\sigma-v)=f(k\sigma +v)$ for all $k\in \Z$. Observe that this implies that, for all $k\in \Z$, 
$$\Theta =\frac{1}{\pi}\int_{k\sigma}^{(k+1)\sigma} \kappa_{\gamma} ||\gamma'|| dv.$$ 
From here, $\Theta=m/n$ and the $2\sigma n$-periodicity of $\gamma(v)$, it follows that the rotation index of $\gamma(v)$ is equal to $m$.

Finally, we determine the symmetry group of $\Gamma(v)$ when $a>1$. In that case we know that $X(v)$ in \eqref{defx} only has critical values at the points of the form $k\sigma$, $k\in \Z$. Also, by \eqref{princur} and since $\psi(u,v)$ intersects ${\bf S}=\M^3(\ep)\cap\{x_3=0\}$ orthogonally along $\Gamma(v)$, we have $$\kappa_{\Gamma} (v)= \kappa_2(0,v)= H- \frac{\mu}{X(v)^2},$$ 
where $\kappa_{\Gamma}$ is the geodesic curvature of $\Gamma$ in ${\bf S}$. Since the stereographic projection $\varphi$ preserves the critical points of the geodesic curvature of regular curves (because it preserves curves of constant curvature, and hence the contact order with these curves), we deduce that $\kappa_{\gamma}(v)$ only has critical points at the values $v=k\sigma$, $k\in \Z$. 

In particular, $\gamma(v)$ has a (finite) dihedral symmetry group, as it is symmetric with respect to the reflections $T_1,\dots, T_n$. So, its isometry group is $D_{n'}$ for some $n'\geq n$, since $m/n$ is irreducible. If $n'>n$, there would exist some additional symmetry line $L'$ for $\gamma(v)$ different from all $L_k$. So, $\Gamma(v_0-v)=\Phi'(\Gamma(v_0+v))$ for some $v_0\not\in\{k\sigma : k\in \Z\}$, where $\Phi'$ is the symmetry with respect to $L'$. Thus, $\kappa_{\gamma}$ would have a critical point at $v_0$, what is a contradiction. Hence, the symmetry group of $\gamma(v)$ is $D_n$, and generated by the reflections $T_1,\dots, T_n$. Therefore, the symmetry group of $\Gamma(v)$ is isomoprhic to $D_n$, as claimed.

Finally, we show that $n\geq 2$. Indeed, if $n=1$, then all the $L_k$'s agree, and this contradicts that $(a,b,c)\in \cO^-$, since $\{\nu_0,\nu_1\}$ would be collinear.
\end{proof}

As an immediate consequence of Proposition \ref{pro:peri}, we have:

\begin{corollary}\label{cor:sigma}
Let $(a,b,c)\in \cO^-$ so that $\Theta(a,b,c)=m/n\in \Q$, where $n\in \N\setminus\{0\}$, with $m/n$ irreducible. Then, $\psi(u,v+2n\sigma)=\psi(u,v)$.
\end{corollary}

\subsection{Construction of CMC annuli}\label{sec:annu}

Following the results in Section \ref{sec:periodo}, given $u_0>0$, we can define $\Sigma_0=\Sigma_0(a,b,c,u_0)$ as the restriction of $\psi(u,v)$ to $[-u_0,u_0]\times \R$. By Corollary \ref{cor:sigma}, if $\Theta(a,b,c)=m/n\in \Q$, we can view $\Sigma_0$ as a compact $H$-annulus in $\M^3(\ep)$ under the identification $(u,v+2n\sigma)\sim (u,v)$. With this, we have:
\begin{theorem}\label{th:nivelperiodo}
Let $(a,b,c)\in \cO^-$ so that $\Theta(a,b,c)=m/n\in \Q$, where $n\in \N\setminus\{0\}$, with $m/n$ irreducible. Then, for any $u_0>0$, the following properties hold for the annulus $\Sigma_0=\Sigma_0(a,b,c,u_0)$:
\begin{enumerate}
\item
$\Sigma_0$ is symmetric with respect to ${\bf S}=\M^3(\ep)\cap \{x_3=0\}$, and with respect to $n\geq 2$ totally geodesic surfaces $\Omega_1,\dots, \Omega_n$ of $\M^3(\ep)$ that intersect equiangularly along a geodesic $\cL$ of $\M^3(\ep)$ orthogonal to ${\bf S}$.
\item 
Along each boundary component $\parc \Sigma_0^i$, $i=1,2$, $\Sigma_0$ intersects at a constant angle $\theta$ a totally umbilic surface $\mathcal{Q}_i$ of $\M^3(\ep)$. Specifically, the intersection angle $\theta$ is the same at both components, and $\mathcal{Q}_1=\Phi (\mathcal{Q}_2)$, where $\Phi$ is the symmetry of $\M^3(\ep)$ with respect to ${\bf S}$. 
\item 
Assume $u_0=\tau$, where $\tau=\tau(a,b,c)$ is defined in Proposition \ref{prosis2}. Then $\theta=\pi/2$, i.e., $\parc \Sigma_0^i$ intersects $\mathcal{Q}_i$ orthogonally.
\item 
Assume that $m_3(u_0)=0$, where $m_3$ denotes the $x_3$-coordinate of the center map $m$ in \eqref{eq:m}. Then both boundary curves $\parc \Sigma_0^1,\parc \Sigma_0^2$ lie in the same totally umbilic $2$-sphere $\mathcal{Q}_1=\mathcal{Q}_2$ of $\M^3(\ep)$.
\item 
If $a>1$, the symmetry group of $\Sigma_0$ is isomorphic to $D_n\times \Z_2$, and generated by the symmetries in item (1). In particular, $\Sigma_0$ is not rotational.
\end{enumerate}
\end{theorem}
\begin{proof}
Item (1) is a direct consequence of Proposition \ref{pro:peri}. Item (2) follows from the symmetry of $\Sigma_0$ with respect to ${\bf S}$ and the fact that the boundary curves of $\Sigma_0$ correspond to the spherical curvature lines $\psi(\pm u_0,v)$.

Item (3) follows from the definition of $\tau$ in Proposition \ref{prosis2}, together with \eqref{eq:hat} and \eqref{expab}.

Regarding item (4), we first note that $S[m(u_0),d(u_0)]$ is one of $\mathcal{Q}_1,\mathcal{Q}_2$; here, we follow the notation of Lemma \ref{lem:spherical}. If $m_3(u_0) =0$, the center of this $\mathcal{Q}_i$ lies in $\{x_3=0\}$. Since $\mathcal{Q}_1=\Phi (\mathcal{Q}_2)$ by item (2), we have $\mathcal{Q}_1=\mathcal{Q}_2= S[m(u_0),d(u_0)]$. We show next that $S[m(u_0),d(u_0)]$ is an umbilic $2$-sphere if $\ep=-1$, a property equivalent to $m(u_0)$ being timelike by the discussion before Lemma \ref{lem:spherical}. Let $\mathcal{P}$ denote the plane where $m(u)$ lies, as specified in Proposition \ref{pro:centros}. Since $(a,b,c)\in \cO^-$, $\mathcal{P}$ is timelike, and by Proposition \ref{pro:plaver} it contains ${\bf e}_3$. Since $\esiz m(u_0),{\bf e}_3\esde =0$, then $m(u_0)$ is timelike, as desired. 

To prove item (5), let $\Phi'$ denote an isometry of $\M^3(\ep)$ that leaves $\Sigma_0$ invariant. Under this isometry, we must have $\Phi' (\Gamma)=\Gamma$, since the points of $\Gamma$ represent the middle points of the (intrinsic) geodesics $\Omega_j\cap \Sigma_0$ of $\Sigma_0$. Therefore, the restriction of $\Phi'$ to ${\bf S}$ is a symmetry of $\Gamma$, and hence a composition $\mathcal{T}$ of the symmetries $\mathcal{T}_j$ with respect to some $\Omega_j$, by Proposition \ref{pro:peri}. If $\Phi'$ takes each boundary component of $\Sigma_0$ to itself, we deduce then that $\Phi'=\mathcal{T}$. Otherwise $\Phi' =\Phi\circ \mathcal{T}$, where $\Phi$ is the symmetry with respect to ${\bf S}$ (note that $\Phi$ interchanges the boundary components of $\Sigma_0$). This proves item (5).
\end{proof}

Theorem \ref{th:nivelperiodo} motivates the next definition and consequence.
 
\begin{definition}\label{hrara}
For any $(a,b,c)\in \cW$, we define $\mh: \cW \to \R$ as the map $\mh(a,b,c) := m_3(\tau(a,b,c))$, where $m_3$ denotes the $x_3$-coordinate function of the center $m(u)$. 
\end{definition}

\begin{corollary}\label{cor:hrara}
Let $(a,b,c)\in \cO^-\cap \cW$ so that $\Theta(a,b,c)=m/n\in \Q$, where $n\in \N\setminus\{0\}$. Assume that $\mh (a,b,c)=0$. Then, choosing $u_0=\tau$, both boundary components of the annulus $\Sigma_0$ in Theorem \ref{th:nivelperiodo} intersect orthogonally the same umbilic 2-sphere $\mathcal{Q}$ of $\M^3(\ep)$.
\end{corollary}

\section{Critical catenoids and nodoids in space forms}\label{sec:rotational}

In this section we introduce the family of rotational CMC surfaces in $\M^3(\varepsilon)=\S^3$ or $\H^3$ following do Carmo-Dacjzer \cite{CD}, and prove that each element within a subfamily of them has a compact piece that is an embedded free boundary rotational annulus in an adequate geodesic ball of $\M^3(\ep)$. This section can be treated independently of the rest of the paper; the proofs are postergated to an appendix.

Up to an isometry, any rotational immersion in $\M^3(1)=\S^3$ can be expressed as
\begin{equation}\label{eq:psiesferico}
\psi(\mathfrak{s},\theta) = (x (\mathfrak{s})\cos\theta,-x (\mathfrak{s})\sin\theta,\sqrt{1 - x (\mathfrak{s})^2}\sin(\phi(\mathfrak{s})),\sqrt{1 - x (\mathfrak{s})^2}\cos(\phi(\mathfrak{s}))).
\end{equation}
for some functions $x = x(\mathfrak{s})$, $\phi = \phi(\mathfrak{s})$. 

In the case of $\M^3(-1)=\H^3$, we will focus on rotational surfaces of {\em elliptic} type, that is, surfaces that are invariant by a compact, continuous 1-parameter subgroup of isometries of $\H^3$. Up to an isometry, any rotational surface of this type can be expressed as  
\begin{equation}\label{eq:psihiperbolico}
    \psi(\mathfrak{s},\theta) = (x(\mathfrak{s}) \cos \theta, -x(\mathfrak{s}) \sin \theta, \sqrt{x (\mathfrak{s})^2 + 1} \sinh(\phi(\mathfrak{s})),\sqrt{x (\mathfrak{s})^2 + 1} \cosh(\phi(\mathfrak{s})).
\end{equation}
If the immersion has CMC $H \geq 0$ and is not totally umbilic, and we choose $\mathfrak{s}$ as the arclength parameter of its profile curve, it can be shown that $x(\mathfrak{s})$ is an analytic function satisfying the following differential equation:
\begin{equation}\label{eq:hx}
        x'^2 = \frac{h(x)}{x^2}:= \frac{x^2 - \varepsilon x^4 - (Hx^2 - \delta)^2}{x^2},
    \end{equation}    
for some constant $\delta \neq 0$. In order for \eqref{eq:hx} to have solutions, it is necessary that the biquadratic polynomial $h(x)$ be non-negative for some $x \in \R$. We will treat the cases $\ep=1$ and $\ep=-1$ separately. 


\subsection{Rotational CMC surfaces in \texorpdfstring{$\S^3$}{}}
In this case, the polynomial $h(x)$ in \eqref{eq:hx} is non-negative for some $x \in \R$ if and only if 
    \begin{equation}\label{deltaS3}
        \delta \in \left[\frac{H - \mu}{2}, \frac{H + \mu}{2}\right]
    \end{equation}
where, as in the previous section, $\mu:=\sqrt{H^2+1}$. 

If $\delta = \frac{H \pm \mu}{2}$ then $h(x)$ is non-positive, with double roots at the values $x = \pm \sqrt{|\delta|/\mu}$. In this case, the only solutions $x(\mathfrak{s})$ of \eqref{eq:hx} are the constant ones, $x(\mathfrak{s}) \equiv \pm \sqrt{|\delta|/\mu}$ (up to an isometry, we can assume that $x(\mathfrak{s})\equiv \sqrt{|\delta|/\mu}$).  For $\delta \in ( \frac{H - \mu}{2}, \frac{H + \mu}{2})$, $\delta\neq 0$, $h(x)$ has four simple roots $\{-x_M,-x_m,x_m,x_M\}$, with $0 < x_m < x_M \leq 1$, and $h(x)\geq 0$ for all $x \in [-x_M,-x_m] \cup [x_m,x_M]$. In this case,  \eqref{eq:hx} has two types of analytic solutions: the constant ones, and the non-constant ones, which oscillate either on  the interval $[-x_M,-x_m]$ or on $[x_m,x_M]$. The only solutions that give rise to CMC immersions in $\S^3$ are the non-constant ones. Up to isometries, we can always assume that $x(0) = x_m$. 
  
\begin{proposition}[\cite{CD}]\label{pro:CMCS3}
    Let $\varepsilon = 1$, $H \geq 0$ and $\delta \neq 0$ such that \eqref{deltaS3} holds. If $\delta = \frac{H \pm \mu}{2}$,  let $x(\mathfrak{s}):= \sqrt{|\delta|/\mu}$, constant.
    Otherwise, let $x(\mathfrak{s})$ be the unique non-constant solution of \eqref{eq:hx} with initial condition $x(0) = x_m$. 
    We define
    \begin{equation}\label{eq:phi}
        \phi(\mathfrak{s}):= \int_{0}^{\mathfrak{s}} \frac{\delta - Hx^2}{x(1 - x^2)}ds.
    \end{equation}
   Then, denoting $\S^1\equiv \R/(2\pi \Z)$, the immersion $\psi:\R\times \S^1\flecha \S^3$ given by \eqref{eq:psiesferico} defines a rotational surface in $\S^3$ with constant mean curvature $H\geq 0$ such that $\langle\psi_\mathfrak{s},\psi_\mathfrak{s}\rangle \equiv 1$. 
    The $\mathfrak{s}$-curves and $\theta$-curves are curvature lines, with respective associated principal curvatures
\begin{equation}\label{eq:kappa}
    \kappa_{\mathfrak{s}} = H + \delta/x^2, \; \; \; \; \kappa_\theta = H - \delta/x^2.
\end{equation}
Conversely, any rotational CMC surface in $\S^3$ must be an open piece of either one of these examples, or of a totally umbilical round sphere. 
\end{proposition}

\begin{definition}[Spherical nodoids, unduloids and catenoids]\label{def:rotacionalesS3}
    Let $\ep=1$. For any $H\geq 0$ and $\delta\neq 0$ such that \eqref{deltaS3} holds, let $\mathcal{S}= \mathcal{S}(\varepsilon,H,\delta)$ be the rotational CMC surface in $\S^3$ of Proposition \ref{pro:CMCS3}.  
 \begin{itemize}
     \item If $H>0$ and $0<\delta <(H + \mu)/2$  we will say that $\mathcal{S}$ is a \emph{spherical nodoid}. 
     \item If $H>0$ and $(H-\mu)/2 < \delta < 0$ we will say that $\mathcal{S}$ is a \emph{spherical unduloid}.  
     \item If $H = 0$ we will say that $\mathcal{S}$ is a  \emph{spherical catenoid}. Since the change $\delta\mapsto -\delta$ just gives a reparameterization of $\mathcal{S}$, we can assume in this case that $\delta>0$. 
    \item If $\delta =(H \pm \mu)/2$, the surface covers a flat torus. For $H=0$, it is a Clifford torus.
 \end{itemize}   
\end{definition}

 \subsection{Rotational CMC surfaces in \texorpdfstring{$\H^3$}{} of elliptic type} 
We recall that, in order for \eqref{eq:hx} to have solutions, it is necessary that $h(x)$ be non-negative for some $x \in \R$. This will happen in any of the following situations:
\begin{enumerate}
    \item $H < 1$,
    \item $H = 1$ and $\delta > -\frac{1}{2}$,
    \item $H > 1$ and $\delta \geq \frac{\mu-H}{2}$, where $\mu:=\sqrt{H^2-1}$.
\end{enumerate}

In the first two cases, $h(x)$ only has two roots $-x_m < 0 < x_m$, and $h(x)\geq 0$ for all $x \in (-\infty,-x_m]\cup [x_m,\infty)$. In the third case, $h(x)$ has four roots $-x_M\leq -x_m < 0 < x_m \leq x_M$, and $x_m=x_M$ if and only if $\delta = \frac{\mu-H}{2}$. In this situation, $h(x) \geq 0$ on $[-x_M,-x_m] \cup [x_m, x_M]$. 

If $H > 1$ and $\delta = \frac{\mu-H}{2}$, then $x_m=x_M=\sqrt{\frac{H - \mu}{2}}$ and the only analytic solutions of \eqref{eq:hx} are the constants $x = \pm x_m$. The resulting surfaces are flat hyperbolic cylinders in $\H^3$.

In the rest of cases, \eqref{eq:hx} has two types of analytic solutions: the constants given by the roots of $h(x)$, and non-constant solutions. The only ones that give rise to actual CMC immersions are those which are not constant. 
More specifically, if $H > 1$, then $x(\mathfrak{s}):\R\flecha \R$ oscillates on either $[-x_M,-x_m]$ or $[x_m, x_M]$. If $H \leq 1$, however, $x(\mathfrak{s}):\R\flecha \R$ is unbounded, taking values on either $(-\infty,-x_m]$ or $[x_m,\infty)$.

In all three cases, up to isometries in $\H^3$, we can always assume that $x(0) = x_m$, where $x_m$ is the smallest positive root of $h(x)$. We have then:

\begin{proposition}[\cite{CD}]\label{pro:CMCH3}
    Let $\varepsilon = -1$, $H  \geq 0$, $\delta \neq 0$. If $H > 1$ and $\delta = \frac{\mu-H}{2}$, let $x(\mathfrak{s})\equiv x_m$. Otherwise, let $x(\mathfrak{s})$ be the unique nonconstant solution of \eqref{eq:hx} with initial condition $x(0) = x_m$. We define
    \begin{equation}\label{eq:phihiperbolico}
        \phi(\mathfrak{s}):= \int_0^\mathfrak{s}\frac{\delta - Hx^2}{x(x^2 + 1)}ds.
    \end{equation}
    Under these conditions, the immersion $\psi:\R\times \S^1\flecha \H^3$ in \eqref{eq:psihiperbolico} is a rotational surface in $\H^3$ with constant mean curvature $H \geq 0$ and $\langle\psi_\mathfrak{s},\psi_\mathfrak{s}\rangle \equiv 1$, with principal curvatures given by \eqref{eq:kappa}. 

 Conversely, any rotational CMC surface of elliptic type in $\H^3$ must be an open piece of either one of these examples, or of a totally umbilical surface of $\H^3$.     
\end{proposition}

\begin{definition}\label{def:rotacionalesH3}
    We will say that the immersion $\psi$ in Proposition \ref{pro:CMCH3} is a \emph{hyperbolic nodoid} (resp. \emph{unduloid}) if $H > 1$ and $\delta > 0$ (resp. $\frac{\mu - H}{2} <\delta < 0$), and denote it by $\mathcal{S}=\mathcal{S}(\ep,H,\delta)$. 
    \end{definition}


\subsection{Existence of free boundary nodoids and catenoids}

Our goal in this section is to show that the nodoids and catenoids given in Propositions \ref{pro:CMCS3}, \ref{pro:CMCH3} for $\delta>0$ (see also  Definitions \ref{def:rotacionalesS3} and \ref{def:rotacionalesH3}) are free boundary in a certain ball of $\M^3(\varepsilon)$. The geodesic balls of $\M^3(\ep)$ will be described as
\begin{equation}\label{eq:B}
B[m,d] := \{x \in \M^3(\varepsilon) \; : \; \langle x, m \rangle \geq d\}.
\end{equation}
Here $m\in \M^3(\ep)$ is the \emph{center} of the ball, while $d\in \R$ satisfies
$|d| < 1$ when $\varepsilon = 1$ and $d < -1$ when $\varepsilon = -1$.

The fact that $x(0) = x_m$ and $\phi(0) = 0$ along with \eqref{eq:hx}, \eqref{eq:phi}, \eqref{eq:phihiperbolico} imply that the function $x(\mathfrak{s})$ is symmetric, while $\phi(\mathfrak{s})$ is antisymmetric. A geometric consequence of this is that the immersions $\psi(\mathfrak{s},\theta)$ given in \eqref{eq:psiesferico}, \eqref{eq:psihiperbolico} are symmetric with respect to the totally geodesic surface ${\bf S} = \{x_3 = 0\} \cap \M^3(\varepsilon)$. Moreover, the rotation axis of these examples is given by the geodesic $\mathcal{L}:= \{x_1 = x_2 = 0\} \cap \M^3(\varepsilon)$.  We note that the balls $B[{\bf e}_4,d]$ centered at the point ${\bf e}_4\in\M^3(\ep)$ are also symmetric with respect to ${\bf S}$ and invariant under rotations with axis $\mathcal{L}$.

Given $\mathfrak{s}_0 > 0$, we define $\mathcal{S}_0$ as the compact annulus $\psi([-\mathfrak{s}_0,\mathfrak{s}_0]\times \S^1) \subset \M^3(\varepsilon)$, where we have identified the points $(\mathfrak{s},\theta + 2\pi) \sim (\mathfrak{s},\theta)$ in the obvious way. Consider the profile curve of $\mathcal{S}_0$,
\begin{equation}\label{pro:s}
\mathfrak{s} \mapsto \psi(\mathfrak{s},0) = (x(\mathfrak{s}),0,x_3(\mathfrak{s}),x_4(\mathfrak{s})).
\end{equation}
We are interested in studying whether $\mathcal{S}_0$ is free boundary in a certain ball. By the symmetries of $\mathcal{S}_0$, it can be shown that its two boundary components are contained in the totally umbilical sphere $$S[{\bf e}_4,\varepsilon x_4(\mathfrak{s}_0)]=\parc B[{\bf e}_4,\varepsilon x_4(\mathfrak{s}_0)] \subset \M^3(\ep).$$
So, if $\mathcal{S}_0$ happened to be free boundary in some geodesic ball of $\M^3(\ep)$, it would necessarily be in $B[{\bf e}_4,\varepsilon x_4(\mathfrak{s}_0)]$.


\begin{proposition}\label{pro:hatp}
Let $\varepsilon \in \{-1,1\}$, $H \geq 0$, $\delta > 0$. If $\varepsilon = 1$, assume further that $\delta < \frac{H + \mu}{2}$. For any $\mathfrak{s}$, denote by $\Gamma_{\mathfrak{s}}$ the geodesic in $\M^3(\ep)$ with initial conditions $\psi(\mathfrak{s},0)$ and $\psi_{\mathfrak{s}}(\mathfrak{s},0)$. 
\begin{enumerate}
    \item There exists $\tilde{\mathfrak{s}}=\tilde{\mathfrak{s}}(H,\delta)>0$ such that the compact annulus $\mathcal{S}_0:= \psi([-\tilde{\mathfrak{s}},\tilde{\mathfrak{s}}]\times \S^1)$ is embedded and free boundary in the ball $B:= B[{\bf e}_4,\varepsilon x_4(\tilde{\mathfrak{s}})] \subset \M^3(\varepsilon)$, with $x_4(\tilde{\mathfrak{s}}) > 0$. Moreover, the principal curvature associated to the profile curve is strictly decreasing on a certain interval $[0,\tilde{\mathfrak{s}} + \epsilon)$, $\epsilon > 0$.
    \item  The map $(H,\delta)\mapsto \tilde{\mathfrak{s}}(H,\delta)$ is analytic. Moreover, if $\varepsilon = 1$, $\tilde{\mathfrak{s}}$ extends continuously to the boundary curve $\delta= \frac{H + \mu}{2}$ by  \begin{equation}\label{eq:SHd}
       \tilde{\mathfrak{s}} \left(H,\frac{H + \mu}{2}\right)=\frac{\pi}{2\sqrt{2\mu(H + \mu)}}. 
    \end{equation}
   
    \item On a neighbourhood $\mathcal{I}\subset \R$ of $\tilde{\mathfrak{s}}=\tilde{\mathfrak{s}}(H,\delta)$, the rotational axis $\mathcal{L}$ of $\mathcal{S}_0$ and the geodesic $\Gamma_{\mathfrak{s}}$ meet at a unique point $\hat p(\mathfrak{s})$ with $\hat p_4(\mathfrak{s}) > 0$. Moreover, the function $\hat p: \mathcal{I} \to \mathcal{L}$ is analytic with  $\hat{p_3}(\tilde{\mathfrak{s}}) = 0$ and $\hat{p_3}'(\tilde{\mathfrak{s}}) > 0$ (here $\hat p_3(\mathfrak{s}),\hat p_4(\mathfrak{s})$ denote the third and fourth coordinates of $\hat p(\mathfrak{s})$ respectively).
\end{enumerate}

\end{proposition}
\begin{proof}
    See Appendix \ref{sec:appendix}.
\end{proof}

\begin{remark}\label{tori}
    The first item of the previous Proposition omits the limit case $\delta = \frac{H + \mu}{2}$, which corresponds to rotational flat tori. However, in this case the functions $x(\mathfrak{s})$, $x_3(\mathfrak{s})$ can be computed explicitly: 
\begin{equation}\label{toro}
    x(\mathfrak{s}) \equiv \sqrt{\frac{\mu + H}{2\mu}}, \; \; \; \; x_3(\mathfrak{s}) = \sqrt{\frac{\mu - H}{2\mu}}\sin\left(\sqrt{2\mu(\mu + H)}\mathfrak{s}\right).
    \end{equation}
    With this, one can check that the embedded annulus $\psi([-\tilde{\mathfrak{s}},\tilde{\mathfrak{s}}]\times \S^1)$, where $\tilde{\mathfrak{s}} = \frac{\pi}{2\sqrt{2\mu(\mu + H)}}$, is free boundary in $B[{\bf e}_4,0]$.
\end{remark}

\section{Delaunay surfaces via double roots}\label{sec:double}

We now consider the case $a=1$ in our discussion of Section \ref{sec:romega}, i.e., we choose $(1,b,c)\in \cO$ and let $\Sigma=\Sigma(1,b,c)$ be the CMC surface in $\M^3(\ep)$ of Definition \ref{def:sabc}. This means that $p(0,x)$ in \eqref{def:pab} has a double root at $x=1/c$. As explained after \eqref{eq:Xv0}, in this case we have $e^{\rho(0,v)}=1/c$. Therefore, $\rho_v(0,v)\equiv 0$ and by uniqueness of the solution to the Cauchy problem for \eqref{sinhg} we have that $\rho=\rho(u)$, i.e., $\rho_v \equiv 0$. It follows then by \eqref{unodos} that the coefficients of $I,II$ for $\Sigma$ only depend on $u$, and so $\Sigma$ is invariant under a $1$-parameter group of ambient isometries of $\M^3(\ep)$. Moreover, each curve $v\mapsto \psi(u,v)$ is an orbit of such $1$-parameter group. 

In our situation, $\Gamma(v):=\psi(0,v)$ is such an orbit, which actually lies in the totally geodesic surface ${\bf S}:=\M^3(\ep)\cap \{x_3=0\}$ of $\M^3(\ep)$, see Lemma \ref{plaor}. Since $\Sigma$ intersects ${\bf S}$ orthogonally along $\Gamma(v)$, the geodesic curvature of $\Gamma(v)$ as a curve in ${\bf S}$ is given by the principal curvature $\kappa_2(0,v)$, which by \eqref{princur} is constant, and equal to $H-\mu c^2$.

Therefore, in the case $\ep=1$, $\Sigma$ is a rotational CMC surface in $\S^3$. The $u$-curves (resp. $v$-curves) of $\Sigma$ correspond to the $\mathfrak{s}$-curves (resp $\theta$-curves) in the parametrization $\psi(\mathfrak{s},\theta)$ of rotational surfaces of Section \ref{sec:rotational}. In this way, $\psi(u,0)$ is a profile curve of $\Sigma$.

The principal curvature $\kappa_1(u)$ associated to this profile curve $\psi(u,0)$ is always positive, by \eqref{princur}. Thus, if $H\neq 0$, $\Sigma$ is the universal cover of either a flat torus or a \emph{spherical nodoid} of $\S^3$. If $H=0$, $\Sigma$ covers a spherical catenoid or a Clifford torus. See Section \ref{sec:rotational}.

\begin{remark}\label{rem:tori}
If $a=c=1$, then $y'(0)=0$ by \eqref{pqabc}, and so $y(u)\equiv 0$. This implies by \eqref{omini} and \eqref{omu} that $\rho(u)\equiv 0$. Thus, $a=c=1$ corresponds to the case where $\Sigma$ covers a flat CMC torus in $\S^3$.
\end{remark}

In the case $\ep=-1$, we have that $\Sigma$ is a \emph{generalized} rotational surface in $\H^3$, i.e., it is invariant by either hyperbolic, elliptic or parabolic rotations in $\H^3$. We will be interested in the elliptic case, i.e. the case where the orbits of these rotations are (compact) circles. This happens if and only if the geodesic curvature of the orbits is greater than $1$ in absolute value, that is, if and only if
\begin{equation}\label{conde}
(H-\mu c^2)^2 >1.
\end{equation}
Thus, if \eqref{conde} holds for our choice of $(H,c)$, then $\Sigma$ will be a Delaunay surface in $\H^3$ with constant mean curvature $H>1$. Again, since $\kappa_1$ is positive by \eqref{princur} and describes the geodesic curvature of the profile curve of $\Sigma$, we deduce that $\Sigma$ is a \emph{hyperbolic nodoid}. See Section \ref{sec:rotational}. An equivalent form of \eqref{conde} is 
\begin{equation}\label{conde2}
H<\frac{\mu}{2}\left(c^2+\frac{1}{c^2}\right),
\end{equation}
where we have used that $\mu^2=H^2-1$ if $\ep=-1$.
This motivates the following definition.

\begin{definition}\label{der}
We let $\cR\subset \R^3$ be open subset of $\cO\cap \{a=1\}$ given by $$\cR:=\{(1,b,c)\in \cO : (H-\mu c^2)^2+\ep >0\}.$$
\end{definition}

Note that $\cR$ is just $\cO\cap \{a=1\}$ if $\ep=1$, and that the inequality defining $\cR$ is \eqref{conde} if $\ep=-1$. We thus have from the discussion above:

\begin{proposition}\label{pro:nod}
Assume that $(1,b,c)\in \cR$. 
Then, $\Sigma=\Sigma(1,b,c)$ is the universal cover of the (spherical or hyperbolic) nodoid ($H\neq 0$) or catenoid ($H=0$) in $\M^3(\ep)$ with neck curvature given by $\kappa=H-\mu c^2$.
\end{proposition}

\begin{remark}\label{rem:ejea1}
If $(1,b,c)\in \cR\cap \cO^-$, then the rotation axis of $\Sigma(1,b,c)$ is the geodesic $\cL$ described in item (1) of Theorem \ref{th:nivelperiodo}. This follows from the fact that, in this case, $\Gamma(v)=\psi(0,v)$ is a (compact) circle, and any of the symmetries in that item (1) must leave the center of $\Gamma(v)$ fixed, i.e. the center lies in the intersection of the symmetry planes $\Omega_k$. Since both the rotation axis and $\Omega_k$ are orthogonal to ${\bf S}$, we conclude from there that the rotation axis agrees with $\cL=\cap_{k\in \Z} \Omega_k$.
\end{remark}

If follows from the above construction that the parameter $c$ determines uniquely the immersion $\psi(u,v)$ that defines $\Sigma(1,b,c)$. On the other hand, the role of the parameter $b$ is to determine the initial values \eqref{ivasis2} of system \eqref{system1} via \eqref{aabc}, \eqref{pqabc}. More conceptually, since $\psi(u,v)$ is rotational, each curve $v\mapsto \psi(u,v)$ is a circle that can be seen as contained in infinitely many $2$-dimensional totally umbilic surfaces $S[m(u),d(u)]$ of $\M^3(\ep)$. The choice of $b$ in $\Sigma(1,b,c)$ determines the values of $m(u), d(u)$ in this description.

We will next describe the behavior of the solutions to system \eqref{system1} associated to our solution $\rho(u,v)$ in our current case $a=1$. So, we
consider the initial conditions \eqref{yzz}, \eqref{aabc}, \eqref{pqabc} for system \eqref{system1}, that give together with \eqref{omini} the solution $\rho(u,v)$ to \eqref{sinhg} starting from $(1,b,c)\in \cO^-$ constructed in Theorem \ref{th:otro}.

The Hamiltonian constants $h,k$ in \eqref{def:h} and \eqref{def:k} can be expressed in terms of $(b,c)$ using $a=1$, \eqref{aabc} and \eqref{pqabc}. This lets us write $q(x)$ in \eqref{def:q} in terms of $(b,c)$ as 

\begin{equation}\label{qfactor}
q(x)=-(x-r_1)^2(x-r_3),
\end{equation}
where 
\begin{equation}\label{roots12}
r_1=-\frac{(b - 1)^2}{4b} \leq 0
\end{equation}
and
\begin{equation}\label{root3}
r_3= \frac{1}{4}\left(c + \frac{1}{c}\right)^2 \geq 1.
\end{equation}

Let $(s(\landa),t(\landa))$ be the solution to \eqref{system2} obtained after the change of coordinates \eqref{change} from our initial solution $(y(u),z(u))$ to \eqref{system1}. We list the following properties, that were obtained in \cite{CFM}, and that will be used later on:

\begin{enumerate}
\item[i)]
$(s(\landa),t(\landa))$ is defined for all $\landa \in \R$, and up to a translation in the $\landa$ parameter it satisfies the initial conditions $s(0)=1$, $t(0)=0$.
\item[ii)]
$s(\landa)$ takes values in $[1,r_3]$, while $t(\landa)$ takes values in $[r_1,0]$. Moreover, $s(\landa)\equiv 1$ if and only if $r_3=1$ and $t(\landa)\equiv 0$ if and only if $r_1=0$.
\item[iii)]
$s(\landa)$ is $2l$-periodic, where
\begin{equation*}
l=\int_1^{r_3} \frac{dx}{\sqrt{x (x-1) q(x)}}<\8.
\end{equation*}
In this way, $s(2l)=s(0)=1$ and $s(l)=r_3$. 
\item[iv)]
If $r_1<0$, then $t(\landa)$ is strictly decreasing, with $t(\landa)\to r_1$ as $\landa\to \8$. 
\end{enumerate}

\section{The period map for nodoids and catenoids}\label{sec:perrot}

In this section we will assume, as in Section \ref{sec:double}, that $a=1$, and keep the same notations. Therefore, $\rho=\rho(u)$, and $\Sigma$ is the rotational example of Proposition \ref{pro:nod}. In particular $\kappa_2 (0,v)= H - \mu c^2$. 

As explained in Section \ref{sec:double}, if $\ep=1$, $\Gamma(v)$ is a circle in $\S^3\cap \{x_3=0\}$. Also, if $\ep=-1$, $\Gamma(v)$ is a curve in $\H^3\cap \{x_3=0\}\equiv \H^2$ of constant curvature $H-\mu c^2$. This curve will be a (compact) circle if and only if \eqref{conde} holds, i.e., if and only if $(1,b,c)\in \cR$ (see Definition \ref{der}). We also recall that the set $\mathcal{O}^-$ and the period map $\Theta$ were introduced in Definition \ref{omenos} and Definition \ref{def:period} respectively. 

We prove next:

\begin{theorem}\label{perota}
Let $(1,b,c)\in \cR\cap \cO^-$. Then, 
\begin{equation}\label{eq:perota}
\Theta(1,b,c)=\frac{-  \sqrt{(H-c^2 \mu)^2+\ep}}{c\mu  \sqrt{\left(\frac{1}{c}+b c\right)\left(\frac{1}{c}+\frac{c}{b}\right)}}.
\end{equation}
\end{theorem}
\begin{proof}
Consider the planar curve $\gamma(v)=\varphi(\Gamma(v))$ defined above \eqref{def:per0}. The metric on ${\bf S}\setminus \{-{\bf e}_4\}$ can be written via the inverse stereographic map $\varphi^{-1}$ as 
\begin{equation}\label{cofes}
\varrho (dx^2+dy^2), \hspace{0.5cm} \varrho:=\frac{4}{(1+\ep (x^2+y^2))^2},
\end{equation} 
where $(x,y)$ are Euclidean coordinates in $\R^2$. From here and \eqref{unodos}, we have 
\begin{equation}\label{metper}
||\gamma'|| = \frac{X}{2\mu \sqrt{\varrho}},
\end{equation}
where we have used our usual notation $X(v)=e^{\rho(0,v)}$. On the other hand, by standard formulas of conformally related Riemannian metrics, the geodesic curvature $\kappa_{\Gamma}$ of $\Gamma(v)$ in {\bf S}, i.e., the geodesic curvature of $\gamma(v)$ with respect to the metric \eqref{cofes}, is related to the Euclidean geodesic curvature $\kappa_{\gamma}$ of $\gamma$ by 
\begin{equation}\label{forcurv}
\kappa_{\gamma}=\frac{\esiz \nabla \sqrt{\varrho},{\bf n}\esde}{\sqrt{\varrho}}+ \sqrt{\varrho} \kappa_{\Gamma},
\end{equation} 
where ${\bf n}$ is the unit normal of $\gamma$ in $\R^2$, and both $\nabla, \esiz,\esde$ are Euclidean. Recall that $\kappa_{\Gamma}= \kappa_2(0,v)= H -\mu c^2<0$. Since $\Gamma(v)$ is a horizontal circle, its stereographic projection $\gamma(v)$ parametrizes a circle of a certain radius $r>0$ in the plane (or in the unit disk of $\R^2$, if $\ep=-1$), which is negatively oriented since $\kappa_{\Gamma}<0$. Along $\gamma(v)$ we have 
\begin{equation}\label{alga}
\sqrt{\varrho}=\frac{2}{1+\ep r^2}, \hspace{0.5cm} \nabla \sqrt{\varrho} = \frac{-4\ep}{(1+\ep r^2)^2} \gamma(v),
\end{equation}
and ${\bf n}= \frac{1}{r} \gamma(v)$. Then, it follows from \eqref{forcurv}, \eqref{alga} and $\kappa_{\gamma}=-1/r$ that $$r =\ep\left(H-c^2 \mu \pm \sqrt{(H-c^2\mu)^2+\ep}\right).$$ Here, we should recall that \eqref{conde} holds when $\ep=-1$, since $(1,b,c)\in \cR$. Moreover, taking into account that $r>0$ if $\ep=1$ and $r\in (0,1)$ if $\ep =-1$, we deduce that, actually,
\begin{equation}\label{radiosin}
r =\ep\left(H-c^2 \mu + \sqrt{(H-c^2\mu)^2+\ep}\right).
\end{equation}

We now compute the value of $\Theta=\Theta(1,b,c)$ using \eqref{def:per0}. First note that, by \eqref{metper}, \eqref{alga} and $X(v)=1/c$, we have $$||\gamma'||=\frac{1+\ep r^2}{4c\mu}.$$ Thus using that $\kappa_{\gamma}=-1/r$ and \eqref{exsi}, we have from \eqref{def:per0}
$$\Theta(1,b,c)=\frac{-(1+\ep r^2)}{2c \mu r  \sqrt{\left(\frac{1}{c}+b c\right)\left(\frac{1}{c}+\frac{c}{b}\right)}}.$$ Using \eqref{radiosin} in this expression, we obtain \eqref{eq:perota}.
\end{proof}

\begin{remark}\label{grafocab}
    For any $p_0=(1,b,c) \in \cR\cap \cO^-$, $c > 1$, a straightforward computation from \eqref{eq:perota} shows that $\Theta_c(p_0) \neq 0$. By the implicit function theorem and the analyticity of $\Theta(a,b,c)$ (Proposition \ref{antet}), it follows that the level set $\Theta(a,b,c)=\Theta_0$, where $\Theta_0:=\Theta(p_0)$ can be locally expressed around $p_0$ as a graph $c = c^{\Theta_0}(a,b)$, where $c$ is real analytic with respect to $a,b$. 
    
   Similarly, if $p_0 \in \cR \cap \cO^-$, $b > 1$, an analogous computation shows that $\Theta_b(p_0) \neq 0$, and so in this case we can express locally the level set $\Theta(a,b,c) = \Theta(p_0)$ as an analytic graph 
   $b = b^{\Theta_0}(a,c)$.
\end{remark}

We now consider $(r_1,r_3)$ given by \eqref{roots12}, \eqref{root3}. The map $(b,c)\mapsto (r_1,r_3)$ is a homemorphism from $\{(b,c): b\geq 1, c\geq 1\}$ onto $\{(r_1,r_3): r_1\leq 0, r_3\geq 1\}$, and so we can view $\Theta(1,b,c)$ as a map $\Theta(r_1,r_3)$. Using \eqref{roots12}, \eqref{root3} and $\mu^2=H^2+\ep$, the expression for $\Theta=\Theta(r_1,r_3)$ in \eqref{eq:perota} simplifies to
\begin{equation}\label{tepe}
\Theta^2 = \frac{\mu (2r_3-1)-H}{2\mu (r_3-r_1)}.
\end{equation}
This can be rewritten alternatively as 
\begin{equation}\label{curvanivel}
    r_3= \frac{\Theta^2}{\Theta^2-1} r_1 + \frac{H + \mu}{2\mu (1-\Theta^2)}
\end{equation}
For any fixed $\Theta=\Theta_0$, this equation represents the line with slope $\Theta_0^2/(\Theta_0^2-1)$ that passes through the point
\begin{equation}\label{def:p0}
p_0:=\left(\frac{H+\mu}{2\mu},\frac{H+\mu}{2\mu}\right).
\end{equation}
As an immediate consequence of  \eqref{eq:perota} and \eqref{tepe}, we have:
\begin{corollary}\label{cor:perota}
$\Theta(1,b,c)\in (-1,0)$ for any $(1,b,c)\in \cR\cap \cO^-$.
\end{corollary}

We next show that in Theorem \ref{perota} we can simply assume $(1,b,c)\in \cR$.
\begin{proposition}\label{rco}
$\mathcal{R}\subset \cO^-$.
\end{proposition}
\begin{proof}
Let $(1,b,c)\in \cR$. Then, $\Sigma(1,b,c)$ is a rotational surface in $\M^3(\ep)$ whose rotation axis is orthogonal to ${\bf e}_2$. 
After a linear isometry of $\R_{\ep}^4$ that fixes ${\bf e}_2, {\bf e}_3$, we can assume that this rotation axis is $\M^3(\ep)\cap \{x_1=x_2=0\}$. We now consider, as we did in Section \ref{sec:periodo}, the stereographic projection $\varphi$  of $\M^3(\ep)$ from $-{\bf e}_4$ to $\R^3$, and the planar curve $\gamma(v):=\varphi (\Gamma(v))$. 
Define next the number 
$$\hat{\theta}:=\frac{1}{\pi}\int_0^{\sigma} \kappa_{\gamma} ||\gamma'|| dv,
$$ just as in \eqref{def:per0}. We use here the new notation $\hat{\theta}$, since $\Theta$ in \eqref{def:per0} was only defined on $\cO^-$; nonetheless, the right hand side of \eqref{def:per0} makes sense in our case. Moreover, the computations in Theorem \ref{perota} show that $\hat{\theta}$ is also given by the right hand side of \eqref{eq:perota}.

In order to show that $(1,b,c)\in \cO^-$ we must prove first of all that $\{\nu_0,\nu_1\}$ are linearly independent. If $\nu_1$ were proportional to $\nu_0$, then $\hat{\theta}$ would be an integer, since $\pi \hat{\theta}$ measures the variation of the unit normal of $\gamma(v)$ along $[0,\sigma]$, and the unit normals at $v=0$ and $v=\sigma$ are collinear (since $\nu_0,\nu_1$ are). But on the other hand, the computations in Corollary \ref{cor:perota} using that $\hat{\theta}$ is given by \eqref{eq:perota} show that $\hat{\theta}\in (-1,0)$.

This shows that $(1,b,c)\in \cO^-$ if $\ep=1$. In the case $\ep=-1$, we also need to check that the plane $\Lambda:={\rm span}\{\nu_0,\nu_1\}$ is spacelike.  This condition holds because $\Lambda$ is orthogonal to the rotation axis $\M^3(\ep)\cap \{x_1=x_2=0\}$ of $\Sigma(1,b,c)$. 
\end{proof}

\section{Detecting critical nodoids and catenoids}\label{detectcritical}

Throughout this section we will consider points of the form $(1,b,c)\in \cR\cap \cW$, and let $\Sigma=\Sigma(1,b,c)$ be the associated rotational $H$-surface in $\M^3(\ep)$, see Proposition \ref{pro:nod}. Using \eqref{roots12}, \eqref{root3} we can actually view $\Sigma=\Sigma(r_1,r_3)$ as depending on $r_1,r_3$. We can then reparametrize the parameter domain $\cR\cap \cW$ in terms of $(r_1,r_3)$ as 
\begin{equation}\label{ogra}
\cR\cap \cW \equiv \hat{\cW} =\left\{(r_1,r_3) {\in \R^2} : r_1\leq 0, r_3 > {\rm max} \left\{\frac{(r_1-1)^2}{1-2r_1},-\ep \frac{H+\mu}{2\mu}\right\} \right\}.
\end{equation}
This follows directly from a computation using \eqref{roots12}, \eqref{root3}, the definition of $\cW$
 in \eqref{defi:W}, Proposition \ref{rco} and \eqref{conde2}. Note that if $\ep=1$, the maximum in \eqref{ogra} is always the first quantity (the other one being negative in that case); see Figure \ref{fig:hW}.

\begin{figure}[h]
  \centering
  \includegraphics[width=0.45\textwidth]{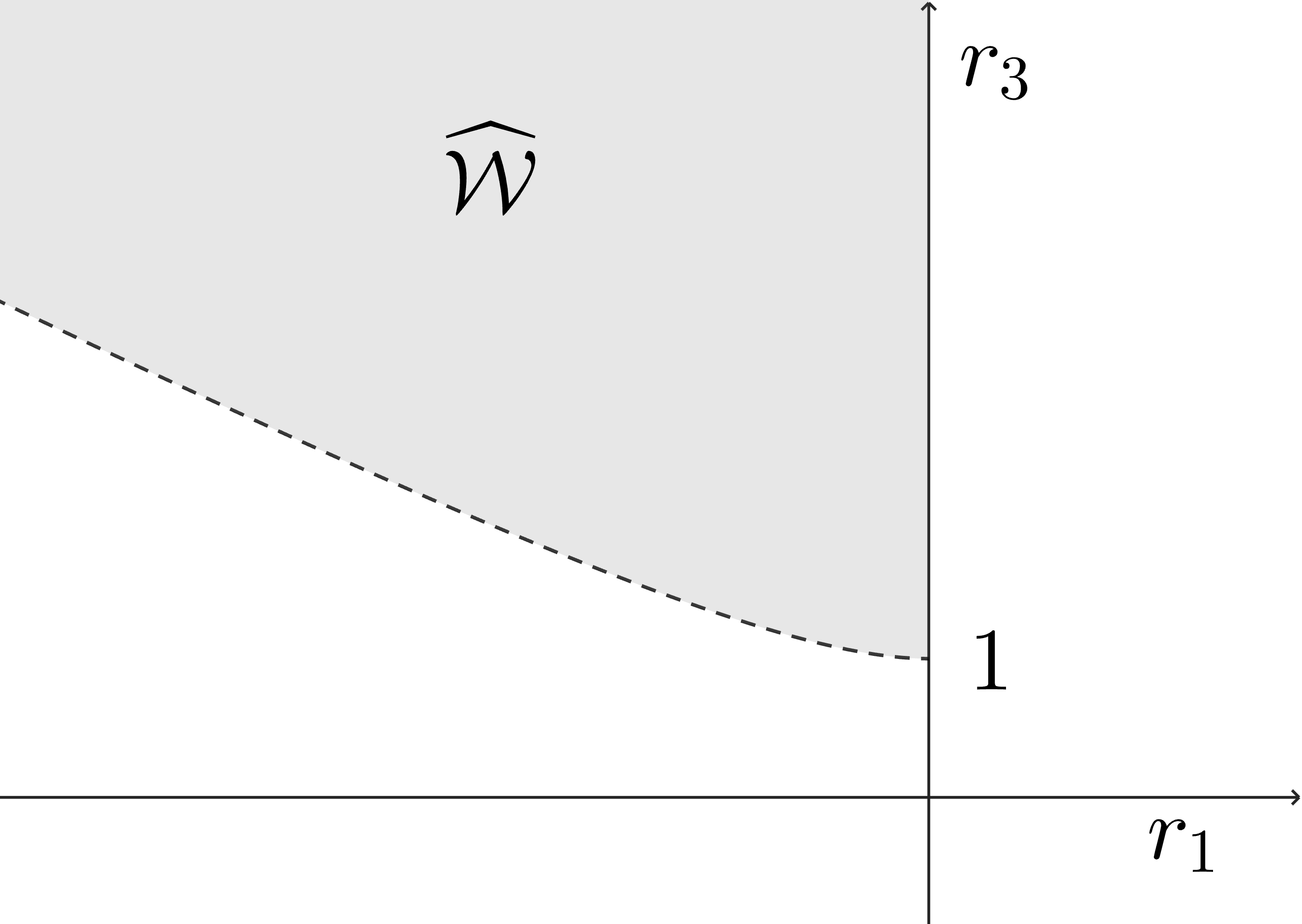}
  \hfill
  \includegraphics[width=0.45\textwidth]{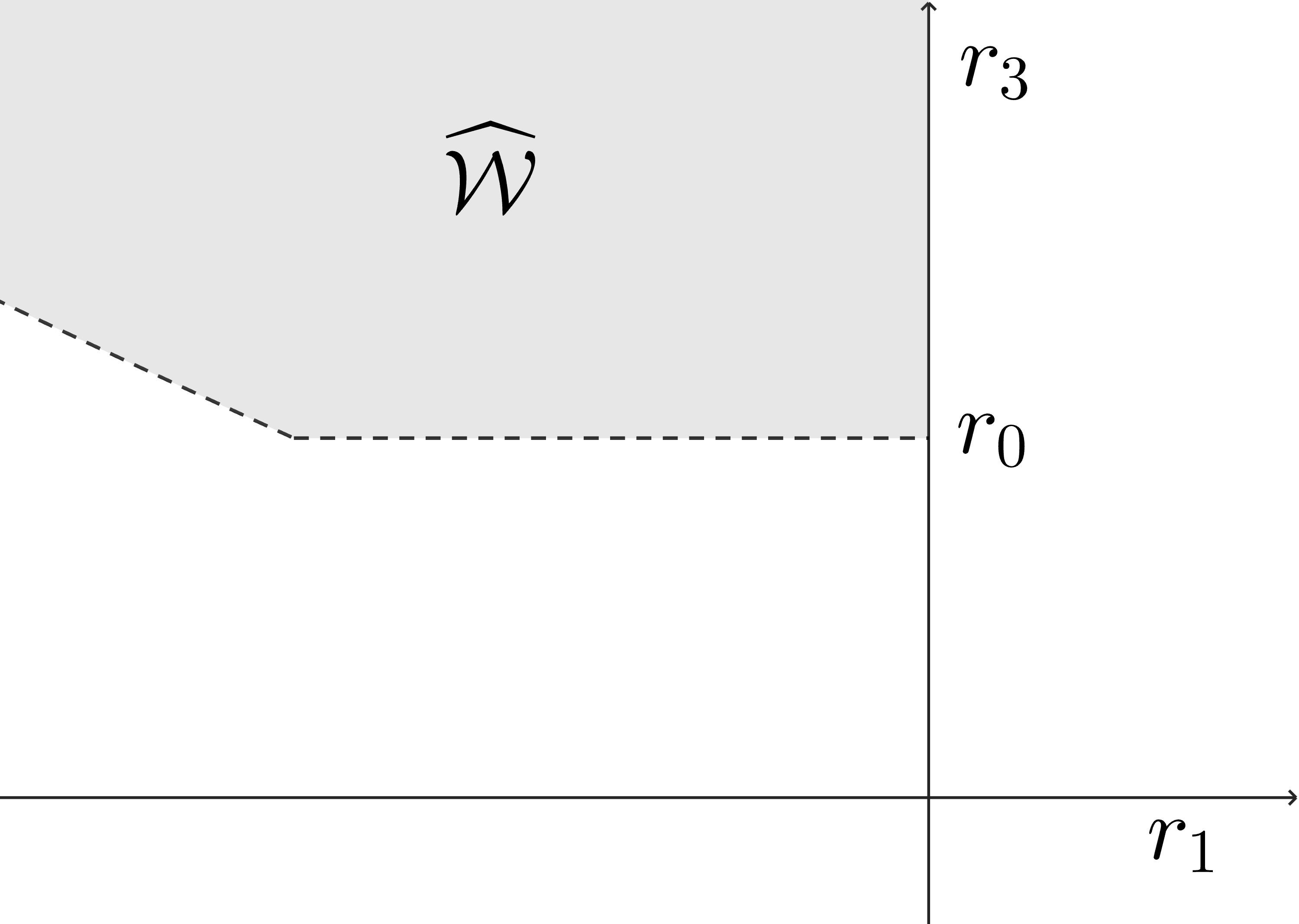}
  \caption{The parameter domain $\hat{\mathcal{W}}$ in \eqref{ogra} for the cases $\varepsilon = 1$ (left) and $\varepsilon = -1$ (right). In the right picture, the value $r_0$ stands for $r_0 = \frac{H + \mu}{2\mu}$.}\label{fig:hW}
\end{figure}

Similarly, we can write the functions $\tau(1,b,c)$ in Proposition \ref{prosis2} and $\mathfrak{h}(1,b,c)$ in Definition \ref{hrara} as $\tau(r_1,r_3)$ and $\mathfrak{h}(r_1,r_3)$, respectively. Both of them are naturally defined in this way on $\hat{\cW}$, and are real analytic.  

By Proposition \ref{pro:nod}, $\Sigma(r_1,r_3)$ corresponds to one of the catenoids or nodoids $\psi(\mathfrak{s},\theta)$ in $\M^3(\ep)$ studied in Section \ref{sec:rotational}, whose profile curves were parametrized by arc-length, that is, $\|\psi_{\mathfrak{s}}\|^2 \equiv 1$. If $\delta>0$ denotes the parameter in Propositions \ref{pro:CMCS3} and \ref{pro:CMCH3} that defines $\psi(\mathfrak{s},\theta)$, then it follows that the map $r_3\mapsto \delta(r_3)$ is real analytic and injective; here we recall that $r_3$ determines $\Sigma(r_1,r_3)$ uniquely, as explained after Remark \ref{rem:ejea1}.  We also emphasize that if $\ep=1$, the value $r_3 = 1$ corresponds to the case where $\Sigma(r_1,1)$ covers a flat torus in $\S^3$; see Remark \ref{rem:tori}. Therefore, it follows from Remark \ref{tori} that $\delta(r_3)$ can be continuously extended to $r_3=1$, with value $\delta(1)=\frac{H + \mu}{2}$.
Besides, since $\psi(\mathfrak{s},\theta)$ is a parametrization by curvature lines of $\Sigma$, it is easy to see that the parameter $u$ of $\Sigma$ is linked to the arclength parameter $\mathfrak{s}$ by $\mathfrak{s}=\mathfrak{s}(u)$, with
\begin{equation}\label{cambisu}
\mathfrak{s}'(u)= \frac{e^{\rho(u)}}{2\mu}, \hspace{0.5cm}\mathfrak{s}(0)=0.
\end{equation}
Here, the fact that $\mathfrak{s}=0$ corresponds to $u=0$ comes from the fact that, in both parametrizations, the value $0$ gives a point where the geodesic curvature of  the profile curve has a maximum. In this way, the change of parameters $\mathfrak{s}=\mathfrak{s}(u)$ is real analytic also with respect to $r_3$, and is independent from $r_1$.

This leads to the following definition.
\begin{definition}\label{tildeu}
We define $\tilde u (r_3)$ as the value $\tilde u (r_3):= u(\tilde{\mathfrak{s}};r_3)$, where $\tilde{\mathfrak{s}}$ is defined in Proposition 
\ref{pro:hatp}.  
 By construction, the map $r_3\mapsto \tilde u(r_3)$ is real analytic.
\end{definition}

\begin{remark}\label{exput}
Assume that $\Theta(r_1,r_3)=m/n\in \Q\cap (-1,0)$, with $n\in \N$ and $m/n$ irreducible. Choose $u_0:=\tilde u(r_3)$, and let $\Sigma_0:=\Sigma_0(1,b,c,u_0)\equiv \Sigma_0(r_1,r_3,u_0)$ be the compact $H$-annulus in $\M^3(\ep)$ of Theorem \ref{th:nivelperiodo}. Then, by the definition of $\tilde u(r_3)$ and the previous discussion, we deduce that $\Sigma_0$ covers a critical catenoid or nodoid $\cN$ in $\M^3(\ep)$. The value of $r_3$ determines the specific catenoid or nodoid (since it determines its neck curvature). Contrastingly, $r_1$ can be regarded as a free parameter in this description.

Let us be more specific about this covering property. Since the central planar geodesic $\Gamma:={\bf S}\cap \Sigma_0$ of $\Sigma_0$ has rotation index $m\in \Z^-$ (by Proposition \ref{pro:peri}), we observe that $\Sigma_0$ is a finite $(-m)$-cover of $\cN$. In particular, if $\Theta(r_1,r_2)=-1/n$ for some $n\geq 2$, then $\Sigma_0$ is an embedding, since it is a trival covering of the embedded compact annulus $\cN \subset \M^3(\ep)$.
\end{remark}
\begin{proposition}\label{pro:hp3}
    Following the above notations, let $(r_1^0,r_3^0) \in \hat \cW$, and assume that $\tilde u (r_3^0)=\tau(r_1^0,r_3^0)$. Then, $\mh(r_1^0,r_3^0) = 0$.
    
\end{proposition}
\begin{proof}
Take $(r_1,r_3) \in \hat \cW$. By Remark \ref{rem:ejea1}, Proposition \ref{rco} and Proposition \ref{pro:plaver}, we know that the rotation axis $\cL$ of $\Sigma(r_1,r_3)$ is contained in the 2-plane $\cP$ of $\R_{\ep}^4$ where $\{m(u):u\in \R\}$ is contained. Also, by continuity of the functions $\tilde u, \tau$, we can define $\hat{p}(\mathfrak{s}(\tau))$ for all $(r_1,r_3)$ close enough to $(r_1^0,r_3^0)$, where $\hat p$ is defined in Proposition \ref{pro:hatp}. By construction, $\hat p(\mathfrak{s})$ lies in $\cL$. In particular, both $m(\tau)$ and $\hat p (\mathfrak{s}(\tau))$ lie in $\cP$.

Let now $\cP'$ denote the $2$-plane of $\R_{\ep}^4$ generated by $\psi(\tau,0)$ and $\psi_u(\tau,0)$, where $\tau=\tau(r_1,r_3)$. By construction, $\hat p(\mathfrak{s}(\tau))\in \cP'$. And on the other hand, it follows from the definition of $\tau$ and \eqref{eq:m} that $m(\tau)\in \cP'$. Note also that $\cP\neq \cP'$, since $\psi(\tau,0)\not\in \cP$, due to the fact that catenoids and nodoids in $\M^3(\ep)$ do not touch their rotation axis, see Section \ref{sec:rotational}. Thus, $m(\tau)$ and $\hat p(\mathfrak{s}(\tau))$ are collinear, since they both lie in the line $\ell:=\cP\cap \cP'$ of $\R_{\ep}^4$. Depending on the value $\varepsilon$, we deduce the following:

    If $\varepsilon = 1$, then $\ell \cap \S^3$ consists of the two antipodal points $\{m(\tau)$, $-m(\tau)\}$. So, changing $m(u)$ by $-m(u)$ if necessary (recall that $m(u)$ was defined up to a multiplicative factor, see the discussion after \eqref{eum}), we deduce from $\hat p(\mathfrak{s}(\tau)) \in \S^3$ that 
\begin{equation}\label{peeme}
\hat p(\mathfrak{s}(\tau)) = m(\tau).
\end{equation}
    Similarly, if $\varepsilon = -1$, then $\ell$ must be timelike, since it contains the point $\hat p(\mathfrak{s}(\tau)) \in \H^3$. In particular, $m(\tau) \in \ell$ is also timelike, and so $m(\tau) \in \H^3$, see again the discussion after \eqref{eum}. The intersection $\ell \cap \H^3$ consists of a single point, and so \eqref{peeme} holds again.

 To prove Proposition \ref{pro:hp3}, assume finally that $(r_1,r_3) = (r_1^0,r_3^0)$. By \eqref{peeme},  Proposition \ref{pro:hatp} and Definition \ref{hrara}, $$\mh(r_1^0,r_3^0) = \hat p_3(\mathfrak{s} (\tau)) = \hat p_3(\mathfrak{s} (\tilde u)) = \hat p_3(\tilde{\mathfrak{s}}) = 0,$$ what completes the proof.
\end{proof}

The following result allows to find roots of $\mh(r_1,r_3)$ along curves in $\hat\cW$ joining two of the boundary curves of $\parc \hat\cW$.

\begin{theorem}\label{th:upsilon}
Let $\Upsilon:[0,1]\to \R^2$, $\Upsilon(r) = (r_1(r),r_3(r))$, be an analytic curve satisfying:
\begin{enumerate}
\item 
$\Upsilon(r)\in \hat{\cW}$ for every $r\in [0,1)$.
\item
$\Upsilon(0)=(0,r_3(0))$.
\item
$\Upsilon(1)= (\overline r_1, \overline r_3)\notin \hat\cW$, with
\begin{equation}\label{upsilon1}
    \overline{r}_3 = \frac{(\overline{r}_1 - 1)^2}{1 - 2\overline{r}_1}, \hspace{0.5cm} \overline{r}_3 > -\ep \frac{H+\mu}{2\mu}.
\end{equation}
\end{enumerate}
Then there exists some value $r^*\in (0,1)$ such that $\mh(\Upsilon(r^*)) = 0$. Moreover, $\mh(r):= \mh(\Upsilon(r))$ changes sign at this point; specifically, $\mh(r) > 0$ for $r \in (r^* - \epsilon, r^*)$ and $\mh(r) < 0$ for $r \in (r^*,r^* + \epsilon)$, for some $\epsilon > 0$.
\end{theorem}

\begin{figure}[h]
  \centering
  \includegraphics[width=0.5\textwidth]{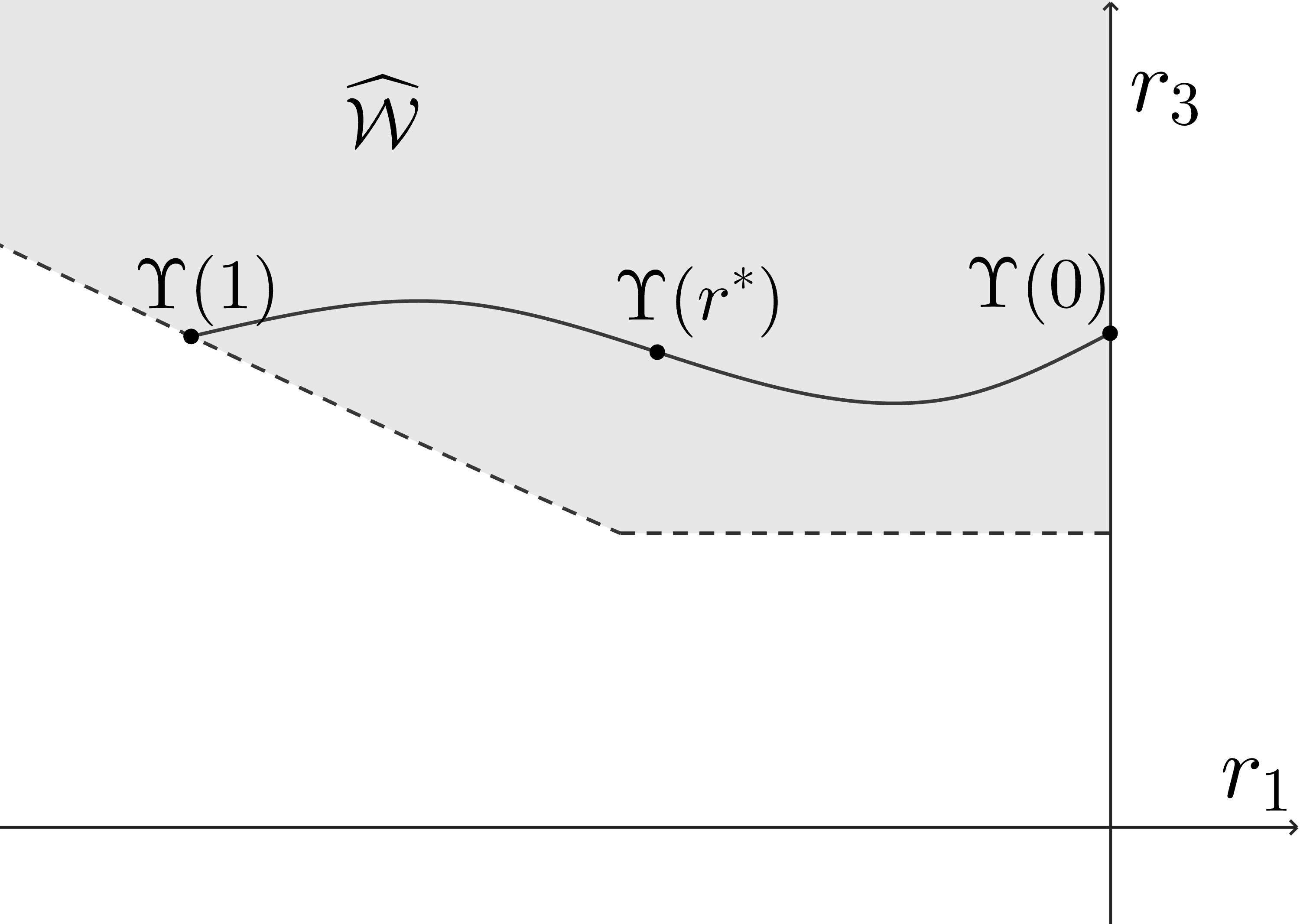}
  \caption{Example of a curve $\Upsilon(r)$ in the conditions of Theorem \ref{th:upsilon}.}
\end{figure}

\begin{proof}
Let us consider the analytic function $f:[0,1) \to \R$ defined as
$$ f(r) := \tau(\Upsilon(r))  - \tilde u (r_3(r)) \equiv \tau(r) - \tilde u(r_3(r)).$$
Our goal will be to prove that $f$ vanishes at some point $r^*\in (0,1)$ and it changes sign. From that, we will conclude that $\mh(r)$ shows the same behaviour at such point. 

We show first that $f(0) > 0$. Since $r_1=0$ at $\Upsilon(0)$, we see from the comments at the end of Section \ref{sec:double} that the solution $(s(\lambda),t(\lambda))$ of the system \eqref{system2} associated to the parameters $(0,r_3(0))$ satisfies $t(\landa) \equiv 0$, while $s(\lambda)$ oscillates between $s = 1$ and $s = r_3(0)$. According to \eqref{change}, $z(u) \equiv 0$ while $y(u)$ is not constant. By definition, $\tau$ is the first value for which $y(\tau) = z(\tau)$ (see Proposition \ref{prosis2}), so in this case $\tau$ is the first positive root of $y(u)$, that is, $\tau(0,r_3(0)) = u_1(0,r_3(0))$; see Proposition \ref{prosis1} for the definition of $u_1$. We will now show that $\tilde u(r_3(0)) < u_1(0,r_3(0))$. The curvature of the profile curve $u \mapsto \psi(u,0)$ is $\kappa_1(u,0) = H + \mu e^{-2\rho(u,0)}$, by \eqref{princur}. By \eqref{omu} and $z(u)\equiv 0$, this curvature is strictly decreasing on the interval $u \in [0,u_1]$, and it actually has a local minimum at $u=u_1$. On the other hand, by Proposition \ref{pro:hatp} the principal curvature $\kappa_1$ is strictly decreasing on an interval $[0, u_0)$ containing $\tilde u$. So, we must have $\tilde u(r_3(0)) < u_1(0,r_3(0))=\tau(0,r_3(0))$. As a consequence, $f(0) > 0$.

We will show next that $\lim_{r \to 1^-}f(r) < 0$. To start, note that the function $\tilde u(r_3(r))$ is positive by construction. Also, $\tilde u(r_3(r))$ can be defined at $r=1$ by the inequality in \eqref{upsilon1}. So, $\lim_{r \to 1^-}\tilde u(r_3(r)) = \tilde u(\overline{r}_3) > 0$.
Now we claim that $\lim_{r \to 1^-} \tau(r) = 0$. Assume by contradiction that there exists a sequence $(r_n)_n \to 1 $ of values in $(0,1)$ such that $\lim \tau(r_n) =: L_0 > 0$. Let $(y_n(u),z_n(u))$ denote the solutions $(y(u),z(u))$ of \eqref{system1} associated to the parameters $\Upsilon(r_n)$, and $(y_0(u),z_0(u))$ be the corresponding solution for the limit case $\Upsilon(1)$. By definition of $\tau$, we have that $y_n(u) > z_n(u)$ for all $u \in (0, \tau(r_n))$. By continuity of solutions of ODEs with respect to initial conditions and parameters, we deduce that $y_0(u) \geq z_0(u)$ for all $u \in [0,L_0]$. This leads to a contradiction, since the equality in \eqref{upsilon1} implies that $y_0(u) < z_0(u)$ for any $u > 0$ small enough; see the discussion after Proposition \ref{prosis2}.

Therefore, the analytic function $f(r)$ changes sign on the interval $[0,1)$. More specifically, there exists a value $r^* \in (0,1)$ with $f(r^*)=0$, and so that $f(r) > 0$ for $r \in (r^* - \epsilon, r^*)$ and $f(r) < 0$ for $r \in (r^*, r^* + \epsilon)$, for $\epsilon>0$ small enough.

Let us now study the function $\mh(r)$. We know that $\tau(r^*) = \tilde u(r_3(r^*))$, so by Proposition \ref{pro:hp3}, $\mh(r^*) = 0$. In order to prove that $\mh(r)$ changes sign at $r^*$, consider $r\in (r^*-\epsilon,r^*)$ for some $\epsilon > 0$ small enough. Then, at $\Upsilon(r)$ we have $\tau>\tilde u$, since $f(r)>0$. Thus, $\mathfrak{s}(\tau)>\mathfrak{s}(\tilde u)=\tilde{\mathfrak{s}}$ and by Proposition \ref{pro:hatp} we have $\hat{p}_3 (\mathfrak{s}(\tau))>\hat{p}_3(\tilde{\mathfrak{s}})=0$, since $\hat p(\mathfrak{s})$ is strictly increasing near $\tilde{\mathfrak{s}}$. Similarily, $\hat{p}_3 (\mathfrak{s}(\tau))<0$ at $\Upsilon(r)$ if $r \in (r^*, r^* + \epsilon)$. Now, using \eqref{peeme}, this implies that $m_3(\tau)$ changes sign at $r=r^*$ along $\Upsilon(r)$. This proves that $\mathfrak{h}(r)$ changes sign at $r^*$, see Definition \ref{hrara}.
\end{proof}

We will next prove a further result that, in the spherical case $\ep=1$, will allow us to find roots of $\mathfrak{h}(r_1,r_3)$ along curves in $\hat\cW$ that start from the \emph{vertex} $(r_1,r_3)=(0,1)$.

\begin{lemma}\label{gammalimite}
    Let $\varepsilon = 1$. Given $l > 0$, consider the half-line given by 
    \begin{equation}\label{halfla}
    L(r):= (r_1(r),r_3(r)) =  (-r,l r + 1)
    \end{equation}
    for all $r > 0$. Then, there exists $\lim_{r \to 0^+} \tau(L(r)) - \tilde u(r_3(r)) > 0$. 
    In particular, $\tau(L(r)) > \tilde u(r_3(r))$ for any $r > 0$ small enough.
\end{lemma}
\begin{proof}
Note that the expression in this directional limit makes sense, since $L(r) \in \hat \cW $ for $r > 0$ small enough. 

To start, we prove that $\tilde u(r_3)$ can be extended to $r_3=1$. Recall first of all that $\delta(r_3)$ converges to $\frac{H + \mu}{2}$ as $r_3\to 1$. By Proposition \ref{pro:hatp}, the limit of $\tilde{\mathfrak{s}}(H,\delta)$ at $\delta=\frac{H + \mu}{2}$ is given by \eqref{eq:SHd}.
On the other hand, if $r_3=1$, then $c=1$ and we are in the conditions of Remark \ref{rem:tori}. So, $\rho(u)\equiv 0$, and from \eqref{cambisu} we have $\mathfrak{s}(u)=u/(2\mu)$ in that situation. Therefore, there exists $$\tilde u(1):=\lim_{r_3\to 1^+}\tilde u (r_3)= 2\mu \, \tilde{\mathfrak{s}}\left(H,\frac{H + \mu}{2}\right)= \frac{\pi}{\sqrt{2}}\sqrt{\frac{\mu}{H + \mu}}\leq \frac{\pi}{\sqrt{2}}.$$
In particular, from \eqref{halfla}, there exists $\tilde u(r_3(0))<\pi$. We will now prove that $\lim_{r \to 0^+} \tau(L(r)) = \pi$, and so the Lemma follows.

Let $r > 0$ so that $L(r) \in \hat \cW$, and define $\hat \tau(r) := \lambda(\tau(L(r)))$, where $\lambda = \lambda(u)$ is the parameter used in the differential system \eqref{system2}, in which we recall that the polynomial $q(x)$ is given by \eqref{qfactor}. Equivalently, $\hat \tau(r)$ can be defined as the first positive value for which $s(\hat \tau(r); r) + t(\hat \tau(r);r) = 1$, where by $s(\lambda;r)$, $t(\lambda;r)$ we denote the solutions of \eqref{system2} for the parameters $(r_1(r),r_3(r))$.  
Now, we define the functions
\begin{equation}\label{cambioST}
    S(\lambda; r) := \frac{s - 1}{r_3(r) - 1} = \frac{s - 1}{l r}, \; \; \; \; T(\lambda;r) := \frac{t}{r_1(r)} = -\frac{t}{r}.
\end{equation}
In this way, $\hat\tau(r)$ corresponds with the first value $\landa>0$ for which 
\begin{equation}\label{efelan}
F(\lambda;r):= l S(\lambda;r) - T(\lambda;r)
\end{equation}
vanishes. From \eqref{system2} and \eqref{qfactor}, we deduce that $S,T$ are solutions of the differential system
\begin{equation}\label{ST}
\left\{\def\arraystretch{1.6} \begin{array}{lll} S'(\landa)^2 & = & S (1-S) (1+r l S)(1+r + r l S)^2, \hspace{0.5cm} (S\geq 0), \\ T'(\landa)^2 & = & r T(1-T)^2 (1+r T) (1+l r + r T), \hspace{0.6cm} (T\geq 0),\end{array} \right.
\end{equation}
with the initial conditions 
\begin{equation}\label{STini}
S(0;r)=T(0;r)=0.
\end{equation}
More specifically, since $r>0$, we have that $s(\landa;r)$ and $t(\landa;r)$ are not constant, see the comments at the end of Section \ref{sec:double}. Thus, from \eqref{cambioST}, we see that $S(\landa;r)$, $T(\landa;r)$ are the unique non-constants solutions to \eqref{ST}-\eqref{STini}. By differentiation of \eqref{ST} and using \eqref{STini}, we obtain $2S''(0;r)=(1+r)>0$ and $2T''(0;r)=r(1+l r)>0$.

The system \eqref{ST} depends analytically on the parameter $r$, and for the limit $r = 0$ case we can solve it directly. By the previous observations, we deduce that $S(\landa;r)$ and $T(\landa;r)$ converge analytically as $r\to 0$ to 
\begin{equation}\label{ST0}
        S(\lambda;0) = \frac{1 -\cos(\lambda)}{2}, \; \; \;  T(\lambda;0) \equiv 0. 
    \end{equation}
Let us prove that $\lim_{r \to 0^+} \hat \tau(r) = 2\pi$.
First, we show that $\limsup \hat \tau(r) \leq 2\pi$. Let $\lambda_1(r):= \lambda(u_1(L(r)))$, see Proposition \ref{prosis1} for the definition of $u_1$. Then, by Proposition \ref{prosis1} and \eqref{change}, \eqref{cambioST}, $\lambda_1(r)$ is the first positive value for which $S(\lambda;r)$ vanishes. According to \eqref{ST}, this quantity is equal to
$$\lambda_1(r) = 2\int_0^1 \frac{dS}{\sqrt{S(1-S)( l r S + 1) (l rS  + 1 + r)^2}},$$
which depends analytically on $r$, and $\lambda_1(0) = 2\pi$. Notice, on the other hand, that $\hat \tau (r) \leq  \lambda_1(r)$ for all $r > 0$, according to Proposition \ref{prosis2}. Consequently, $\limsup \hat \tau(r) \leq 2\pi$.

We will now prove that $\delta := \liminf_{r \to 0^+} \hat \tau(r) \geq 2\pi$. First, note that $F(0;r) = F'(0;r) = 0$ by \eqref{ST}, \eqref{STini}. 

By definition of $\delta$ and $F(\landa;r)$, there exists a sequence $(\lambda_n,r_n)$ such that $\lambda_n \to \delta$, $r_n \to 0$ and $F(\lambda_n;r_n) = 0$ for all $n$, with $F(\landa;r_n)\neq 0$ for all $\landa\in (0,\landa_n)$. By continuity of $F$, it follows that $F(\delta;0) = 0$. By \eqref{ST0} and \eqref{efelan}, we see that $\delta$ must be either $0$ or $2\pi$ (recall that $\limsup \hat \tau(r) \leq 2\pi$). Assume by contradiction that $\delta = 0$. Since $F(0;r_n) = F'(0;r_n) = F(\lambda_n;r_n) = 0$ for all $r_n$, we deduce the existence of a sequence $\lambda^*_n \in (0,\lambda_n)$ such that $F''(\lambda^*_n;r_n) = 0$ for all $n$. Since $\delta=0$ by hypothesis, we have $\lambda^*_n \to 0$, and so $F''(0;0) = 0$ by continuity. However, a direct computation using \eqref{efelan}, \eqref{ST0} shows that $F''(0;0) = \frac{l}{2} > 0$, a contradiction. In conclusion, $\liminf_{r \to 0^+} \hat \tau(r) = \lim_{r \to 0^+} \hat \tau(r) = 2\pi$.

Finally, to compute $\lim_{r \to 0^+} \tau(L(r))$, we just make use of the change of variables $u = u(\lambda)$ given by \eqref{changecor} and the fact that $s(\lambda;r), t(\lambda;r)$ converge uniformly to the constants $1$ and $0$ respectively; see the comments at the end of Section \ref{sec:double}. This yields
$$\lim_{r \to 0^+} \tau(L(r)) = \lim_{r \to 0^+} \int_0^{\hat \tau(r)}\frac{s(\lambda;r) - t(\lambda;r)}{2}d\lambda = \int_0^{\hat \tau(0)}\frac{1 - 0}{2}d\lambda = \pi.$$
This completes the proof of Lemma \ref{gammalimite}.
\end{proof}

\section{Free boundary minimal and CMC annuli}\label{sec:proofmain}

In this section we will prove Theorems \ref{th:main1} and \ref{th:main2}. The general idea behind these results is to study the rational level sets of the Period map (see Definition \ref{def:period}), $\Theta(a,b,c) = \Theta_0 \in (-1,0) \cap \Q$. Suppose that for some parameters $(a,b,c)$ in this level set we have that $\mh(a,b,c) = 0$ (see Definition \ref{hrara}). By Theorem \ref{th:nivelperiodo}, the associated surface $\Sigma_0$ is an annulus, and by Corollary \ref{cor:hrara}, $\Sigma_0$ intersects orthogonally a certain totally umbilic sphere $\mathcal{Q}$ of $\M^3(\varepsilon)$. Under certain conditions, we will be able to prove that $\Sigma_0$ is actually contained in the geodesic ball ${\bf B}$ whose boundary is $\mathcal{Q}$, that is, $\Sigma_0$ is {\em free boundary} in ${\bf B}$. The embeddedness of $\Sigma_0$ will also be studied. 

The results obtained in Section \ref{detectcritical} can be applied to find roots of $\mh(a,b,c)$ when $a = 1$. In that case, the corresponding compact annuli $\Sigma_0(1,b,c,u_0)$ are pieces of nodoids or catenoids: see Section \ref{sec:rotational}. Our goal now is to find new zeros of $\mh$ with $a > 1$, which will correspond with non-rotational CMC annuli.

\subsection{Proof of Theorem \ref{th:main1}}
The proof will consist of two steps. On the first one, we will show that for any $H \geq 0$, $\varepsilon = \pm 1$ such that $H^2 + \varepsilon > 0$, there are countably many values $\Theta_0 \in (-1,0) \cap \Q$ such that the function $\mh(a,b,c)$ vanishes along an analytic curve $\cC(\eta) \subset \mathcal{W}$ on the level set given by $\Theta(a,b,c) = \Theta_0$. Then, on the second step, we will show that the family of surfaces $\Sigma_0 = \Sigma_0(\eta)$ associated to the curve $\cC(\eta)$ constitute examples of free boundary $H$-annuli in $\M^3(\varepsilon)$.

For the {\bf first step}, we distinguish two cases, depending on the sign of the quantity $8H^2 - \varepsilon$. 

\textbf{First step, case $8H^2 - \varepsilon > 0$.} Let $\mathcal{J}$ be the interval 
\begin{equation}\label{eq:intervalo}
\mathcal{J} =
\begin{cases}
    \left(-\frac{1}{\sqrt{3}},0\right), & \text{if } \varepsilon = -1, \\ \\
    \left(-\frac{1}{\sqrt{3}},-\sqrt{\frac{\mu - H}{2\mu}}\right), & \text{if } \varepsilon = 1,
\end{cases}
\end{equation}
where we recall that $\mu:=\sqrt{H^2+\ep}$. Note that $\mathcal{J}$ is indeed an interval when $\ep=1$, due to $8H^2-\ep>0$. 
Given some $\Theta_0 \in \mathcal{J} \cap \Q $, consider the level set of the Period map given by $\Theta(a,b,c) = \Theta_0$. On the plane $\{a = 1\}$, this level set can be expressed as \eqref{curvanivel} in terms of $(r_1, r_3)$ given by \eqref{roots12}, \eqref{root3}, i.e., as the line
\begin{equation}\label{upsilonr}
    \Upsilon(r) := (r_1(r),r_3(r)) = \left(-r,\frac{\Theta_0^2}{1 - \Theta_0^2} r + \frac{H + \mu}{2\mu (1-\Theta_0^2)}\right).
\end{equation}
We now show that a segment of $\Upsilon(r)$ satisfies the properties listed in the statement of Theorem \ref{th:upsilon}.

First, it is clear from \eqref{upsilonr} and $\Theta_0\in \mathcal{J}$ that $\Upsilon(0)=(0,r_3(0))\in \hat \cW$, i.e. the second condition of Theorem \ref{th:upsilon} holds.
Also, $\Upsilon (r)\in \hat \cW$ for $r\geq 0$ small enough; see Figure \ref{fig:firstcase}.

\begin{figure}[h]
  \centering
  \includegraphics[width=0.45\textwidth]{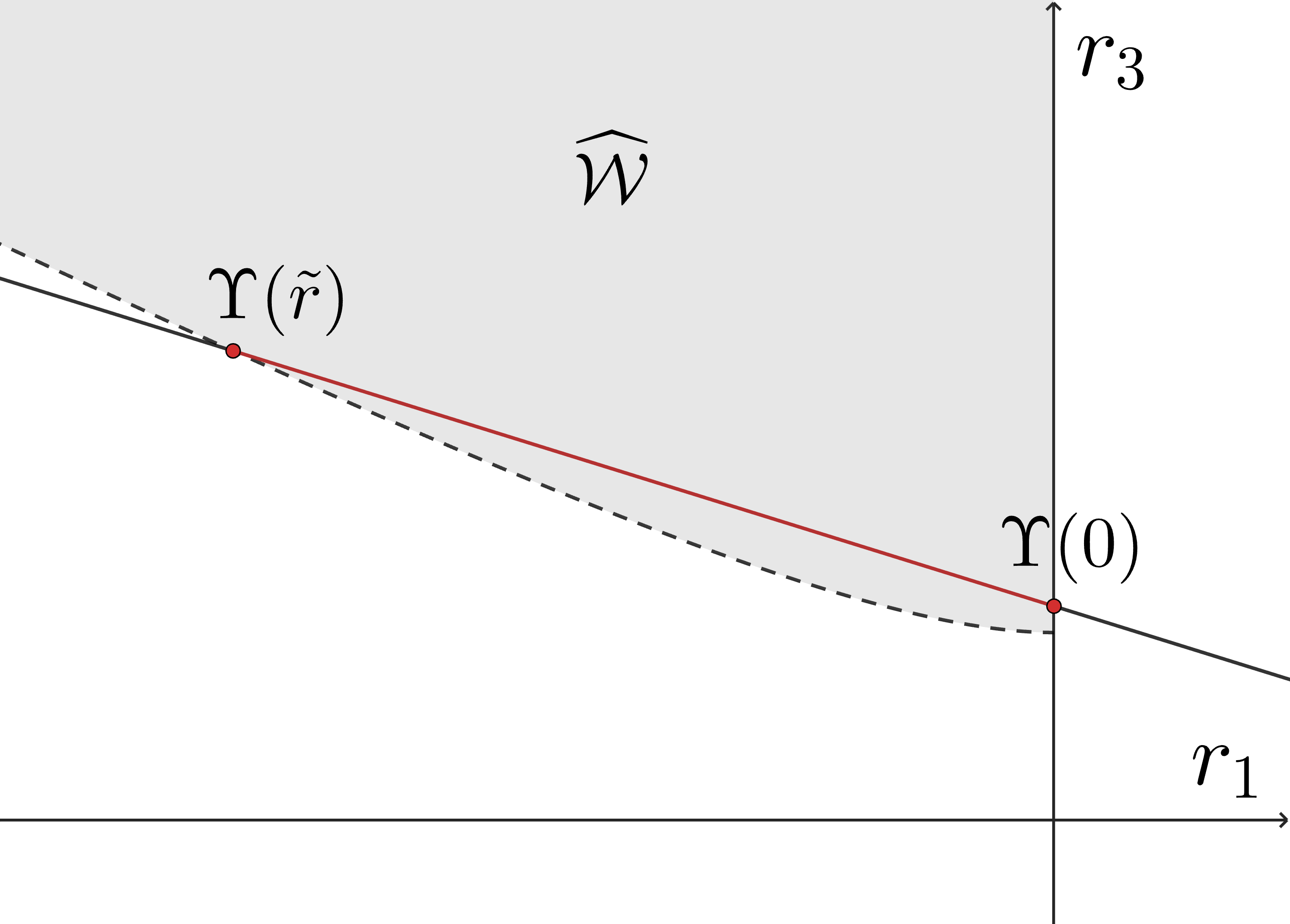}
  \hfill
  \includegraphics[width=0.45\textwidth]{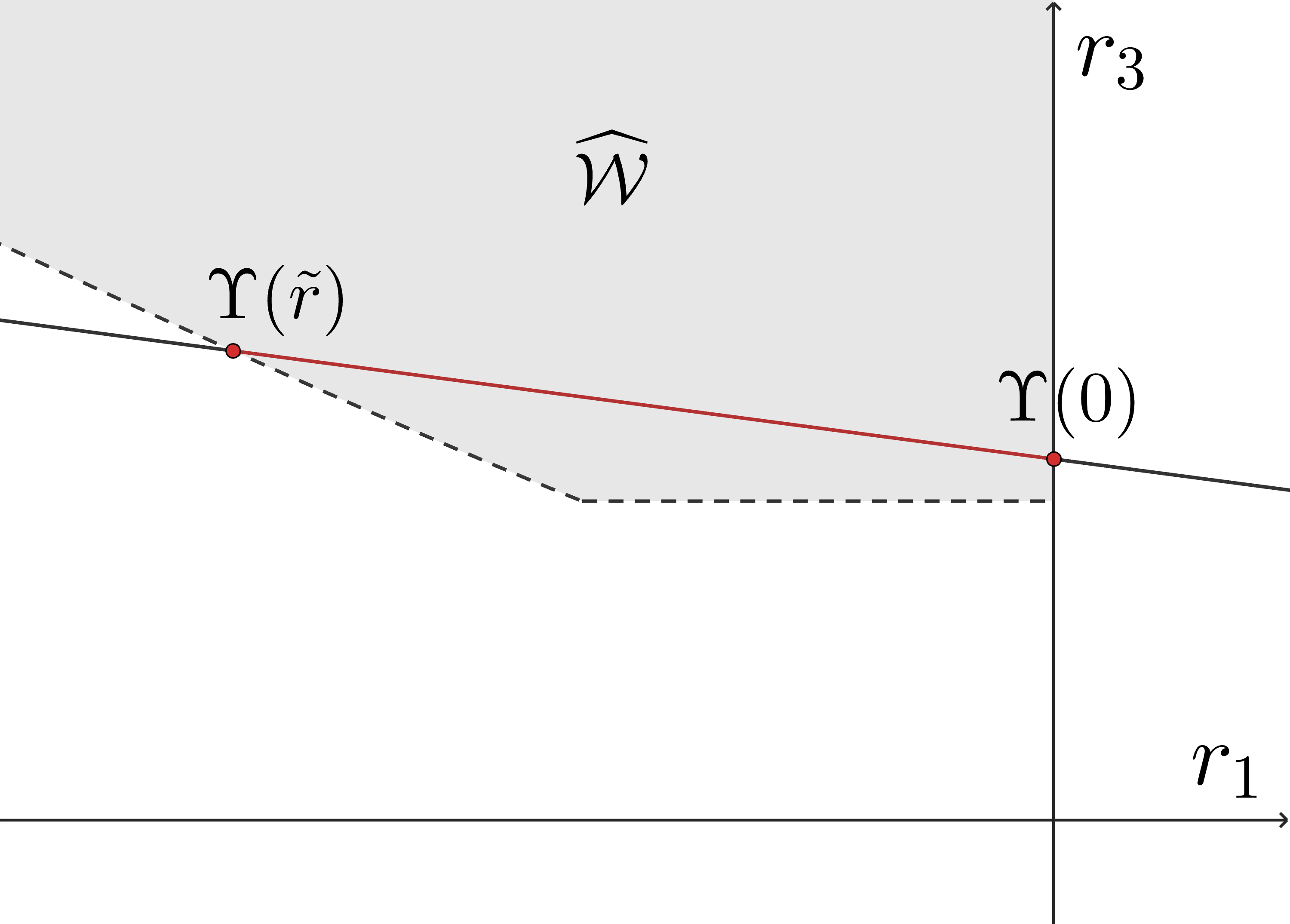}
  \caption{Level lines $\Upsilon$ of the period map on the parameter domain $(r_1,r_3)$ for the cases $\varepsilon = 1$ (left) and $\varepsilon = -1$ (right). If $8H^2 - \varepsilon > 0$, there are infinitely many level lines for which a segment of $\Upsilon$ satisfies the hypotheses of Theorem \ref{th:upsilon}.} 
  \label{fig:firstcase}
\end{figure}

On the other hand, $\Upsilon(r) \not \in \hat \cW$ for $r$ big enough; indeed, consider the inequality for $r_3$ in \eqref{ogra}. It is clear that $r_3(r)>-\ep \frac{H+\mu}{2\mu}$ for any $r\geq 0$, since $\Upsilon(r)$ has negative slope. However, we can prove that $$r_3(r)< \mathcal{G}(r):= \frac{(r_1(r)-1)^2}{1-2r_1(r)}$$ for $r$ large enough. This follows directly from $\Theta_0^2 <1/3$ (since $\Theta_0\in \mathcal{J}$) and the inequalities 
$$0<r_3'(r)=\frac{\Theta_0^2}{1-\Theta_0^2}<\frac{1}{2}=\lim_{r\to \8} \mathcal{G}'(r). $$
As a consequence, there is a first value $\tilde r > 0$ such that $\Upsilon(\tilde r) \not \in \hat \cW$, for which the equality in \eqref{upsilon1} is satisfied. Thus, the segment $\Upsilon(r):[0,\tilde{r}]\flecha \R^2$ is in the conditions of Theorem \ref{th:upsilon}, after the linear change $r\mapsto r/\tilde{r}$.

Thus, by Theorem \ref{th:upsilon}, there exists $r^*>0$ such that $\Upsilon (r^*)\in \hat \cW$, with $\mh(\Upsilon(r^*)) = 0$ and so that $\mh(\Upsilon(r))$ changes sign at $r^*$. Now, let $(1,b^*,c^*)$ be the point related to $(r_1(r^*),r_3(r^*))$ by the change \eqref{roots12}, \eqref{root3}. Consider the function
\begin{equation}\label{ghc}
    g(a,b) := \mh(a,b,c^{\Theta_0}(a,b)),
\end{equation}
defined on a neighborhood of $(1,b^*)\in \R^2$, where $c^{\Theta_0}(a,b)$ is the analytic map defined in Remark \ref{grafocab}, which parametrizes the level set $\Theta(a,b,c) = \Theta_0$ in a neighbourhood of $(1,b^*,c^*)\in \R^3$. We know by our previous discussion that $g(1,b^*) = 0$ and, moreover, that $g$ changes sign at $b^*$, meaning that $g(1,b) < 0$ (resp. $g(1,b)> 0$) for $b \in (b^* - \epsilon, b^*)$ (resp. $b \in (b^*, b^* + \epsilon)$), for $\epsilon>0$ small enough. Since $g(a,b)$ is analytic, there exists a real analytic curve $\zeta(\eta) := (a(\eta), b(\eta))$, $\eta\in [0,\delta)$ for $\delta$ small, satisfying $g(\zeta(\eta)) \equiv 0$ and $\zeta(0) = (1,b^*)$. Moreover, as $b \mapsto g(1,b)$ changes sign, then $a(\eta) > 1$ for all $\eta > 0$. We define then the curve
\begin{equation}\label{eq:cC}
    \cC(\eta):= (a(\eta),b(\eta),c^{\Theta_0}(\zeta(\eta))),
\end{equation}
which by definition is contained in the level set $\Theta(a,b,c) = \Theta_0$, and satisfies $\mathfrak{h}(\cC(\mu)) \equiv 0$. Observe that we obtain a different curve $\cC(\eta)$ for every $\Theta_0$ in the countable set $\mathcal{J} \cap \Q$. This concludes the proof of the {\bf first step} in the case $8H^2 - \varepsilon > 0$.

\textbf{First step, case $8H^2 - \varepsilon \leq 0$.} This implies in particular that $\varepsilon = 1$, and so $\mu > H$. Now, take any $l > 0$ such that $l < \frac{\mu - H}{\mu + H}$, and consider the line $L(r) := (-r, l r + 1)$. Also, for any 
\begin{equation}\label{deti}
\Theta_0 \in \left(-\sqrt{\frac{\mu - H}{2\mu}},-\sqrt{\frac{l}{1 + l}}\right)
\end{equation}
we consider the level line of the period map $\Upsilon = \Upsilon(r;\Theta_0)$ in \eqref{upsilonr}. The conditions on $\Theta_0$ in \eqref{deti} imply that $\Upsilon$ and $L$ meet at a point $L(\overline r) = \Upsilon(\overline r)$, $\overline r = \overline r(\Theta_0) > 0$; see Figure \ref{fig:LU}. The value $L(\overline r)$ depends analytically on $\Theta_0$, and for the limit case $\Theta_0 = -\sqrt{\frac{\mu - H}{2\mu}}$ it holds $\overline r = 0$, that is, $L$ and $\Upsilon$ meet at $(r_1,r_3) = (0,1)$. 

Note also that, by Lemma \ref{gammalimite}, the inequality $\tau(L(r)) > \tilde u(lr+1)$ holds for all $r > 0$ sufficiently close to zero. In particular, there exists a smaller interval $$\mathcal{J}: = \left(-\sqrt{\frac{\mu - H}{2\mu}},\tilde \Theta\right)$$ such that, for all $\Theta_0 \in \mathcal{J}$:
\begin{enumerate}
    \item The point $L(\overline r)$ belongs to $\hat \cW$, and so $\tau(L(\overline{r}))$ is well defined.
    \item The inequality $\tau(L(\overline r)) > \tilde u (l\overline r +1)$ is satisfied.
\end{enumerate}

We now fix some $\Theta_0 \in \mathcal{J} \cap \Q$ and define $\tau(r):= \tau(\Upsilon(r))$, $\mh(r) := \mh(\Upsilon(r))$. We will prove that there exists a value $r^* \in (0, \overline r)$ where $\tau(r^*) = \tilde u(r_3(r^*))$, with $r_3(r)$ as in \eqref{upsilonr}, and so $\mh(r^*) = 0$. 

By hypothesis, we know that $\Upsilon(\overline r)=  L(\overline r) \in \hat \cW$, and in fact $\tau (\overline r) > \tilde u(r_3(\overline r))$, by item (2) above. Since $r_3(r)$ is strictly increasing with $r_3(0)<1$ and $r_3 (\overline r) > 1$, we can define $r_b := r_3^{-1}(1) \in (0,\overline{r})$. 
Moreover, $\Upsilon(r_b) = (-r_b, 1)$ does not lie in $\hat \cW$, since the inequality for $r_3$ in \eqref{ogra} does not hold at that point. 
\begin{figure}[h]
  \centering
\includegraphics[width=0.46\textwidth]{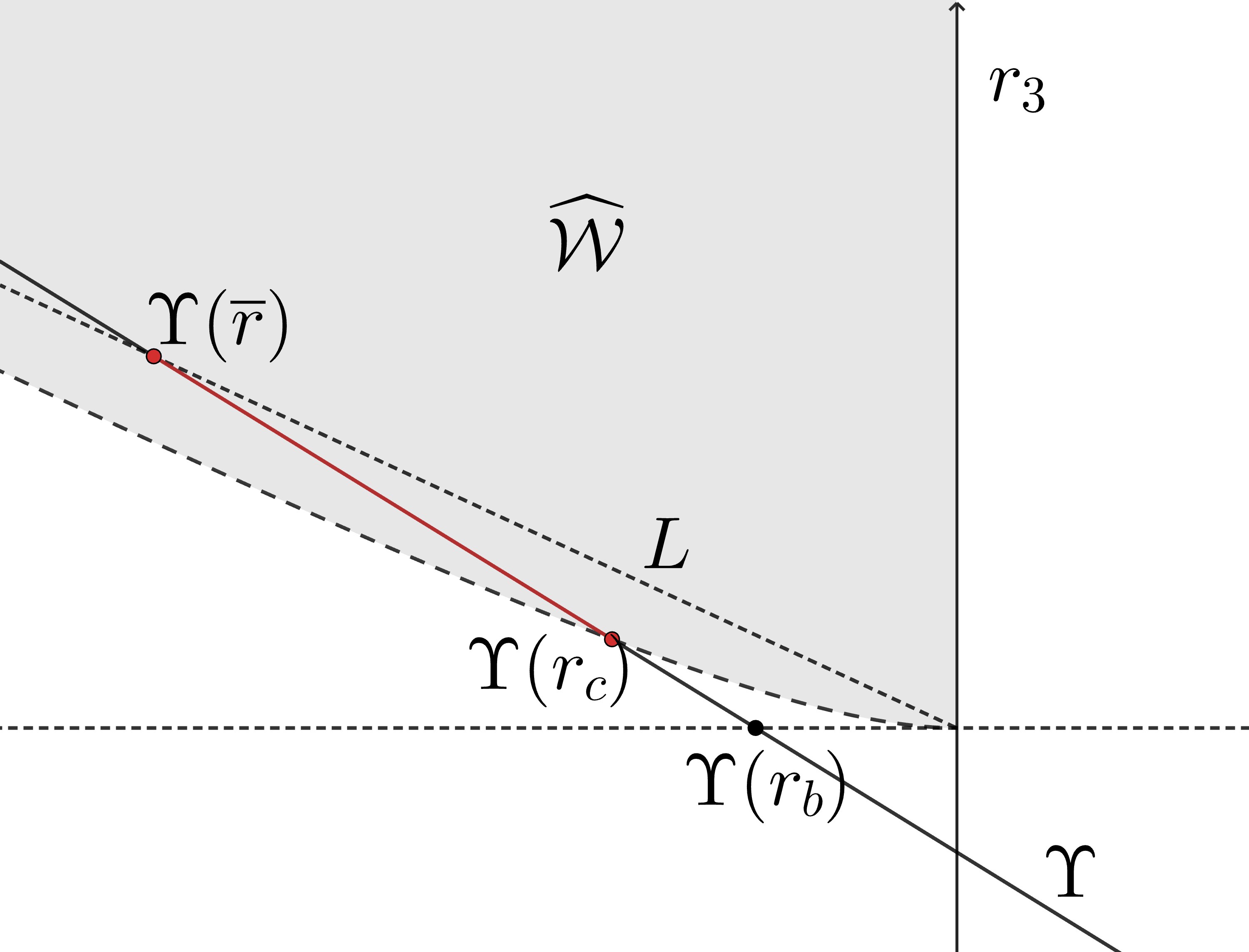}
  \caption{The level line $\Upsilon$ and the line $L$ meet at a point $\Upsilon(\overline r) \in \hat{\mathcal{W}}$.}\label{fig:LU}
\end{figure}

Consequently, there exists a certain interval $(r_c , \overline r] \subsetneq (r_b,\overline r]$ such that $\Upsilon(r) \in \hat\cW$ for all $r \in (r_c , \overline r]$, and $\Upsilon(r_c) \in \partial \hat\cW$; see Figure \ref{fig:LU}. By a similar argument to the one in the proof of Theorem \ref{th:upsilon}, it is possible to check that $\lim_{r \to r_c^+} \tau(r) = 0$. On the other hand, $\tilde u(r_3(r))$ is a positive function defined at $r = r_c$, so 
$$\lim_{r \to r_c^+} (\tau(r) - \tilde u(r_3(r))) = - \tilde u(r_3(r_c)) < 0.$$
In particular, there exists some $r^* \in (r_c, \overline r)$ where $\tau(r^*) = \tilde u(r_3(r^*))$. In fact, since $\tilde u$ and $\tau$ are analytic and they do not coincide, we can take $r^*$ so that the function $f(r):= \tau(r) - \tilde u (u_3(r))$ changes sign at $r = r^*$, being negative (resp. positive) for any $r < r^*$ (resp. $r > r^*$) close enough to $r^*$.

From the definition of $r^*$ and Proposition \ref{pro:hp3} it follows that $\mh(r^*) = 0$. In fact, since $f(r)$ changes sign at $r^*$, arguing as in the last part of the proof of Theorem \ref{th:upsilon}, we deduce that 
$\mh(r)$ also changes sign. Now, let $(1,b^*,c^*) \in \cO^-$ be the point associated with $(r_1(r^*),r_3(r^*))$ by \eqref{roots12}, \eqref{root3}. We deduce that the analytic function $b \mapsto g(1,b)$, with $g(a,b)$ given by \eqref{ghc}, changes sign at $b = b^*$. Consequently, as explained in the proof of the case $8H^2-\ep>0$ above, there exists an analytic curve $\zeta(\eta) = (a(\eta),b(\eta))$ such that $g(\zeta(\eta)) \equiv 0$, $\zeta(0) = (1,b^*)$ and $a(\eta) > 1$ for all $\eta > 0$. We deduce that the curve $\cC(\eta)$ defined as in \eqref{eq:cC}, contained in the level set $\Theta(a,b,c)= \Theta_0$, satisfies $\mathfrak{h}(\cC(\mu)) \equiv 0$. We remark that we obtain (at least) one such curve for every $\Theta_0$ in the countable set $\mathcal{J} \cap \Q$. This completes the {\bf first step of the proof}.

To sum up: so far, we have proved that for any values $H \geq 0$, $\varepsilon = \pm 1$ with $H^2 + \varepsilon > 0$, there is a countable number of curves $\cC_q := \cC_q(\eta):[0,\delta(q))\to \cW\cap \cO^-$, each of them contained in the level set of the Period map $\Theta(a,b,c) = q \in \mathcal{J} \cap \Q \subset (-1,0) \cap \Q$, with the property that $\mh$ vanishes identically along $\cC_q(\eta)$. Let us now define $\A_q = \A_q(\eta)$ as the compact annulus $\Sigma_0=\Sigma_0(a,b,c,\tau(a,b,c))$ of Theorem \ref{th:nivelperiodo} associated to the point $(a,b,c)=\cC_q(\eta)$. Our goal is to prove that $\A_q(\eta)$ satisfies each of the properties listed in Theorem \ref{th:main1}.

Items (2) and (5) of Theorem \ref{th:main1} are a direct consequence of Theorem \ref{th:nivelperiodo} and the fact that $a(\eta) > 1$ for all $\eta > 0$. Similarly, item (3) follows from Proposition \ref{pro:peri} and item (4) from Proposition \ref{pro:nod} and Remark \ref{exput}. Item (6) holds by construction, as any surface $\Sigma = \Sigma(a,b,c)$ has constant mean curvature $H$ and is foliated by spherical curvature lines. Let us now prove item (1). Since $\mh$ vanishes along the curve $\cC_q(\eta)$, we deduce by Corollary \ref{cor:hrara} that the annuli $\A_q(\eta)$ meet orthogonally a certain totally umbilic 2-sphere $\mathcal{Q} = \mathcal{Q}(q,\eta)$ of $\mathbb{M}^3(\ep)$ along their boundary. Let us denote by ${\bf B}(q,\eta)$ the geodesic ball of $\mathbb{M}^3(\ep)$ whose boundary is the sphere $\mathcal{Q}(q,\eta)$; in the case $\ep=1$ there are two such balls, and we choose the one for which $\psi(0,0) \in {\bf B}(q,\eta)$. For $\eta = 0$, we know that $\tau(\cC_q(0)) = \tilde u(\cC_q(0))$, so the compact nodoid or catenoid $\A_q(0)$ is one of the examples constructed in item 1 of Proposition \ref{pro:hatp}. In particular, this rotational annulus is contained in ${\bf B}(q,0)$. By real analyticity, we conclude that the (non-rotational) annuli $\A_q(\eta)$, $\eta \in (0,\delta_0(q))$, will also be contained in their respective balls ${\bf B}(q,\eta)$, at least for some $\delta_0(q) > 0$ small enough. This completes the proof of Theorem \ref{th:main1}.

\subsection{Proof of Theorem \ref{th:main2}}\label{sec:prueba2} The key idea is to study the level sets of the form $\Theta(a,b,c) = -1/n$, where $n \in \N$, $n \geq 2$. Suppose that either $\varepsilon = -1$ or $H \geq \frac{1}{\sqrt{3}}$. Then, the inequality
\begin{equation}\label{eq:periodon}
    \frac{\mu - H}{2\mu} \leq \frac{1}{n^2}
\end{equation}
is satisfied for some $n \geq 2$. More specifically, if $\varepsilon = 1$ and $H \geq \frac{1}{\sqrt{3}}$, then it holds for a finite set of natural numbers, while if $\varepsilon = -1$ it is true for all $n$. We will split our analysis into two cases, depending on whether the inequality \eqref{eq:periodon} is strict or not.

Suppose first that \eqref{eq:periodon} is strict for some $n \geq 2$. This means that we are in the case $8H^2-\ep>0$ detailed in the proof of Theorem \ref{th:main1}, and that $q_n:= -1/n$ lies in the open interval $\mathcal{J}$ defined in \eqref{eq:intervalo}. By the proof of Theorem \ref{th:main1}, we deduce that there is a curve $\cC_{q_n}(\eta):[0,\delta(n)) \to \cO^-$, contained in the level set $\Theta (a,b,c)= q_n$, such that the associated annuli $\A_{q_n}(\eta)$ are free boundary in a geodesic ball ${\bf B}(q_n,\eta)$ of $\M^3(\ep)$. It just remains to check the embeddedness of these examples, which will be studied later.

Suppose now that \eqref{eq:periodon} holds for the equality case. In that situation, we cannot use Theorem \ref{th:main1} directly. By the equality in \eqref{eq:periodon}, and using \eqref{upsilonr}, the level curve $\Theta(1,b,c) = q_n:=-1/n$ is expressed in terms of $(r_1,r_3)$ as
    $$\Upsilon(r) := (r_1(r),r_3(r)) =  \left(-r,\frac{r}{n^2 - 1} + 1\right).$$
This is the equation of a line $L(r):=\Upsilon(r)$ that is in the conditions of Lemma \ref{gammalimite}, so for any $r > 0$ small enough, we have $\Upsilon(r) \in \hat\cW$ and $\tau(\Upsilon(r)) > \tilde u(r_3(r))$. However, it is possible to check that $\Upsilon(r) \not \in \hat \cW$ for $r > 0$ large enough, by a similar argument to the first step of the proof of Theorem \ref{th:main1}. Also following the proof of that first step, we can deduce that there exists a point $r^*$ where $\tilde u(r_3(r^*)) = \tau(\Upsilon(r^*))$, and so $\mh(\Upsilon(r^*)) = 0$. In fact, we can show that there is a real analytic curve $\cC_n(\eta)$ such that $\mh$ vanishes identically along $\cC_n$, and from this we obtain a 1-parameter family of free boundary annuli $\A_n(\eta)$ that satisfies the properties stated in Theorem \ref{th:main1}. 

In conclusion, for any $n \geq 2$ such that \eqref{eq:periodon} holds, there exists a 1-parameter family of immersed, free boundary annuli $\A_{q_n}(\eta)$. It just remains to prove that the annuli $\A_{q_n}(\eta)$,  $\eta \in [0,\delta(n))$, are embedded for some $\delta(n) > 0$ small enough.

For every $\eta \in [0,\delta(n))$, let us denote by $\psi_\eta(u,v)$ our usual parametrization by curvature lines of the compact annulus $\A_{q_n}(\eta)$. We know by Corollary \ref{cor:sigma} that $\psi_\eta(u,v) = \psi_\eta(u,v + 2n \sigma)$, where $\sigma$ depends analytically on $\eta$. Identifying $(u,v) \sim (u,v + 2n\sigma)$ as usual (see Section \ref{sec:annu}), we can view $\psi_\eta$ as a parametrization $\psi_\eta:[-\tau(\eta),\tau(\eta)] \times \S^1 \to \A_{q_n}(\eta)$ of $\A_{q_n}(\eta)$. Observe that for $\eta = 0$, the annulus $\A_{q_n}(0)$ is an embedding, since it is a trivial covering of a critical catenoid or nodoid $\cN$ embedded in $\M^3(\ep)$; see Remark \ref{exput}. Thus, $\psi_0$ is injective.

Now, by the real analyticity of the family of compact annuli $\A_{q_n}(\eta)$, we deduce that for all $\eta$ sufficiently close to zero, the parametrizations $\psi_\eta$ are also injective, and so the annuli $\A_{q_n}(\eta)$ are embedded. This completes the proof.

\section{Embedded capillary minimal and CMC annuli in \texorpdfstring{$\S^3$}{}}\label{sec:capillary}

In Theorem \ref{th:main2}, we constructed embedded examples of non-rotational, free boundary CMC annuli in geodesic balls of $\H^3$ (for $H>1$) and $\S^3$ (for $H\geq 1/\sqrt{3}$). In this section we will show that if we relax the free boundary condition to \emph{capillarity}, then there exist embedded non-rotational capillary CMC annuli in $\S^3$ for any $H\geq 0$.

Recall that a compact  surface $\Sigma$ in $\M^3(\ep)$ is called a \emph{capillary} surface in a geodesic ball ${\bf B}\subset \M^3(\ep)$ if $\Sigma\subset {\bf B}$ intersects $\parc {\bf B}$ at a constant angle along $\parc \Sigma$. We prove next:

\begin{theorem}\label{th:main3}
For any $H\geq 0$ and any $n\geq 2$ there exists a real analytic $2$-parameter family of embedded capillary annuli $\A_n(a,\eta)$ with constant mean curvature $H$ in a geodesic ball ${\bf B}={\bf B}(n,a,\eta)$ of $\S^3$, with a prismatic symmetry group of order $4n$.
\end{theorem}

The rough idea behind this result is as follows: consider the level set of the period $\Theta(a,b,c) = -1/n$, and suppose that for some $(a,b,c)\in \cO^-$ with $a>1$ in this level set there exists a value $u^* > 0$ such that $m_3(u^*) = 0$, that is, the third coordinate of the center function vanishes. According to items (2), (4) and (5) of Theorem \ref{th:nivelperiodo}, the compact annulus $\Sigma_0 (a,b,c,u^*)$ intersects along $\parc \Sigma_0$ at a constant angle a totally umbilic sphere $\mathcal{Q}$ of $\S^3$, and $\Sigma_0$ has a prismatic symmetry group of order $4n$. Consequently, it suffices to check that the annuli $\Sigma_0$ are embedded and contained in  a geodesic ball ${\bf B}$ of $\S^3$ whose boundary is $\mathcal{Q}$. 

We will make use of the following lemma:
\begin{lemma}\label{lem:uestrella}
    Suppose that for some $(a_0,b_0,c_0) \in \cO^-$ there exists $u_0 > 0$ such that $m_3(u_0) = 0$. Then, there exists a neighbourhood $\mathcal{V}$ of $(a_0,b_0,c_0)$ and an analytic function $u^* = u^*(a,b,c): \mathcal{V} \cap \cO^- \to \R$ such that $u^*(a_0,b_0,c_0) = u_0$ and $m_3(u^*(a,b,c)) \equiv 0$.
\end{lemma}
\begin{proof}
The Lemma is a direct consequence of the implicit function theorem if we prove that $m_3'(u_0) \neq 0$. 
Assume by contradiction that $m_3'(u_0) = 0$. Since $m_3(u_0) = 0$, we deduce that the function $\tilde m_3(u)$ in \eqref{funcioncentro} satisfies $\tilde m_3(u_0)=\tilde m_3'(u_0)=0$, due to \eqref{mtildem}. Moreover, $\tilde m_3$ satisfies the differential equation \eqref{tildem}, and so we conclude that $\tilde m_3(u)$, and hence, $m_3(u)$, vanish identically. This is a contradiction, as \eqref{eq:m} and the initial conditions \eqref{idat} imply that $m_3(0) = 1$.
\end{proof}

\subsection{Proof of Theorem \ref{th:main3}} Let $n \geq 2$ and $\ep=1$. We will distinguish two cases, depending on whether or not \eqref{eq:periodon} holds. 

Assume first that \eqref{eq:periodon} holds, and consider the level set $\Theta(a,b,c) = -1/n=: \Theta_0$. Following the proof of Theorem \ref{th:main2} in Section \ref{sec:prueba2}, we know that there is a point $(1,b^*,c^*)$ in that level set such that $\mh(1,b^*,c^*) = m_3(\tau(1,b^*,c^*)) = 0$. By applying Lemma \ref{lem:uestrella} with $u_0 = \tau(1,b^*,c^*)$, we deduce the existence of a function $u^*$ defined on a neighbourhood $\mathcal{V}$ of $(1,b^*,c^*)$ such that $m_3(u^*(a,b,c))$ vanishes identically. Now, consider the analytic function $$(a,b) \mapsto u^*(a,b):= u^*(a,b,c^{\Theta_0}(a,b)),$$
where $c^{\Theta_0}(a,b)$ is the analytic map in Remark \ref{grafocab}. Let $\A_n(a,b) \equiv \psi([-u^*,u^*] \times \S^1)$ be the compact annulus associated to the parameters $(a,b,c^{\Theta_0}(a,b))$, where we identify the points $(u,v) \sim (u, v + 2n\sigma)$ as in Remark \ref{cor:sigma}. By construction, any of these annuli meets a totally umbilic sphere $\mathcal{Q}(n,a,b)$ of $\S^3$ with constant angle along its boundary $\parc \A_n(a,b)$, so in order to prove Theorem \ref{th:main3} we just need to check that the annuli $\A_n(a,b)$ are embedded and contained in a geodesic ball of $\S^3$ 
bounded by their corresponding sphere $\mathcal{Q}$. Notice that in Theorem \ref{th:main2} we already proved this for the annulus $\A_n(1,b^*)$, so by real analyticity, there is an open neighbourhood $G \subset \R^2$ of $(1,b^*)$ such that the annuli $\A_n(a,b)$ are also embedded and contained in a geodesic ball ${\bf B}(n,a,b)$ of $\S^3$ bounded by $\mathcal{Q}(n,a,b)$ for all $(a,b) \in G$. We also recall that there are two such geodesic balls in $\S^3$; we make the same choice for it that we did when proving Theorem \ref{th:main2}.

Take next some $n \in \N$, $n\geq 2$, such that \eqref{eq:periodon} does not hold, 
and consider the level curve $\Theta(1,b,c)=\Theta_0:=-1/n$. In the $(r_1,r_3)$-coordinates, this  curve is given by $\Upsilon(r)$ in \eqref{upsilonr}, and so it meets the horizontal line ${r_3 = 1}$ at a certain point $\Upsilon(r^{\alfa}) = (-r^{\alfa}, 1)$, where $r^{\alfa} > 0$. Let us see that $(1,b^{\alfa},c^{\alfa})\equiv \Upsilon(r^{\alfa})$ is in the conditions of Lemma \ref{lem:uestrella}.

Let $\Sigma$ denote the rotational $H$-surface associated to $(r_1,r_3)=\Upsilon(r^{\alfa})\equiv (1,b^{\alfa},c^{\alfa})$; note that $c^{\alfa}=1$ since $r_3=1$, see \eqref{root3}. Thus, by Remark \ref{rem:tori}, $\Sigma$ covers a flat CMC torus in $\S^3$. Since $\Theta=-1/n$, we can consider for any $u_0>0$ the compact immersed $H$-annulus $\Sigma_0=\Sigma_0(1,b^{\alfa},1,u_0)$ as defined in Section \ref{sec:romega}, after the identification $(u,v)\sim (u,v+2n\sigma)$. Under this identification, the $v$-curves $\psi(u_0,v):\S^1\flecha \S^3$ are injective parametrizations of circles.

We find next an explicit parametrization for the profile curve $\psi(u,0)$ of $\Sigma_0$, using the expresion of the flat $H$-torus in $\S^3$ given in Remark \ref{tori} in terms of the parameters $(\mathfrak{s},\theta)$. To start, note that it holds $e^{\rho(u)}\equiv 1$ for $\Sigma_0$ due to $c=1$, and so the reparametrizacion $u = u(\mathfrak{s})$ in \eqref{cambisu} is just $u = 2\mu \mathfrak{s}$. In addition, the rotation axis of the flat torus in Remark \ref{tori} is $\S^3\cap \{x_1=x_2=0\}$, which agrees with the geodesic of $\S^3$ that contains the centers $m(u)$ of $\Sigma$; see Remark \ref{rem:ejea1}. Thus, by \eqref{toro}, we see that the profile curve of $\Sigma$ is 
\begin{equation}\label{profitoro}
\psi(u,0) = \left(\sqrt{\frac{\mu + H}{2\mu}}, 0, \sqrt{\frac{\mu - H}{2\mu}}\sin\left(\sqrt{\frac{\mu + H}{2\mu}}u\right), \sqrt{\frac{\mu - H}{2\mu}}\cos\left(\sqrt{\frac{\mu + H}{2\mu}}u\right)\right).
\end{equation}

Let $u_0\in (0,\overline u)$, where $\overline u := \sqrt{\frac{2\mu}{\mu + H}}\pi$. From the expression of the profile curve $\psi(u,0)$ and the previous discussion, it follows that the parametrization $\psi:[-u_0,u_0]\times \S^1 \to \S^3$ of $\Sigma_0$ is injective, where we identify $(u,v) \sim (u, v + 2n \sigma)$ as usual. Moreover, $\Sigma_0$ is contained in the ball $B[{\bf e}_4,x_4(u_0)]$.

We now claim that there exists some $u_0 \in (0, \overline u)$ such that $m_3(u_0) = 0$. To prove this, note first that $f(u):= \langle m(u), \psi_u(u,0) \rangle$ never vanishes; indeed, by \eqref{eq:m}, it holds $f(u) = \frac{e^\rho}{2\mu}|\hat N | \sin \theta$. But now, we have $|\hat N| > 0$ by construction, and $\sin\theta$ cannot be zero by \eqref{expab}, as $\beta$ is bounded since the functions $(s,t)$ in \eqref{system2} are; see also \eqref{qfactor}. Therefore, $f(u)$ does not change sign. On the other hand, observe that, by \eqref{profitoro},
\begin{align*}
    f(0) &= \langle m(0), \psi_u(0,0)\rangle = \frac{m_3(0)}{2\mu},\\ 
    f(\overline u) &= \langle m(\overline u), \psi_u(\overline u,0)\rangle = -\frac{m_3(\overline u)}{2\mu}.
\end{align*}
Since the sign of $f(u)$ is constant in $(0,\overline u)$, we deduce that there is some $u_0 \in (0,\overline u)$ for which $m_3(u_0) = 0$. So, by Lemma \ref{lem:uestrella},  there exists a function $u^*(a,b,c)$ defined on a neighbourhood of $(1,b^\alpha,1)$, with $u^*(1,b^{\alfa},1) = u_0$ and $m_3(u^*(a,b,c)) \equiv 0$.

We consider the analytic function $u^*(a,c):= u^*(a,b^{\Theta_0}(a,c),c)$ and define for any $(a,c)$ in a neighborhood of $(1,1)$ the compact $H$-annuli 
$$\A_n(a,c) :=\Sigma_0(a,b^{\Theta_0}(a,c),c,u^*(a,c)),$$
where $b^{\Theta_0}(a,c)$ is the analytic map in Remark \ref{grafocab}.  By construction, $\A_n(a,c)$ meets a totally umbilic sphere $\mathcal{Q}(n,a,c)$ with constant angle along $\parc \A_n(a,c)$, according to Theorem \ref{th:nivelperiodo}. The annulus $\A_n(1,1)$ is equal to $\Sigma_0(1,b^{\alfa},1, u_0)$, which by our previous discussion, is embedded and contained in a geodesic ball of $\S^3$ bounded by $\mathcal{Q}(n,1,1)$. Consequently, there is a neighbourhood of $(1,1)$ such that the annuli $\A_n(a,c)$ are embedded capillary CMC annuli in a geodesic ball of $\S^3$ bounded by $\mathcal{Q}(n,a,c)$. This completes the proof of Theorem \ref{th:main3}.

\appendix

\section{Appendix: Proof of Proposition \ref{pro:hatp}}\label{sec:appendix}

In this Appendix we will prove separately the three items of Proposition \ref{pro:hatp}.

\subsection{Proof of item 1 of Proposition \ref{pro:hatp}}\label{ap:tildes}
By the symmetries of $\mathcal{S}_0$, if this annulus were free boundary, it should necessarily be in the ball $B[{\bf e}_4,\varepsilon x_4(\tilde{\mathfrak{s}})]$. In such case, $\mathcal{S}_0$ and $S[{\bf e}_4,\varepsilon x_4(\tilde{\mathfrak{s}})]$ would meet orthogonally along $\parc \mathcal{S}_0$. This happens if and only if the geodesic $\Gamma_{\tilde{\mathfrak{s}}}$ passes through the point ${\bf e}_4$, i.e., the center of the ball. As we will see in Appendix \ref{ap:geometriaF}, this property holds for the first positive root $\tilde{\mathfrak{s}}$ of the function
\begin{equation}\label{eq:F}
F(\mathfrak{s}):= x_3(\mathfrak{s}) x'(\mathfrak{s}) - x_3'(\mathfrak{s}) x(\mathfrak{s}),
\end{equation}
where $x, x_3$ are defined in \eqref{pro:s}. Consequently, our goal will be to show the existence of such $\tilde{\mathfrak{s}} = \tilde{\mathfrak{s}} (H,\delta)$. We will deal with several cases depending on the values $\varepsilon\in \{-1,1\}$, $H\geq 0$, $\delta>0$, but the general strategy is to prove that $F(0) < 0$ and that there is some $\mathfrak{s}_0 > 0$ with $F(\mathfrak{s}_0) > 0$.

We first consider the case where $\varepsilon = 1$ and $\delta < \frac{H + \mu}{2}$, or $\varepsilon = -1$ and $H > 1$. Then, we showed in Section \ref{sec:rotational} that the function $x(\mathfrak{s})$ in \eqref{eq:hx} takes values on the interval $[x_m,x_M]$, where $x_m < x_M$ are the two positive roots of $h(x)$; we recall here our assumption that $x(0)=x_m>0$. In this situation, we define $\mathfrak{s}_2 > 0$ as the first positive value for which $x(\mathfrak{s}_2) = x_M$. Consequently, $x(\mathfrak{s})$ will be strictly increasing for all $\mathfrak{s} \in [0,\mathfrak{s}_2]$. We remark that if $x \geq 0$, then $h(x) \geq 0$ if and only if $x \in [x_m,x_M]$. Additionally, if $\varepsilon = -1$, $H \leq 1$, then $x(\mathfrak{s})$ is unbounded, taking values on $[x_m,\infty)$, where $x_m$ is the unique positive root of $h(x)$. The function $x(\mathfrak{s})$ will be strictly increasing for all $\mathfrak{s} \geq 0$. In this case, if $x \geq 0$, then $h(x) \geq 0$ if and only if $x \geq x_m$. 

Let us now show that $F(0) < 0$. By construction, $x_3(0) = 0$, $x(0) = x_m > 0$ and $x'(0) = 0$, so it remains to prove that $x_3'(0) > 0$. By \eqref{eq:psiesferico}, \eqref{eq:phi} for $\ep=1$ and \eqref{eq:psihiperbolico}, \eqref{eq:phihiperbolico} for $\ep=-1$, this is equivalent to the fact that $\delta - H x_m^2 > 0$, that is, $x_0 := \sqrt{\delta/H} > x_m$. Assume by contradiction that $x_0 \leq x_m$. This implies that $h(x_0) \leq 0$, as $h(x)\leq 0$ in $[0,x_m]$. Now, a direct computation shows that $h(x_0) = x_0^2 - \varepsilon x_0^4$. This quantity is clearly positive when $\varepsilon = -1$ and also when $\varepsilon = 1$, since $x_0 \leq x_m < 1$. We reach a contradiction, what proves that $F(0)<0$.

We are going to show next that there is some $\mathfrak{s}_0 >0$ such that $F(\mathfrak{s}_0) > 0$, and so the existence of $\tilde{\mathfrak{s}}$ follows. We split our analysis into five different scenarios: the first three concern the spherical case ($\varepsilon = 1$), while the fourth and fifth cover the hyperbolic case ($\varepsilon = -1$).

\textbf{Case 1:} $\varepsilon = 1$, $\delta\in (0,H)$.

In this case, $x_0 = \sqrt{\delta/H}$ satisfies $x_m < x_0 < x_M$, so the function $\mathfrak{s} \mapsto \delta - Hx(\mathfrak{s})^2$, and consequently $\phi'(\mathfrak{s})$, changes sign on the interval $(0,\mathfrak{s}_2)$. Let $\mathfrak{s}_1 \in (0,\mathfrak{s}_2)$ be the value for which $\phi'(\mathfrak{s}_1) = 0$, so that $\phi(\mathfrak{s})$ is increasing on $[0,\mathfrak{s}_1]$. If $\phi(\mathfrak{s}_1) > \pi/2$, then we define $\mathfrak{s}_0 < \mathfrak{s}_1$ as the value for which $\phi(\mathfrak{s}_0) = \pi/2$. Otherwise, let $\mathfrak{s}_0 = \mathfrak{s}_1$. In any of the cases, we see that $x(\mathfrak{s}_0), x'(\mathfrak{s}_0) > 0$, $x_3'(\mathfrak{s}_0) < 0$ and $x_3(\mathfrak{s}_0) \geq 0$, so necessarily $F(\mathfrak{s}_0) > 0$. We note that the function $x_4(\mathfrak{s})$ defined in \eqref{pro:s} is strictly decreasing and positive for all $\mathfrak{s} \in [0,\mathfrak{s}_0)$.

\textbf{Case 2:} $\varepsilon = 1$, $\delta = H$.

A direct computation using \eqref{eq:hx} shows that $x_M = 1$. This implies that $x_3(0) = x_3(\mathfrak{s}_2) = 0$, and since $x_3'(0) > 0$, there exists a first $\mathfrak{s}_0 \in (0,\mathfrak{s}_2)$ such that $x_3'(\mathfrak{s}_0) = 0$. $x_3(\mathfrak{s})$ is increasing on $[0,\mathfrak{s}_0]$, so in particular $x_3(\mathfrak{s}_0) > 0$, and hence $F(\mathfrak{s}_0) > 0$. Note also that the function $\phi(\mathfrak{s})$, which must be increasing on $[0,\mathfrak{s}_0]$, satisfies $\phi(\mathfrak{s}_0) < \frac{\pi}{2}$: otherwise, there would be an intermediate value $\overline{\mathfrak{s}}$ such that $\phi(\overline{\mathfrak{s}}) = \frac{\pi}{2}$, and so $x_3'(\overline{\mathfrak{s}}) < 0$ at that point, reaching a contradiction. A consequence of this fact is that $x_4(\mathfrak{s})$ is decreasing and positive for all $\mathfrak{s} \in [0,\mathfrak{s}_0)$.

\textbf{Case 3:} $\varepsilon = 1$, $\delta > H$. 

In this case, $\delta - Hx^2 > H(1 - x^2) \geq 0$, so $\phi(\mathfrak{s})$ is increasing for all $\mathfrak{s} \in \R$. If we manage to prove that $\phi(\mathfrak{s}_2) > \pi/2$, then there exists a unique $\mathfrak{s}_0 \in (0,\mathfrak{s}_2)$ such that $\phi(\mathfrak{s}_0) = \pi/2$. In particular, $x_3(\mathfrak{s}_0) > 0$, $x_3'(\mathfrak{s}_0) < 0$, and so $F(\mathfrak{s}_0) > 0$. Additionally, we deduce that $x_4(\mathfrak{s})$ is decreasing and positive for all $\mathfrak{s} \in [0,\mathfrak{s}_0)$.

Consider the change of variables $w(\mathfrak{s}) = x(\mathfrak{s})^2$ on the integral that describes $\phi(\mathfrak{s}_2)$. Defining $w_m = w(0) = x_m^2$, $w_M = w(\mathfrak{s}_2) = x_M^2$, we get using \eqref{eq:hx} that
$$\phi(\mathfrak{s}_2) = \int_{w_m}^{w_M} \frac{\delta - Hw}{2\sqrt{w}(1 - w)\sqrt{w - w^2 - (Hw - \delta)^2}}dw.$$
We analyze this integral via residues. Note that 
\begin{equation}\label{resii}
z - z^2 - (Hz - \delta)^2=-(H^2+1)(z-w_m)(z-w_M)
\end{equation}
by definition of $w_m,w_M$, and consider from there the meromorphic function
$$f(z):= i\frac{\delta - Hz}{2\sqrt{H^2 + 1}(1 - z)\sqrt{z}\sqrt{z - w_m}\sqrt{z - w_M}}$$
\begin{figure} [hbt!]
 \begin{center}
\includegraphics[width=0.43\textwidth]{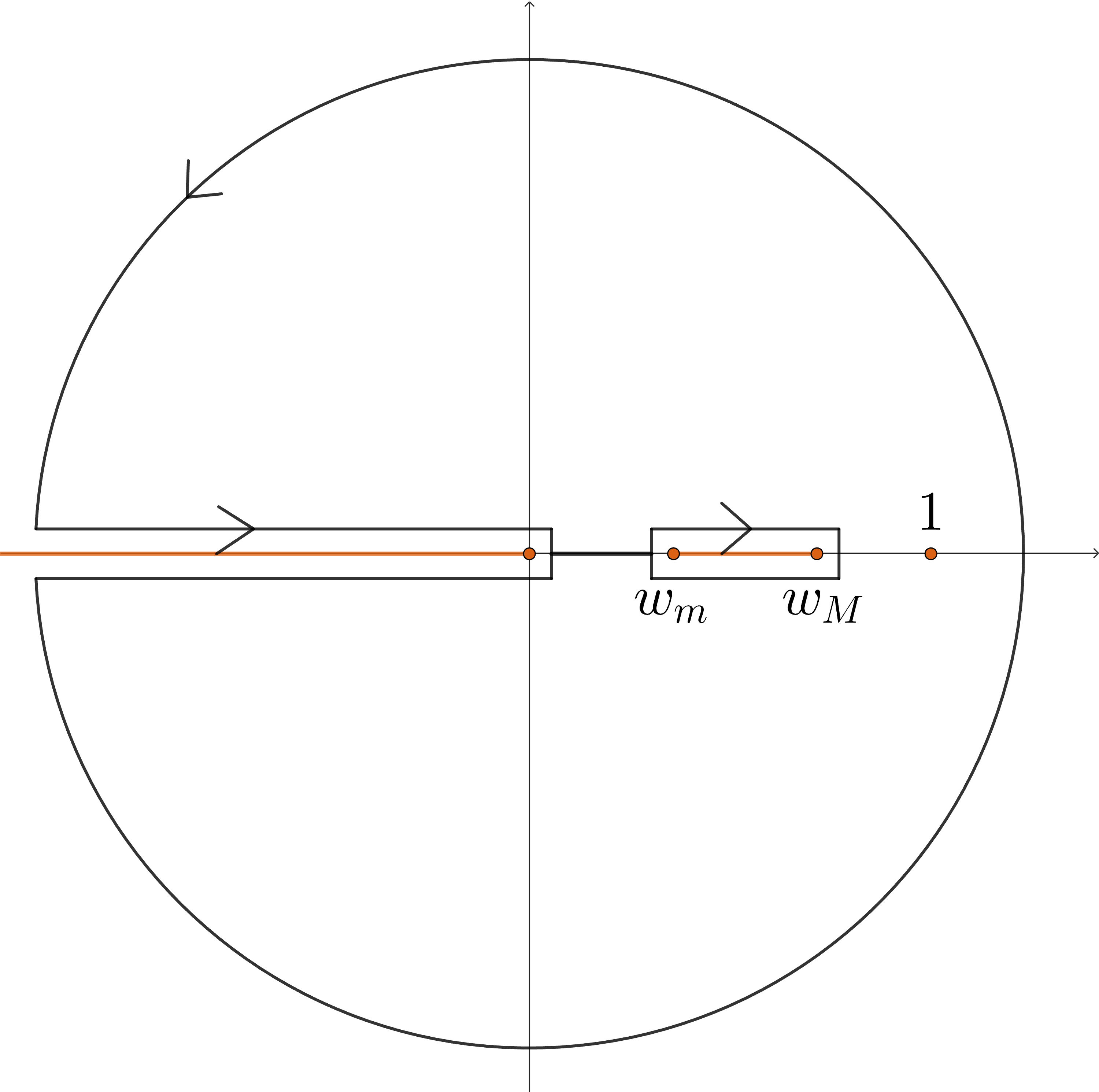}
\caption{Integration path $C_n$ for $f(z)$.} 
\label{fig:recinto}
\end{center}
 \end{figure}
defined on $\C \setminus ((-\infty,0] \cup [w_m,w_M])$ and with a pole at $z = 1$. We will integrate along a sequence of closed paths $C_n$ shown in Figure \ref{fig:recinto}.
Each such path $C_n$ can be divided into five pieces: a circle arc $C_n^{(1)}$ of radius $n$, a curve $C_n^{(2)}$ enclosing the segment $[w_m,w_M]$, another curve $C_n^{(3)}$ around the interval $(-\infty,0]$ with the same boundary points as $C_n^{(1)}$, and a pair of segments $S_n^{(1)}, S_n^{(2)}$ which connect $C_n^{(2)}$, $C_n^{(3)}$. The segments $S_n^{(1)}, S_n^{(2)}$ coincide but have opposite orientations. We take $C_n^{(2)}$ and $C_n^{(3)}$ so that they converge to the intervals $[w_m,w_M]$ and $(-\infty,0]$ respectively as $n$ grows to infinity.

By the residue theorem, and using \eqref{resii}, 
$$\int_{C_n} f(z) dz = 2\pi i \, \text{Res}(f,1) =  \pi.$$
A careful analysis of $f(z)$ shows that
$$\lim_{n \to \infty} \int_{C_n^{(1)}} f(z)dz = 0,$$
$$\lim_{n \to \infty} \int_{C_n^{(2)}} f(z)dz = 2 \int_{w_m}^{w_M} i\frac{\delta - Hw}{2i\sqrt{H^2 + 1}(1 - w)\sqrt{w}\sqrt{w - w_m}\sqrt{w_M - w}}dw = 2 \phi(\mathfrak{s}_2),$$
$$\lim_{n \to \infty}\int_{C_n^{(3)}} f(z)dz =  2 \int_{-\infty}^{0} i\frac{\delta - Hw}{2i^3\sqrt{H^2 + 1}(1 - w)\sqrt{-w}\sqrt{w_m-w}\sqrt{w_M - w}}dw = - M,$$
for some positive constant $M > 0$. On the other hand, since the segments $S_n^{(1)}$ and $S_n^{(2)}$ have opposite orientations, we deduce that 
$$\int_{S_n^{(1)}}f(z)dz = -\int_{S_n^{(2)}}f(z)dz,$$
and so we end up with
$$\pi = 2 \phi(\mathfrak{s}_2)  - M < 2 \phi(\mathfrak{s}_2),$$
as we wanted to prove. From this, the existence of $\tilde {\mathfrak{s}}$ is immediate. It is also possible to deduce that $x_4(\mathfrak{s})$ is decreasing and positive for all $\mathfrak{s} \in [0,\mathfrak{s}_0)$.

\textbf{Case 4:} $\varepsilon = -1$, $H > 0$.

    In this case, we define $\mathfrak{s}_0$ as the first value for which $x(\mathfrak{s}_0) = x_0$, where $x_0 = \sqrt{\delta/H}$. Let us prove that $\mathfrak{s}_0$ exists. First, if $H > 1$, then $x(\mathfrak{s})$ oscillates between $x_m$ and $x_M$, and the fact that $h(x_0) > 0$ implies that $x_0 \in (x_m,x_M)$, and so $\mathfrak{s}_0$ exists indeed. In the case $H \leq 1$, it holds $x_0 \geq x_m$. Since for $\mathfrak{s}\geq 0$ the function $x(\mathfrak{s})$ is increasing and satisfies $\lim_{\mathfrak{s} \to \infty} x(\mathfrak{s}) = \infty$, we see again that $\mathfrak{s}_0$ exists. Now, notice that $\phi(\mathfrak{s})$ is an increasing function on the interval $[0,\mathfrak{s}_0]$, and $\phi'(\mathfrak{s}_0) = 0$. In particular, $\phi(\mathfrak{s}_0) > 0$. A direct computation using \eqref{eq:psihiperbolico} shows that
    $$F(\mathfrak{s}_0) = x'(\mathfrak{s}_0)\frac{\sinh(\phi(\mathfrak{s}_0))}{\sqrt{x(\mathfrak{s}_0)^2 + 1}} > 0,$$ as we wanted to prove. It is also immediate to prove that $\varepsilon x_4(\mathfrak{s}) = - x_4(\mathfrak{s})$ is decreasing for all $\mathfrak{s} \in [0,\mathfrak{s}_0]$.

\textbf{Case 5:} $\varepsilon = -1$, $H = 0$.

In this case, $\phi(\mathfrak{s})$ is an increasing function satisfying $\lim_{\mathfrak{s} \to \infty} \phi(\mathfrak{s})  =: \phi_M < \infty$ (the integral in \eqref{eq:phihiperbolico} is convergent). On the other hand, $x(\mathfrak{s})$ increasing and unbounded, and satisfies \eqref{eq:hx}. We deduce that 
$$\lim_{\mathfrak{s} \to \infty}  F(\mathfrak{s}) = \lim_{\mathfrak{s} \to \infty}\left(\frac{x'\sinh\phi 
}{\sqrt{x^2 + 1}} - \delta\frac{\cosh\phi}{\sqrt{x^2 + 1}}\right) = \sinh(\phi_M) > 0,$$
In particular, it follows that $F(\mathfrak{s})$ vanishes for some $\tilde{\mathfrak{s}} > 0$, as we wanted to prove. It also follows that $- x_4(\mathfrak{s})$ is decreasing for all $\mathfrak{s} \in [0,\tilde{\mathfrak{s}}]$.

In any of the considered cases, we deduce the existence of a first value $\tilde{\mathfrak{s}} > 0$ such that $F(\tilde{\mathfrak{s}}) = 0$.  As commented before, this means that $\mathcal{S}_0$ meets orthogonally the boundary sphere $S[{\bf e}_4,\ep x_4(\tilde{\mathfrak{s}})]$ (see Appendix \ref{ap:geometriaF}). 

Moreover, we also find that $x(\mathfrak{s}),x'(\mathfrak{s}), x_3(\mathfrak{s}) > 0$, $F(\mathfrak{s}) < 0$ for all $\mathfrak{s} \in (0,\tilde{\mathfrak{s}})$. In addition, $x'(\tilde{\mathfrak{s}}) > 0$, which shows by \eqref{eq:kappa} that the principal curvature $\kappa_{\mathfrak{s}}(\mathfrak{s})$ associated to the profile curve must be strictly decreasing on $[0,\tilde{\mathfrak{s}} + \epsilon)$ for some $\epsilon > 0$. Another consequence of these inequalities is that $x_3'(\mathfrak{s}) > 0$ on $(0,\tilde{\mathfrak{s}})$. Since the function $x_3(\mathfrak{s})$ is odd, we obtain that $x_3(\mathfrak{s})$ is injective on $[-\tilde{\mathfrak{s}},\tilde{\mathfrak{s}}]$. As a result, the annulus $\mathcal{S}_0$ must be embedded.

Finally, we will show that $\mathcal{S}_0$ is free boundary in $B$. It suffices to check that $\mathcal{S}_0$ is contained in that ball. This is a consequence of the already proven monotonicity of $\varepsilon x_4(\mathfrak{s})$, since
$$\langle {\bf e}_4, \psi(\mathfrak{s},\theta) \rangle = \varepsilon x_4(\mathfrak{s}) \geq \varepsilon x_4(\tilde{\mathfrak{s}}),$$
 for all $\mathfrak{s} \in [0,\tilde{\mathfrak{s}}]$, which shows that $\mathcal{S}_0 \subset B[{\bf e}_4,\varepsilon x_4(\tilde{\mathfrak{s}})]$.

\subsection{Proof of item 2 of Proposition \ref{pro:hatp}}\label{ap:sanalitico}
We will first prove the analiticity of $\tilde{\mathfrak{s}}$. This fact is a consequence of the implicit function theorem applied on the function $F = F(\mathfrak{s};H,\delta)$ defined in \eqref{eq:F}. We remark that this function not only depends analytically on $\mathfrak{s}$, but also on the parameters $H,\delta$ which define the functions $x$ and $x_3$. If we differentiate $F$ with respect to $\mathfrak{s}$ at $\mathfrak{s} = \tilde{\mathfrak{s}}$, we obtain using $F(\tilde{\mathfrak{s}}) = 0$ that
    \begin{equation}\label{Fp}
        F' = -xx'\left(\frac{x_3''}{x'} - \frac{x_3'x''}{x'^2}\right) = -xx'\left(\frac{x_3'}{x'}\right)'.
    \end{equation}
Since $x(\tilde{\mathfrak{s}}), x'(\tilde{\mathfrak{s}}) > 0$, we just need to check that $\left(\frac{x_3'}{x'}\right)' \neq 0$ at $\tilde{\mathfrak{s}}$ to deduce analyticity. If $\varepsilon = 1$, then by \eqref{eq:hx}, \eqref{eq:phi} and the fact that $F(\tilde{\mathfrak{s}}) = 0$, we deduce after a long but direct computation that
    \begin{equation}\label{x3px}
        \left(\frac{x_3'}{x'}\right)'\Big|_{\mathfrak{s} = \tilde{\mathfrak{s}}} = -\frac{\cos(\phi) \sqrt{1 - x^2} (\delta + H x^2)}{x^2 x'^2} < 0,
    \end{equation}
where the final quantity is negative since $\delta > 0$, $H \geq 0$ and $0< \phi(\tilde{\mathfrak{s}}) < \pi/2$ (see Appendix \ref{ap:tildes}). Similarly, if $\varepsilon = -1$, then by \eqref{eq:hx}, \eqref{eq:phihiperbolico}, we obtain at $\mathfrak{s}=\tilde{\mathfrak{s}}$ that
    \begin{equation}\label{x3pxhiperbolico}
       \left(\frac{x_3'}{x'}\right)'\Big|_{\mathfrak{s} = \tilde{\mathfrak{s}}} = -\frac{\cosh(\phi) \sqrt{x^2  + 1} (\delta + H x^2)}{x^2 x'^2} < 0. 
    \end{equation}
    In any of the cases, we deduce that $\tilde{\mathfrak{s}} = \tilde{\mathfrak{s}}(H,\delta)$ is analytic.

We will now prove that the extension of $\tilde{\mathfrak{s}}(H,\delta)$ along the boundary curve $\delta = \frac{H + \mu}{2}$ given by \eqref{eq:SHd} is continuous. We consider the function $F(\mathfrak{s};H,\delta): \Omega \to \R$ in \eqref{eq:F} defined on the set
$$\Omega:= \left\{(\mathfrak{s},H,\delta)\; : \; \mathfrak{s} \in \R, \; H\geq 0, \; 0 < \delta \leq \frac{H + \mu}{2}\right\}.$$

We know that $F$ is continuous on $\Omega$ since $x$, $x_3$ and their derivatives with respect to $\mathfrak{s}$ depend continuously on the parameters $H, \delta$. We now claim that $F(\mathfrak{s};H,\delta) < 0$ for all $0\leq \mathfrak{s} < \tilde{\mathfrak{s}}$, and that $\mathfrak{s} \mapsto F(\mathfrak{s};H,\delta)$ changes sign at $\mathfrak{s} = \tilde{\mathfrak{s}}$. This is true when $\delta < \frac{H + \mu}{2}$, according to the results in Appendix \ref{ap:tildes} and the fact that $F'(\tilde{\mathfrak{s}}) > 0$; see \eqref{Fp}, \eqref{x3px}. In the remaining case $\delta = \frac{H + \mu}{2}$, it is possible to compute explicitly the functions $x(\mathfrak{s})$, $x_3(\mathfrak{s})$, which are given by \eqref{toro}. This allows us to prove the claims on $F$ in this situation, where we are defining $\tilde{s}\left(H, \frac{H + \mu}{2}\right)$ as in \eqref{eq:SHd}.

We now fix some $H_0 \geq 0$, and denote $\mu_0:=\sqrt{H_0^2+1}$. We need to prove that there exists
\begin{equation}\label{eq:limites}
    \lim_{(H,\delta) \to \left(H_0, \frac{H_0 + \mu_0}{2}\right)} \tilde{\mathfrak{s}}(H,\delta) = 
   \frac{\pi}{2\sqrt{2\mu_0(H_0 + \mu_0)}}=: \tilde{\mathfrak{s}}_0.
\end{equation}

Now, let $\mathfrak{s}_I \leq \mathfrak{s}_S$ denote respectively the limits inferior and superior of the left hand-side of \eqref{eq:limites}. We need to prove that these limits coincide and are equal to $\tilde{\mathfrak{s}}_0$. First, we will prove that $\mathfrak{s}_S \leq \tilde{\mathfrak{s}}_0$. By the continuity of $F$, it follows that $F(\mathfrak{s};H_0,\frac{H_0 + \mu_0}{2}) \leq 0$ for all $\mathfrak{s} \in [0,\mathfrak{s}_S]$. Moreover, since $F(\mathfrak{s};H_0, \frac{H_0 + \mu_0}{2})> 0$ for any $\mathfrak{s} > \tilde{\mathfrak{s}}_0$ sufficiently close to $\tilde{\mathfrak{s}}_0$, we deduce that $\mathfrak{s}_S \leq \tilde{\mathfrak{s}}_0$. 

Now, we will prove that $\tilde{\mathfrak{s}}_0\leq \mathfrak{s}_I$. Notice first that $\mathfrak{s}_I \geq 0$ since $\tilde{\mathfrak{s}}(H,\delta)$ is positive for all $\delta < \frac{H + \mu}{2}$. Moreover, the fact that $F(\tilde{\mathfrak{s}}(H,\delta);H, \delta) \equiv  0$ implies by continuity that $F(\mathfrak{s}_I;H_0, \frac{H_0 + \mu_0}{2}) = 0$. However, $\tilde{\mathfrak{s}}_0$ is the first positive root of $\mathfrak{s} \mapsto F(\mathfrak{s};H_0,\frac{H_0 + \mu_0}{2})$, so necessarily $\tilde{\mathfrak{s}}_0 \leq \mathfrak{s}_I$. As a consequence, $\mathfrak{s}_I = \mathfrak{s}_S = \tilde{\mathfrak{s}}_0$, and \eqref{eq:limites} holds.

\subsection{Proof of item 3 in Proposition \ref{pro:hatp}}
\label{ap:geometriaF}

We will deal with the spherical and hyperbolic cases separately. Assume first that $\varepsilon = 1$ and let $\S^2 := \{x_2 = 0 \} \cap \S^3$. Consider the totally geodesic projection $P:\S^2 \cap \{x_4 > 0\} \to \R^2$,
\begin{equation}\label{projection}
(\overline x, \overline y) = P(x,0,x_3,x_4) = \left(\frac{x}{x_4},\frac{x_3}{x_4}\right).
\end{equation}

    The image by $P$ of the rotation axis $\mathcal{L}$ is the line $\{\overline x = 0\}$. Recall next the notation in \eqref{pro:s} and the definition of $\Gamma_{\mathfrak{s}_0} \subset \M^2(\varepsilon)$ as the geodesic passing through $\psi(\mathfrak{s}_0,0)$ with tangent vector $\psi_\mathfrak{s}(\mathfrak{s}_0,0)$. For any $\mathfrak{s}_0$ such that $x_4(\mathfrak{s}_0) > 0$, the image of $\Gamma_{\mathfrak{s}_0}$ by $P$ is exactly the tangent line of the curve $\mathfrak{s} \mapsto P(\psi(\mathfrak{s},0))$ at $\mathfrak{s} = \mathfrak{s}_0$, since $P$ is totally geodesic. This tangent line intersects the axis $\{\overline x = 0\}$ exactly at the point
    \begin{equation}\label{overliney}
        \overline y_0(\mathfrak{s}_0) = \frac{x_3(\mathfrak{s}_0)}{x_4(\mathfrak{s}_0)} - \frac{x(\mathfrak{s}_0)}{x_4(\mathfrak{s}_0)}\frac{\left(\frac{x_3}{x_4}\right)'|_{\mathfrak{s}_0}}{\left(\frac{x}{x_4}\right)'|_{\mathfrak{s}_0}}.
    \end{equation}
    A direct computation shows that, at $\mathfrak{s}_0$,
    \begin{equation}\label{overlineyF}
        \overline y_0(x'x_4 - x x_4') = x_3 x' - x x_3' = F(\mathfrak{s}_0).
    \end{equation}
    Observe that $x_4(\tilde{\mathfrak{s}}) > 0$, so $\psi(\tilde{\mathfrak{s}},0)$ lies in the definition domain of $P$. We will prove later on that the inequality
    \begin{equation}\label{ineqx4}
        x'x_4 - xx_4' > 0
    \end{equation}
    holds at $\mathfrak{s}_0=\tilde{\mathfrak{s}}$, but let us assume it for now. An immediate consequence of this is that $\overline{y}_0(\tilde{\mathfrak{s}}) = 0$. Geometrically, this implies that the geodesic $\Gamma_{\tilde{\mathfrak{s}}}$ passes through $P^{-1}(0,0) = (0,0,0,1) ={\bf e}_4$. This proves a claim made in Appendix \ref{ap:tildes} for the spherical case. 
    
    Since \eqref{ineqx4} holds at $\tilde{\mathfrak{s}}$, there is a neighbourhood $\mathcal{I}$ of this point where the inequality is also satisfied. For all $\mathfrak{s}_0 \in \mathcal{I}$, we can express the function $\hat p(\mathfrak{s}_0)$ as
    $$\hat p(\mathfrak{s}_0) = P^{-1}(0,\overline{y}_0(\mathfrak{s}_0)).$$
    This shows that $\hat p$ is analytic, as $\overline{y}_0$ depends analytically on $\mathfrak{s}_0$. Moreover, a direct computation using \eqref{overlineyF} and the fact that $F'(\tilde{\mathfrak{s}}) > 0$ (see \eqref{Fp}, \eqref{x3px}) shows that $\hat p_3'(\tilde{\mathfrak{s}}) > 0$. This completes the proof of item 3 of Proposition \ref{pro:hatp} when $\varepsilon = 1$.

    We will now deal with the hyperbolic case. Let $\varepsilon = -1$, $\H^2 := \{x_2 = 0 \} \cap \H^3$ and consider the totally geodesic projection $P:\H^2 \to \D$ given by \eqref{projection}, where $\D$ is the unit disk of $\R^2$. Just as before, the image of $\Gamma_{\mathfrak{s}_0}$ by $P$ is the tangent line of the curve $\mathfrak{s} \mapsto P(\psi(\mathfrak{s},0))$ at $\mathfrak{s} = \mathfrak{s}_0$, which meets the axis $\{\overline x = 0\}$ at a point $\overline y_0$ given by \eqref{overliney}.  If we assume again that \eqref{ineqx4} holds at $\tilde{\mathfrak{s}}$, then we obtain that the geodesic $\Gamma_{\tilde{\mathfrak{s}}}$ passes through $P^{-1}(0,0) = {\bf e}_4$. In particular, this proves the claim made in Appendix \ref{ap:tildes} for the hyperbolic case. From this point, the proof of item 3 of Proposition \ref{pro:hatp} for the hyperbolic case is entirely analogous to spherical one, so we omit it.

    It just remains to prove that \eqref{ineqx4} holds at $\tilde{\mathfrak{s}}$. Let $\chi := \frac{x'(\tilde{\mathfrak{s}})}{x(\tilde{\mathfrak{s}})} - \frac{x_4'(\tilde{\mathfrak{s}})}{x_4(\tilde{\mathfrak{s}})}$. Differentiating
    $$\langle\psi(\mathfrak{s},0),\psi(\mathfrak{s},0)\rangle = x^2 + x_3^2 + \varepsilon x_4^2 = \varepsilon$$ 
    with respect to $\mathfrak{s}$, we obtain at $\tilde{\mathfrak{s}}$ that
    $$0 = x x' + x_3 x_3' + \varepsilon x_4 x_4' = \frac{x'}{x}\left(x^2 + x_3^2 + \varepsilon x_4^2\right) - \varepsilon\chi x_4^2  = \varepsilon \left(\frac{x'}{x} - \chi x_4^2\right),$$
    where we have used the fact that $F(\tilde{\mathfrak{s}}) = x_3x' - x_3'x =  0$. Since $x$ is a positive function and $x'(\tilde{\mathfrak{s}}) > 0$ by Proposition \ref{pro:hatp}, we deduce that $\chi > 0$, and consequently \eqref{ineqx4} holds at $\tilde{\mathfrak{s}}$.

\bibliographystyle{amsalpha}

\vskip 0.2cm

\noindent Alberto Cerezo

\noindent Departamento de Geometría y Topología \\ Universidad de Granada (Spain) \\ Departamento de Matemática Aplicada I \\ Universidad de Sevilla (Spain)

\noindent  e-mail: {\tt cerezocid@ugr.es}

\vskip 0.2cm

\noindent Isabel Fernández

\noindent Departamento de Matemática Aplicada I,\\ Instituto de Matemáticas IMUS \\ Universidad de Sevilla (Spain).

\noindent  e-mail: {\tt isafer@us.es}

\vskip 0.2cm

\noindent Pablo Mira

\noindent Departamento de Matemática Aplicada y Estadística,\\ Universidad Politécnica de Cartagena (Spain).

\noindent  e-mail: {\tt pablo.mira@upct.es}

\vskip 0.4cm

\noindent This research has been financially supported by Project PID2020-118137GB-I00 funded by MCIN/AEI /10.13039/501100011033, and CARM, Programa Regional de Fomento de la Investigación, Fundación Séneca-Agencia de Ciencia y Tecnología Región de Murcia, reference 21937/PI/22

\end{document}